\documentclass[aop]{imsart}

\usepackage{tikz}
\usetikzlibrary{calc}
\usepackage{ucs}
\usepackage[utf8x]{inputenc}
\usepackage[T1]{fontenc}

\usepackage{amsmath}
\usepackage{amsfonts}
\usepackage{amssymb}
\usepackage{amsthm}
\usepackage{mathtools}
\mathtoolsset{showonlyrefs}	
\usepackage{xcolor}
\usepackage{enumitem}
\usepackage[margin=1cm,labelfont=sc,font={small}]{caption}
\usepackage{graphicx}
\DeclareGraphicsExtensions{.png,.pdf,.jpg}
\graphicspath{{Pics/}}

\usepackage{wrapfig}
\RequirePackage[numbers]{natbib}

\startlocaldefs

\theoremstyle{plain}
\newtheorem{theorem}{Theorem}[section]
\newtheorem{lemma}[theorem]{Lemma}
\newtheorem{proposition}[theorem]{Proposition}

\newtheorem{definition}[theorem]{Definition}

\newtheorem{remark}[theorem]{Remark}

\theoremstyle{definition}

\renewcommand{\theequation}{\arabic{section}.\arabic{equation}}
\renewcommand{\thesection}{\arabic{section}}
\renewcommand{\thesubsection}{\arabic{section}.\arabic{subsection}}
\renewcommand{\thesubsubsection}{\arabic{section}.\arabic{subsection}.\arabic{subsubsection}}

\newcommand{\maxtwo}[2]{\max_{\substack{#1 \\ #2}}} 
\newcommand{\abs}[1]{\left\lvert #1\right\rvert}
\newcommand{\nabs}[1]{\lvert #1\rvert}
\newcommand{\babs}[1]{\bigl\lvert #1\bigr\rvert}
\newcommand{\Babs}[1]{\Bigl\lvert #1\Bigr\rvert}

\newcommand{\calB}{\mathcal{B}}
\newcommand{\calC}{\mathcal{C}}
\newcommand{\calD}{\mathcal{D}}
\newcommand{\calE}{\mathcal{E}}

\newcommand{\calG}{\mathcal{G}}
\newcommand{\calH}{\mathcal{H}}

\newcommand{\calL}{\mathcal{L}}

\newcommand{\calP}{\mathcal{P}}

\newcommand{\calS}{\mathcal{S}}
\newcommand{\calT}{\mathcal{T}}

\newcommand{\calZ}{\mathcal{Z}}

\newcommand{\frc}{\mathfrak{c}}

\newcommand{\fri}{\mathfrak{i}}
\newcommand{\frj}{\mathfrak{j}}

\newcommand{\frn}{\mathfrak{n}}

\newcommand{\bbH}{\mathbb{H}}

\newcommand{\bbL}{\mathbb{L}}

\newcommand{\bbN}{\mathbb{N}}

\newcommand{\bbR}{\mathbb{R}}
\newcommand{\bbS}{\mathbb{S}}

\newcommand{\bbV}{\mathbb{V}}

\newcommand{\bbZ}{\mathbb{Z}}

\newcommand{\sfa}{\mathsf a}

\newcommand{\sfc}{\mathsf c}

\newcommand{\sfe}{\mathsf e}

\newcommand{\sfh}{\mathsf h}

\newcommand{\sfo}{{\mathsf o}}
\newcommand{\sfp}{{\mathsf p}}

\newcommand{\sfs}{{\mathsf s}}

\newcommand{\sfu}{{\mathsf u}}
\newcommand{\sfv}{{\mathsf v}}
\newcommand{\sfw}{{\mathsf w}}

\newcommand{\sfz}{{\mathsf z}}

\newcommand{\sfA}{\mathsf{A}}
\newcommand{\sfB}{\mathsf{B}}
\newcommand{\sfC}{\mathsf{C}}

\newcommand{\sfE}{\mathsf{E}}

\newcommand{\sfG}{\mathsf{G}}

\newcommand{\sfK}{\mathsf{K}}
\newcommand{\sfL}{\mathsf{L}}

\newcommand{\sfO}{\mathsf{O}}
\newcommand{\sfP}{\mathsf{P}}

\newcommand{\sfS}{\mathsf{S}}
\newcommand{\sfT}{\mathsf{T}}

\newcommand{\sfW}{\mathsf{W}}
\newcommand{\sfX}{\mathsf{X}}

\newcommand{\sfZ}{\mathsf{Z}}

\newcommand{\uomega}{\underline{\omega}}
\newcommand{\uOmega}{\underline{\Omega}}

\newcommand{\setof}[2]{\left\{#1 \,:\, #2 \right\}}

\newcommand{\bsetof}[2]{\bigl\{#1\,:\,#2\bigr\}}
\newcommand{\Bsetof}[2]{\Bigl\{#1\,:\,#2\Bigr\}}
\newcommand{\given}{\,|\,}
\newcommand{\bgiven}{\bigm|}
\newcommand{\Bgiven}{\Bigm|}
\newcommand{\defby}{\stackrel{\text{\tiny{\rm def}}}{=}}
\newcommand{\comp}{{\mathrm{c}}}

\newcommand{\IF}[1]{\1_{\{#1\}}}
\newcommand{\bIF}[1]{\1_{\bigl\{#1\bigr\}}}

\newcommand{\lb}{\left(}
\newcommand{\rb}{\right)}
\newcommand{\lbr}{\left\{}
\newcommand{\rbr}{\right\}}

\newcommand{\dd}{{\rm d}}

\newcommand{\step}[1]{S{\small TEP}\,#1.}

\newcommand{\Var}{\bbV{\rm ar}}

\newcommand{\1}{{\text{\usefont{U}{dsrom}{m}{n}1}}}

\newcommand{\smo}[1]{{\mathrm o}(#1)}
\newcommand{\bgo}[1]{{\mathrm O}(#1)}

\newcommand{\leqs}{\lesssim}            

\newcommand{\be}[1]{\begin{equation}\label{#1}}
	\newcommand{\ee}{\end{equation}}

\newcommand{\transferOp}{\mathbf{T}}
\newcommand{\scalingop}{\mathrm{Sc}_N}

\newcommand{\normI}[1]{\|#1\|_{\scriptscriptstyle 1}}
\newcommand{\normII}[1]{\|#1\|_{\scriptscriptstyle 2}}
\newcommand{\normsup}[1]{\|#1\|_{\scriptscriptstyle\infty}}

\newcommand{\uvec}{\sfe}

\newcommand{\Ham}{\calH}		
\newcommand{\PF}{\mathbf{Z}}	

\newcommand{\BOX}{\sfB}	
\newcommand{\BOXs}{\sfB^*}	
\newcommand{\PHS}{\bbH_+}		
\newcommand{\NHS}{\bbH_-}		
\renewcommand{\emptyset}{\varnothing}
\newcommand{\betac}{\beta_{\mathrm{c}}}

\newcommand{\st}{\tau_\beta} 	
\newcommand{\lc}[1]{x^{#1}_{\mathrm{\ell}}} 	
\DeclareRobustCommand{\rc}[1]{x^{#1}_{\mathrm{\vphantom{\ell}r}}} 	

\DeclareMathOperator{\inte}{int}
\DeclareMathOperator{\ext}{ext}
\DeclareMathOperator{\diam}{diam}

\newcommand{\taub}{\tau_\beta}

\newcommand{\fcone}{\mathcal{Y}^\blacktriangleleft}
\newcommand{\bcone}{\mathcal{Y}^\blacktriangleright}
\newcommand{\bend}{\mathsf{b}}
\newcommand{\fend}{\mathsf{f}}

\newcommand{\CPts}{\textnormal{CPts}}
\newcommand{\INts}{\textnormal{INts}}

\newcommand{\SetRootMarkBackCont}{\mathfrak{B}_L}
\newcommand{\SetRootMarkForwCont}{\mathfrak{B}_R}
\newcommand{\SetRootDiaCont}{\mathfrak{A}}

\newcommand{\BOXN}{N}

\newcommand{\PNbeta}{\sfP_{N , \beta }}
\newcommand{\Pbeta}[1]{\sfP_{\beta}^{#1}}
\newcommand{\Pbetap}[1]{\sfP_{\beta , +}^{#1}}
\newcommand{\PNbetal}{\sfP_{N , \beta , \lambda/N }}

\newcommand{\Pbetapl}[1]{\sfP_{\beta , \lambda/N , +}^{#1}}
\newcommand{\hPbetapl}[1]{\widehat\sfP_{\beta , \lambda/N , +}^{#1}}

\newcommand{\Ebetap}[1]{\sfE_{\beta , +}^{#1}}

\newcommand{\WNbeta}{\sfW_{N , \beta }}
\newcommand{\Wbeta}{\sfW_{ \beta }}

\newcommand{\etap}{\eta^\prime}
\newcommand{\gammap}{\gamma^\prime}
\DeclareMathOperator{\gap}{\mathsf{gap}}

\newcommand{\eig}{\sfa}

\newcommand\NoMarkerFootnote[1]{%
	\begingroup
	\renewcommand\thefootnote{}\footnote{#1}%
	\addtocounter{footnote}{-1}%
	\endgroup
}

\endlocaldefs

\begin{document}
	
\begin{frontmatter}

\title{Critical prewetting in the 2d Ising model}
\runtitle{Critical prewetting in the 2d Ising model}

\begin{aug}
\author[A]{\fnms{Dmitry} \snm{Ioffe}\ead[label=e1]{ioffe.dmitry@gmail.com}},
\author[B]{\fnms{Sébastien} \snm{Ott}\ead[label=e2]{ott.sebast@gmail.com}},
\author[C]{\fnms{Senya} \snm{Shlosman}\ead[label=e3]{s.shlosman@skoltech.ru}}
\and
\author[D]{\fnms{Yvan} \snm{Velenik}\ead[label=e4]{yvan.velenik@unige.ch}}

\address[A]{Faculty of IE\&M,
	Technion, Haifa 32000, Israel
	\printead{e1}}
\address[B]{Dipartimento di Matematica e Fisica,
	Università degli Studi Roma Tre, 00146 Roma, Italy,
	\printead{e2}}
\address[C]{Center for Advanced Studies,
	Skoltech, Moscow 143026, Russia {\&} Aix Marseille Univ, Universite de Toulon,
	CNRS, CPT, Marseille, France,
	\printead{e3}}
\address[D]{Section de Mathématiques,
	Université de Genève, CH-1205 Genève, Switzerland,
	\printead{e4}}
\end{aug}

\begin{abstract}
	In this paper we develop a detailed analysis of critical prewetting in the context of the two-dimensional Ising model.
	Namely, we consider a two-dimensional nearest-neighbor Ising model in a \(2N\times N\) rectangular box with a boundary condition inducing the coexistence of the \(+\) phase in the bulk and a layer of \(-\) phase along the bottom wall.
	The presence of an external magnetic field of intensity \(h=\lambda/N\) (for some fixed \(\lambda>0\)) makes the layer of \(-\) phase unstable.
	For any \(\beta>\betac\), we prove that, under a diffusing scaling by \(N^{-2/3}\) horizontally and \(N^{-1/3}\) vertically, the interface separating the layer of unstable phase from the bulk phase weakly converges to an explicit Ferrari-Spohn diffusion.
\end{abstract}



\end{frontmatter}

\NoMarkerFootnote{This paper was close to being completed, when the first author, Dmitry Ioffe, passed away. He took part in working on the manuscript to his last days.  The reader who is familiar with his style will recognize his technical powers and his clarity of ideas. We have been trying to keep Dima's style while preparing the final version of the manuscript.}

\section{The model and the result}

\subsection{Introduction}
\label{sub:intro}

The study of interfaces separating different equilibrium phases has a very long history that goes back at least to Gibbs' famous monograph~\cite{Gibbs-1878}.
Of particular interest are the \emph{surface phase transitions} that such systems can undergo, which manifest in singular behavior of the interfaces as some parameter is changed.

Probably the most investigated surface phase transition is the \emph{wetting transition}.
The latter can occur in systems with three different coexisting phases (or with two different phases interacting with a substrate).
Wetting occurs when a \emph{mesoscopic} layer (that is, a layer the width of which diverges in the thermodynamic limit) of, say, phase \(C\) appears at the interface between phases \(A\) and \(B\).
Various aspects of the wetting phase transition have been discussed rigorously: wetting of a substrate in the Ising model~\cite{Pfister+Velenik-96,Frohlich+Pfister-87,Pfister+Velenik-97,Bodineau+Ioffe+Velenik-01}, wetting of the interface between two ordered phases by the disordered phase in the Potts model (when the transition is of first order) or wetting of the \(+/-\) interface by the \(0\) phase in the Blume-Capel model at the triple point~\cite{Bricmont+Lebowitz-87,deConinck+Messager+MiracleSole+Ruiz-88,Hryniv+Kotecky-02,Messager+MiracleSole+Ruiz+Shlosman-91}, etc.

In this phenomenon, the mesoscopic layer wetting the interface (or the substrate) is made of an equilibrium phase.
It is interesting to understand how such a layer behaves when the parameters of the model are slightly changed in order that the corresponding ``phase'' becomes thermodynamically unstable (but the two phases separated by the interface remain equilibrium phases).
In that case, when the free energy of the unstable phase is close enough to that of the stable phases, an interface separating the latter phases may remain wet by a \emph{microscopic} layer of the unstable phase (that is, a layer the width of which remains bounded in the thermodynamic limit).
See, for instance, the numerical analysis in~\cite{Carlon+Igloi+Selke+Szalma-99} of the layer of unstable disordered phase in the two-dimensional Potts model when the temperature is taken slightly below the temperature of order-disorder coexistence.
Here, one usually speaks of \emph{interfacial adsorption} (that is, the appearance of an excess of the unstable phase along the interface, as compared to its density in the bulk of the system), or of \emph{prewetting}.
When the width of the layer of unstable phase continuously diverges as the system is brought closer and closer to the coexistence point, the system is said to undergo \emph{critical prewetting}.

\medskip
In this paper, we provide a detailed analysis of critical prewetting in the two-dimensional Ising model.
Let us briefly describe the setting in an informal way, precise definitions being provided in the next section.
We consider the Ising model in a square box of sidelength \(n\) in \(\mathbb{Z}^2\).
We consider \(+\) boundary condition on three sides of the box and \(-\) boundary condition on the fourth (say, the bottom one).
At phase coexistence (that is, in the absence of a magnetic field and when the temperature is below critical), the bulk of the system is occupied by the \(+\) phase, with a layer of \(-\) phase along the bottom wall.
The width of this layer is of order \(\sqrt{n}\) (thus mesoscopic).
The statistical properties of the interface in this regime are now well understood: after a diffusive rescaling, the latter has been proved to converge weakly to a Brownian excursion~\cite{Ioffe+Ott+Velenik+Wachtel-20}.

\begin{wrapfigure}{R}{4.8cm}
	\centering
	\frame{\includegraphics[width=4.5cm]{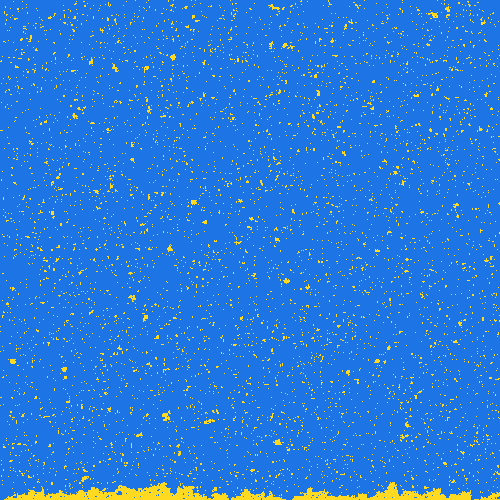}}
\end{wrapfigure}
In order to analyze critical prewetting in this model, we make the \(-\) phase unstable by turning on a positive magnetic field \(h\).
The width of the layer now becomes microscopic, with a width of order \(h^{-1/3}\)~\cite{Velenik-04}.
Our goal in this paper is to provide a detailed analysis of the behavior of this layer in a suitable scaling limit.
Namely, we shall let the intensity \(h\) of the magnetic field vanish at the same time as we let the size \(n\) of the box diverge.
It turns out that a natural way to do that is to set \(h=\lambda/n\), for some fixed constant \(\lambda>0\).
(One reason this particular choice is natural is given in the next paragraph.)
Our main result is a proof that the interface, once rescaled by \(h^{1/3}\) vertically and \(h^{2/3}\) horizontally, weakly converges to an explicit, nondegenerate diffusion process (a Ferrari-Spohn diffusion with suitable parameters, see the next section).

\medskip
One way to understand why the particular choice \(h=\lambda/n\) is natural is to consider a slightly different geometry.
Namely, consider a two-dimensional Ising model in a square box of sidelength \(n\), with \(-\) boundary on all four sides.
Of course, in the absence of a magnetic field and below the critical temperature, the \(-\) phase fills the box. When a positive magnetic field \(h\) is turned on, the \(-\) phase becomes unstable.
It follows that, when the box is taken large enough, the bulk of the box must be occupied by the \(+\) phase (with a possible layer of unstable \(-\) phase along the boundary).
However, for small enough boxes (or a weak enough magnetic field \(h\)), the effect of the boundary condition should dominate and the unstable \(-\) phase should still fill the box.
Intuitively, the critical scale can be obtained as follows: the energetic contribution due to the boundary condition is of order \(n\), while the effect of the magnetic field is of order \(hn^2\).
These two effects will thus compete when \(h\) is of order \(1/n\).
This argument was made precise in~\cite{Schonmann+Shlosman-96}.
Namely, setting \(h=\lambda/n\) with \(\lambda>0\), it was proved that there exists \(\lambda_{\rm c}>0\) such that, when \(n\) is taken large enough, the unstable \(-\) phase occupies the box when \(\lambda<\lambda_{\rm c}\), while the \(+\) phase occupies the box when \(\lambda > \lambda_{\rm c}\). A detailed description of the macroscopic geometry of typical configurations in large boxes was actually provided: when \(\lambda > \lambda_{\rm c}\), typical configurations contain a macroscopic droplet of \(+\) phase squeezed in the box; the droplet's boundary is made up of four straight line segments, one along each of the four sides of the box (at least as seen at the macroscopic scale), with four arcs joining them near the corners of the box (these arcs being suitably scaled quarters of the corresponding Wulff shape).
The layers corresponding to the four line segments should have precisely the same behavior as the layer in the geometry we analyze in the current paper.
\begin{figure}[h]
	\centering
	\frame{\includegraphics[width=4.5cm]{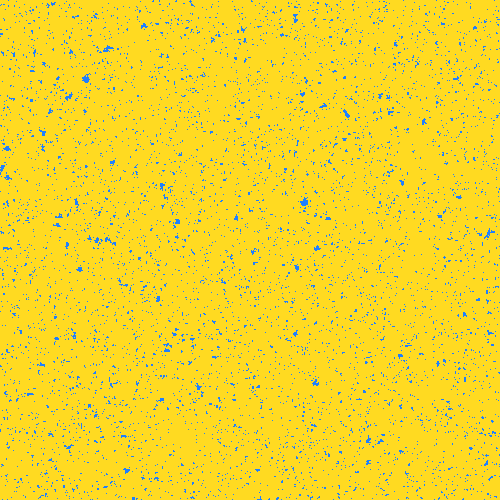}}\hspace{1.5cm}\frame{\includegraphics[width=4.5cm]{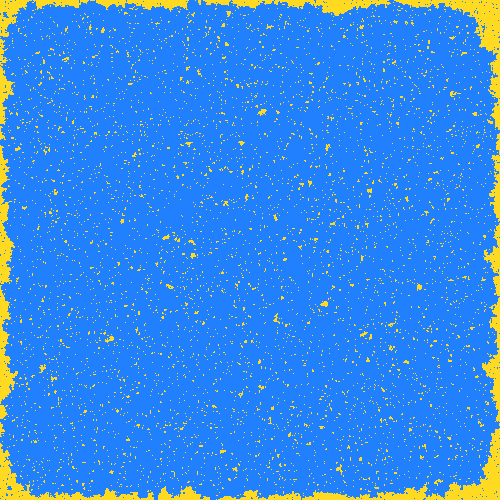}}
	\caption{Typical configurations in a box of sidelength \(n\) with \(-\) boundary condition and magnetic field \(h=\lambda/n>0\). Left: \(\lambda<\lambda_{\rm c}\); right: \(\lambda>\lambda_{\rm c}\).}
	\label{fig:metastability}
\end{figure}

\medskip
We now briefly discuss some earlier rigorous results pertaining to the problem investigated here.
Most of those were dealing with the simpler case of effective models, in which the Ising model interface is replaced by a suitable one-dimensional random walk in an external potential mimicking the effect of the magnetic field.
The first work along these lines is~\cite{Abraham+Smith-86}, in which a particular choice of the random walk transition probabilities made the resulting system integrable and thus allowed for the explicit computation of various quantities of interest.
Similar, but somewhat less precise, results were then obtained for a very general class of random walks (and external potentials) in~\cite{Hryniv-Velenik_2004}.
In essentially the same setting, the scaling limit of the random walk trajectory was proven to be given by a suitable Ferrari--Spohn diffusion in~\cite{Ioffe+Shlosman+Velenik-15}.
The first results for the Ising model itself were obtained in~\cite{Velenik-04}, where the order of the width of the layer of unstable phase was determined in a weak sense and up to logarithmic error terms.
	These results were very recently considerably improved upon in~\cite{GG20}.
	In the latter work, the precise identification of the width of the layer was finally obtained, together with detailed information about various local and global properties of the interface (area, maximum, etc).
	
	It should be noted here that, although being both based on some ideas introduced in~\cite{Velenik-04}, the current paper and~\cite{GG20} are very different, both in their goals and from a technical point of view. In particular, the present work does not rely on results from~\cite{GG20}.
	
	Roughly speaking, from a technical point of view, \cite{GG20} improves on the arguments in~\cite{Velenik-04} and combines them with clever local resampling and coupling ideas.
	In contrast, our approach is based on a coupling between the interface and a directed random walk under area-tilt, relying on a suitable version of the Ornstein--Zernike theory, as developed in~\cite{Campanino-Ioffe-Velenik_2003, Campanino-Ioffe-Velenik_2008, Ott-Velenik_2017}.
	We can then analyze the tilted random walk using ideas from our previous works~\cite{Ioffe+Shlosman+Velenik-15,Ioffe-Velenik-Wachtel-18}.
	
	In the present paper we prove the invariance principle to the Ferrari--Spohn diffusion for the interface described above. The authors of~\cite{GG20} consider their paper as a step in this direction, but their actual results are of a rather different nature and it is not clear to us how the technology introduced in~\cite{GG20} might be adapted to get the invariance principle. On the other hand, one might address some of the remaining open problems in~\cite{GG20} using our coupling between the interface and a direct random walk with area-tilt, though we have not done it, this paper being long enough as it is.
	The only exception is~Lemma~\ref{lem:UB-Area-g}, in which we provide one of the estimate missing in~\cite{GG20}, as it is useful for our undertaking. We should stress, however, that in its current form our coupling is not particularly well suited to prove the type of \emph{global} estimates that~\cite{GG20} are interested in.
	It should be possible, by improving parts of the current work, to lift these restrictions and to get asymptotically sharp versions of the bounds obtained in~\cite{GG20}.
	
\subsection{{The model}}
Let \(\Lambda\Subset\bbZ^2\), \(\beta\geq 0\) and \(h\in\bbR\).
For any configuration \(\sigma\in\Omega\defby\{-1,1\}^{\bbZ^2}\), we define
\[
\Ham_{\Lambda;\beta,h}(\sigma) \defby -\beta \sum_{\substack{\{i,j\}
\cap\Lambda\neq\emptyset\\\normI{j-i}=1}} {(\sigma_i\sigma_j-1)} - h \sum_{i\in\Lambda} \sigma_i .
\]
In the sequel, we will always assume that \(\beta>\betac\), the inverse critical temperature of the 2d nearest-neighbor Ising model.
As we will work with \(\beta\) fixed, we shall casually remove it from the notation.
It should be kept in mind, however, that various numerical constants we employ below are allowed to depend on it.

Given \(\eta\in\Omega\) and \(\Lambda\Subset\bbZ^2\), set
\[
\Omega_\Lambda^\eta\defby \setof{\sigma\in\Omega}{\sigma_i=\eta_i\;
\forall i\not\in\Lambda} .
\]
The Gibbs measure in \(\Lambda\) with boundary condition \(\eta\) is defined as
\[
\mu_{\Lambda;\beta,h}^\eta (\sigma) \defby
\frac{\IF{\sigma\in\Omega_\Lambda^\eta}}{\PF_{\Lambda;\beta,h}^\eta}
\exp\bigl(-\Ham_{\Lambda;\beta,h}(\sigma) \bigr) ,
\]
where
\[
\PF_{\Lambda;\beta,h}^\eta \defby
\sum_{\sigma\in\Omega_\Lambda^\eta} \exp\big(-\Ham_{\Lambda ; \beta,h}(\sigma)\bigr) .
\]

Let us introduce
\begin{align*}
&\PHS\defby\setof{i=(i_1,i_2)\in\bbZ^2}{i_2\geq 0},\
\NHS \defby \bbZ^2\setminus\PHS\ \\
&\qquad {\rm  and}\\
 &\calL \defby\setof{x=(x_1,x_2)\in\bbR^2}{x_2=-\tfrac12}
\end{align*}
Given a subset \(A\subset \bbZ^2\), we write
\begin{equation}\label{eq:Ahat}
	\hat{A} \defby \cup_{i\in A} \lb i+ [-1/2 ,1/2]^2\rb \subset \bbR^2
\end{equation}
and
\[
	\partial^{\rm ext} A \defby \setof{i\in\bbZ^2\setminus A}{\exists j\in A,\, \normI{i-j}=1}.
\]

We will be mostly interested in the behavior of the model in the box
\[
\BOX_N \defby \{-N,\ldots,N\}\times\{0,\ldots,N\}
\]
and with a magnetic field of the form \(h=\lambda/N\) for some \(\lambda>0\).
Moreover, the following choices for the boundary condition will be important in the sequel: the constant boundary conditions \(\eta^+\equiv 1\) and \(\eta^-\equiv -1\) and the mixed boundary condition \(\eta^\pm\) defined by
\[
\eta^\pm_i \defby \IF{i\in\PHS} - \IF{i\in\NHS} .
\]
 The corresponding measures and partition functions will be denoted \(\mu_{\BOXN;\beta,\lambda/N}^+\),\linebreak \(\PF_{\BOXN;\beta,\lambda/N}^+\), \(\mu_{\BOXN;\beta,\lambda/N}^-\),
 \(\PF_{\BOXN;\beta,\lambda/N}^-\), \(\mu_{\BOXN;\beta,\lambda/N}^\pm\) and \(\PF_{\BOXN;\beta,\lambda/N}^\pm\) respectively.

In the sequel, we use \(\bbZ^{2,*} \defby (\frac12 , \frac12) +\bbZ^2\) to denote the dual lattice and
\[
	\BOXs_N \defby \Bsetof{(i,j)\in\bbZ^{2,*}}{\abs{i}\leq N+\tfrac12,\, -\tfrac12\leq j\leq N+\tfrac12}
\]
to denote the box dual to \(\BOX_N\) (see Fig.~\ref{fig:PrimalDualBoxes}).
\begin{figure}[t]
	\centering
	\includegraphics{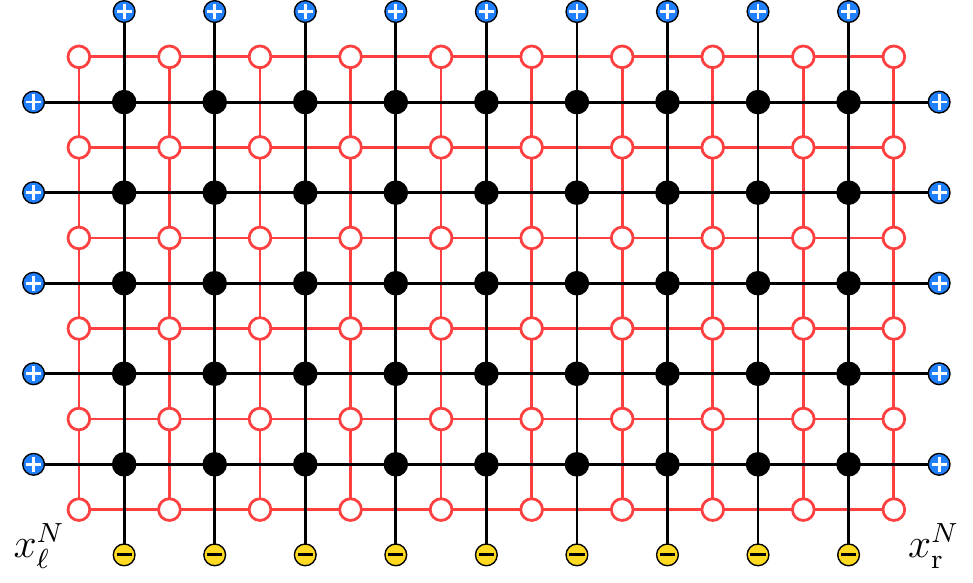}
	\caption{
			The primal box \(\BOX_N\) (in black) and its dual \(\BOXs_N\) (in red) for \(N=4\). The two bottom corner vertices of the dual box
			\(\lc{N}\) and \(\rc{N}\) are also indicated, as well as the boundary condition \(\eta^\pm\) acting on \(\BOX_N\).
	}
	\label{fig:PrimalDualBoxes}
\end{figure}

\medskip
Consider a configuration \(\sigma\in\Omega_{\BOX_N}^\pm \defby\Omega_{\BOX_N}^{\eta_\pm}\).
One can see the boundary of the set \(\bigcup_{i\in\bbZ^2:\,\sigma_i=-1} i + [-\frac12,\frac12]^2\) as a collection of edges of the dual lattice.
Two dual edges in this collection are said to be connected if they still belong to the same component after applying the following deformation rules:\label{deformationRules}\\[3mm]
\includegraphics[width=\textwidth]{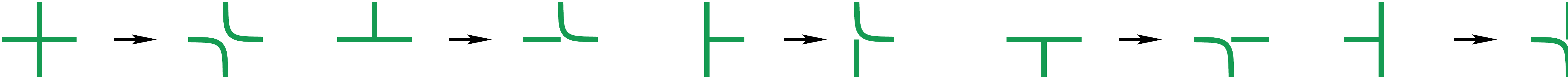}\\[3mm]
The resulting connected components of dual edges are called the Peierls contours of the configuration.
Among those, all form finite, closed loops, denoted \(\hat\gamma_1,\dots,\hat\gamma_m\), except for one infinite component that coincides with \(\calL\) outside the dual box \(\BOXs_N\).
We denote by \(\gamma\) the restriction of this infinite component to \(\BOXs_N\) or, more precisely, the part connecting the lower left corner \(\lc{N}\) to the lower right corner \(\rc{N}\), that is,
\begin{equation}\label{eq:corners}
\lc{N}\defby (-N-\tfrac12,-\tfrac12),\quad
\rc{N}\defby (N+\tfrac12,-\tfrac12).
\end{equation}
The open contour \(\gamma\) will be the central object in our investigations.

\subsection{Scaling, Ferrari-Spohn--diffusions and the main result}
\label{sub:result}

Before stating our main result, we need to introduce the relevant limiting diffusion, as well as the properly rescaled interface.

\subsubsection{Limiting Ferrari--Spohn diffusion}

Let \(\Delta_N \defby \{-N, \dots , N\}^2\).
Let us denote by \(m^*_\beta\defby \lim_{N\to\infty} \mu_{\Delta_N;\beta,0}^{\eta^+}(\sigma_0)\) the spontaneous magnetization and by \(\chi_\beta \defby \bigl(\tau_\beta(\sfe_2) + \tau''_\beta\bigr)^{-1}\), with \(\tau''_\beta = \big(\frac{d}{d \theta}\big)^2 \tau_\beta(\cos(\theta),\sin(\theta))\big|_{\theta=0}\), the curvature of the Wulff shape \(\mathbf{K}_\beta\) (see~\eqref{eq:WulffShape} below) in direction \(\sfe_2\). Here, \(\tau_\beta\) is the surface tension defined below in~\eqref{eq:DefSurfaceTension}.

Let \(\mathsf{Ai}\) denote the Airy function and let \(-\omega_1\) be its first zero. Let us introduce $\varphi_0(r) \defby \mathsf{Ai}\bigl((4\lambda m^*_\beta {\chi_\beta}^{1/2})^{1/3}\, r - \omega_1\bigr)$. The relevant Ferrari--Spohn diffusion for our setting is the diffusion on $(0,\infty)$ with generator
\begin{equation}
\label{eq:FSGenerator}
L_\beta\psi
\defby
\frac{1}{2}\psi'' +  \frac{\varphi_0'}{\varphi_0}\psi'
\end{equation}
and Dirichlet boundary condition at \(0\): \(\psi(0)=0\).
See Appendix~\ref{app:FS} for more information on this class of diffusion processes.

\subsubsection{Scaled interface}

Given a realization of the interface \(\gamma\), let us denote by \(\omega_\gamma\in\Omega_{\BOX_N}^\pm\) the configuration having \(\gamma\) as its unique contour.
We can then define the upper and lower ``envelopes'' \(\gamma^\pm:\bbZ\to\bbZ\) of \(\gamma\) by
\begin{gather*}
	\gamma^+(i) \defby \max\setof{j\in\bbZ}{\sigma_{(i,j)}(\omega_\gamma)=-1}
	+1 ,\\
	\gamma^-(i) \defby \phantom{\max}\mathllap{\min}\setof{j\in\bbZ}{\sigma_{(i,j)}(\omega_\gamma)=+1}
	-1 .
\end{gather*}
Note that \(\gamma^+(i)>\gamma^-(i)\) for all \(i\in\bbZ\).
It will follow from our analysis in Section~\ref{sec:OZTheory} that there exists \(K=K(\beta)<\infty\) such that, with probability
tending to \(1\) as \(n\to\infty\),
\begin{equation}\label{eq:Ising:Interface Width}
	\max_{|i|\leq R} |\gamma^+(i) - \gamma^-(i)| \leq K\log R.
\end{equation}
for any \(R>0\).
We use the same notation \(\gamma^\pm\) for their extension to functions on \(\bbR\) by linear interpolation.

Finally, we define the rescaled profiles \(\hat\gamma^\pm:\bbR \to \bbR\) by
\begin{equation}\label{eq:Ising:rescaled}
\hat\gamma^\pm_N(t) = N^{-1/3} {\chi_\beta}^{-1/2} \gamma^\pm(N^{2/3}t).
\end{equation}
Thanks to~\eqref{eq:Ising:Interface Width},
\[
\lim_{N\to\infty} \mu_{\BOXN;\beta,\lambda/N}^\pm \bigl(\sup_{t\in\bbR} |\hat\gamma^+_N(t) - \hat\gamma^-_N(t)| < \epsilon \bigr) = 1,\quad \forall\epsilon>0,
\]
so that in order to analyze the scaling limit of \(\gamma\), it is sufficient
to understand the limiting behavior of \(\hat\gamma^+_N\). More precisely, for any $T > 0$, we are interested in the distribution of \(\hat\gamma^+_N\) on the space of continuous functions $\sfC[-T,T]$ induced by the measure \(\mu_{\BOXN;\beta,\lambda/N}^\pm\).

\subsubsection{The main result}

\begin{theorem}\label{thm:Main}
	Consider the 2D Ising model at the inverse temperature $\beta > \beta_{c}$ with boundary condition \(\eta^\pm\) in the box \(\BOX_N\) and under the magnetic field  \(h=\lambda/N\), \(\lambda>0\). Let  \(\gamma\) be the corresponding interface (open contour), and \(\hat\gamma^+_N\) be its scaled version (defined by~\eqref{eq:Ising:rescaled}).
	As \(N\to\infty\), the distribution of \(\hat\gamma^+_N\) converges weakly to the distribution of the trajectories of the stationary Ferrari--Spohn diffusion  on $(0,\infty)$ with generator $L_\beta\psi = \frac{1}{2}\psi'' +  \frac{\varphi_0'}{\varphi_0}\psi'$,
	see~\eqref{eq:FSGenerator}, and Dirichlet boundary condition at \(0\).
\end{theorem}

\begin{remark}
	In particular, the typical height of the interface  \(\gamma\) is $\sim N^{1/3} {\chi_\beta}^{1/2}$.
\end{remark}

\subsection{Structure of the proof}\label{sec:StructureProof}
\label{sub:sketch}

The proof of Theorem~\ref{thm:Main} is based on a coupling between the interface \(\gamma\) of the Ising model and the interface of an effective model based on a directed random walk with area-tilt.
The proof of convergence to the corresponding Ferrari--Spohn diffusion then follows the scheme used for a similar class of effective models in~\cite{Ioffe+Shlosman+Velenik-15,Ioffe-Velenik-Wachtel-18} (see also Appendix~C of~\cite{Ioffe+Velenik-MPRF-2018}).
Here we briefly sketch the main steps to provide the reader with a picture of the whole argument. Precise statements and detailed proofs are given in later sections. We also emphasize that, as the present paper builds on previous work of some subsets of the authors, some arguments will not be repeated in full details. Therefore many references to previous papers are to be expected. In particular, while not necessary to follow the proofs, some familiarity with (a subset of)~\cite{Ioffe+Shlosman+Velenik-15,Ioffe-Velenik-Wachtel-18,Ioffe+Velenik-MPRF-2018,Ott-Velenik_2017,Ioffe+Ott+Velenik+Wachtel-20} will definitely smoothen the reading.

\smallskip
First note that any realization of the interface \(\gamma\) partitions the box \(\BOX_N\) into two (not necessarily connected) pieces: the part located ``above the interface'', denoted \(\BOX_N^+[\gamma]\), and the part located ``below the interface'', denoted \(\BOX_N^-[\gamma]\) (see Section~\ref{sec:tools} for precise definitions).

\subsubsection*{Weak localization of the interface}
We start with a very weak result on the maximal height reached by the interface (much stronger bounds could be extracted from the same argument, but are not needed).
Namely, we fix some arbitrary \(\delta\in (0,1)\) and show that, with high probability, the maximal height reached by the interface \(\gamma\) is lower than \(\delta N\) (Lemma~\ref{lem:RoughLocalizationInterface}).
This follows from the simple observation that the same is true in the absence of a magnetic field and that the presence of the latter only makes such an event less likely.
In particular, we deduce the following rough bound on the area below \(\gamma\):
\begin{equation}\label{eq:informal:RoughBoundOnArea}
	\abs{\BOX_N^-[\gamma]} \leq \delta (2N+1)N .
\end{equation}

\subsubsection*{All other contours are small}
The next step consists in proving the existence of \(\kappa=\kappa(\beta)\) such that, with high probability, all closed contours have diameter at most \(\kappa \log N\) (Proposition~\ref{prop:no_large_contours}).
The proof consists of two parts, dealing respectively with the contours inside \(\BOX_N^+[\gamma]\) and those inside \(\BOX_N^-[\gamma]\).

The validity of the claim in \(\BOX_N^+[\gamma]\) is rather obvious, since the plus phase is stable, and indeed follows immediately from the validity of the corresponding statement in the absence of a magnetic field and the fact that the presence of the latter only makes the diameter of contours smaller.

The claim in \(\BOX_N^-[\gamma]\) is more delicate, since the spins in this set are subjected to the \(-\) boundary condition and the \(-\) phase is only metastable: the competition between the boundary condition and the magnetic field may lead to the creation of giant contours inside \(\BOX_N^-[\gamma]\).
However, thanks to~\eqref{eq:informal:RoughBoundOnArea}, one can check that there is not enough room in \(\BOX_N^-[\gamma]\) to accommodate a supercritical droplet once \(\delta\) has been chosen small enough.

We denote \(\mu_{\BOXN;\beta,\lambda/N}^{\pm,\kappa}\) the measure \(\mu_{\BOXN;\beta,\lambda/N}^\pm\) conditioned on the absence of contours larger than \(\kappa\log N\). The above shows that both measures lead to the same behavior as \(N\to\infty\).

\subsubsection*{Effective weight due to the magnetic field}

The probability of a realization \(\gamma\) under the Gibbs measure with and without the magnetic field can be related by
\[
\mu_{\BOXN;\beta,\lambda/N}^{\pm,\kappa} (\gamma)
\propto
\mu_{\BOXN;\beta,0}^{\pm,\kappa} \Bigl( \exp\Bigl[ \tfrac{\lambda}{N} \sum_{i\in\BOX_N} \sigma_i \Bigr] \Bigm| \gamma \Bigr) \;
\mu_{\BOXN;\beta,0}^{\pm,\kappa} (\gamma) .
\]
Now, since all closed contours are small under \(\mu_{\BOXN;\beta,\lambda/N}^{\pm,\kappa}\), we expect that, conditionally on \(\gamma\), \(\sum_{i\in\BOX_N} \sigma_i \cong m^*_\beta \abs{\BOX_N^+[\gamma]} - m^*_\beta \abs{\BOX_N^-[\gamma]} = m^*_\beta\abs{\BOX_N} - 2 m^*_\beta \abs{\BOX_N^-[\gamma]}\), where \(m_\beta^*\) is the spontaneous magnetization.
This can be made precise using results in~\cite{Ioffe-Schonmann_1998}, which allows us to obtain
\begin{equation}\label{eq:sketch:effectiveMagnetic Field}
\mu_{\BOXN;\beta,\lambda/N}^{\pm,\kappa} (\gamma)
\propto
\exp\Bigl[ - 2 \frac{\lambda m^*_\beta}{N} \abs{\BOX_N^-[\gamma]} + \frac{\lambda \psi_N(\gamma)}{N}  \Bigr] \;
\mu_{\BOXN;\beta,0}^{\pm,\kappa} (\gamma) ,
\end{equation}
where \(\abs{\psi_N (\gamma)} \leq \nu_\beta \abs{\gamma}\), for all \(N\) and for any realization of the interface \(\gamma\).
The precise formulation appears in Proposition~\ref{prop:Comp-BN-bound}.

From this, it is easy to show that there exists \(C=C(\beta)\) such that, with high probability, \(\abs{\gamma}\leq CN\) (Proposition~\ref{prop:C-short}).
Note that this implies, in particular, that the term \(\lambda \psi_N(\gamma)/N\) above can be neglected when considering events the probability of which vanishes as \(N\to\infty\).

\subsubsection*{Entropic repulsion and reduction to infinite-volume weights}

In the next step, we prove the following rough entropic repulsion bound (Proposition~\ref{prop:ER}): for any small \(\epsilon>0\),
\begin{equation}\label{eq:sketch:roughEntropicRepulsion}
	\mu_{\BOXN;\beta,\lambda/N}^{\pm} (\gamma\cap\calB = \emptyset)
	\geq
	1 - N^{-\frac\epsilon2},
\end{equation}
where \(\calB \defby \setof{i=(i_1,i_2) \in {\BOX_N}}{\abs{i_1}\leq N^{1-5\epsilon}; 0 \leq i_2 \leq N^\epsilon}\).
To prove this, one first establishes a similar claim for a small box of width \(N^{2/3-\epsilon}\) and height \(N^{1/3+\epsilon}\).
In such a box, the effect of the magnetic field \(\lambda/N\) is weak enough to be ignored.
This allows us to deduce the claim by importing a recent entropic repulsion result from~\cite{Ioffe+Ott+Velenik+Wachtel-20} that applies to the model without magnetic field.
Claim~\eqref{eq:sketch:roughEntropicRepulsion} then follows from a union bound (over translates of the small box).

\subsubsection*{Coupling of the interface with an effective random walk model}

Section~\ref{sec:OZTheory} is devoted to the construction of a coupling between the law of \(\gamma\) under \(\mu_{\BOXN;\beta,\lambda/N}^{\pm}\) and the law of the trajectory of a directed random walk with area-tilt.
Let us first describe the latter and then make a few comments on the derivation of this coupling.
The presentation here is simplified (in other words, take what is said here with a grain of salt or, better, look at the actual claims in Section~\ref{sec:OZTheory}), since being precise would require the introduction of too much terminology.

We consider a (directed) random walk on \(\mathbb{Z}^2\), with transition probabilities \(\sfp\) supported on the cone \(\fcone=\setof{x=(x_1,x_2)\in \mathbb{Z}^2}{\abs{x_2}\leq x_1}\) and having exponential tails.
Given \(\sfu,\sfv\in\mathbb{Z}^2\) both in the upper-half plane and with \(\sfv-\sfu\) in the cone \(\fcone\),
we denote by \(\Pbetap{\sfu,\sfv}\) the law of this random walk conditioned to start at \(\sfu\), end at \(\sfv\) and remain inside the upper half-plane (the number of steps is not fixed).

Given a trajectory \(\underline s = (s_0=\sfu,s_1,\dots,s_M=\sfv)\) sampled from \(\Pbetap{\sfu,\sfv}\), we denote by \(\sfA(\underline s)\) the area delimited by the (linearly-interpolated) trajectory and the first coordinate axis.
We then define the area-tilted measure by
\[
\Pbetapl{\sfu,\sfv}(\underline s)
\propto
\exp\Bigl\{ -\frac{2\lambda m^*_\beta}{N} \sfA (\underline s) \Bigr\}
\Pbetap{\sfu,\sfv}(\underline s) .
\]
The main result of Section~\ref{sec:OZTheory} is then the existence of a coupling between the measure \(\Pbetapl{\sfu,\sfv}\) and the law of \(\gamma\) under \(\mu_{\BOXN;\beta,\lambda/N}^{\pm}\).
Of course, the transition probabilities \(\sfp\) have to be chosen properly and one has to appropriately average over the endpoints \(\sfu\) and \(\sfv\).

\smallskip
The construction of this coupling rests on an combination of the Ornstein--Zernike theory with the rough estimates mentioned above.
We first show that a typical realization of the interface under \(\mu_{\BOXN;\beta,0}^{\pm}\) can be written as a concatenation of microscopic pieces, sampled from a suitable \emph{product measure} conditioned on the fact that the concatenation must be a path from \(\lc{N}\) to \(\rc{N}\).
Moreover, using~\eqref{eq:sketch:effectiveMagnetic Field} and the comment following this equation, we can extend this result to a typical interface \(\gamma\) under \(\mu_{\BOXN;\beta,\lambda/N}^{\pm}\).

By construction, the microscopic pieces are confined inside ``diamonds'', so that the area \(\abs{\BOX_N^-[\gamma]}\) is well approximated by the area delimited by the corresponding polygonal approximation and the first coordinate axis, see Fig.~\ref{fig:approx-volume}.

\begin{figure}[h]
	\centering
	\resizebox{7cm}{!}{\includegraphics{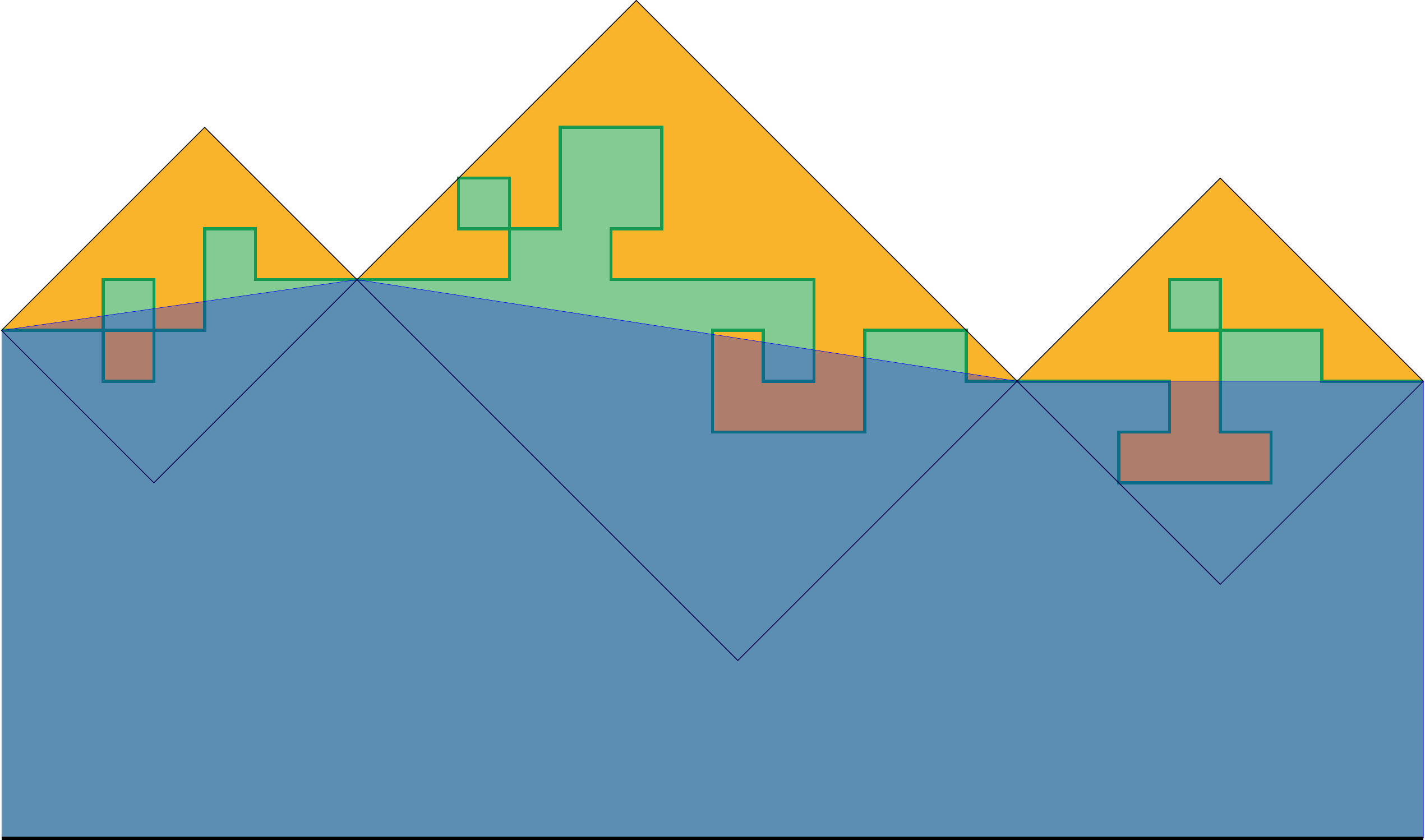}}
	\caption{Three consecutive microscopic pieces in the decomposition of the interface \(\gamma\). Each piece is contained inside a ``diamond''. The area \(\abs{\BOX_N^-[\gamma]}\) is well approximated by area delimited by the polygonal approximation and the first coordinate axis (blue shaded region).}
	\label{fig:approx-volume}
\end{figure}

By~\eqref{eq:sketch:roughEntropicRepulsion}, all the microscopic pieces located far from the two vertical walls of \(\BOX_N\) are, with high probability, above \(\calB\).
In particular, their distance to the boundary of \(\BOX_N\) diverges with \(N\).
This allows us to replace their finite-volume distributions by some limiting \emph{translation-invariant} infinite-volume distribution.

After that, the coupling is essentially a book-keeping exercise, apart from the need to have some weak control on the associated boundary points \(\sfu\) and \(\sfv\).

\subsubsection*{Proof of convergence for the effective random walk (ERW) model}

At this stage, we are left with proving the convergence of a suitably rescaled version of the area-tilted random walk to a Ferrari--Spohn diffusion.
We already proved such a result in a slightly different setting in~\cite{Ioffe+Shlosman+Velenik-15}.
The difference is that in the latter work, we were considering the space-time trajectory of a one-dimensional random walk rather than the spatial trajectory of a directed two-dimensional random walk.
Technically, the main issue is that the number of steps of our directed random walk between two points is random, while it was deterministic in the setting of~\cite{Ioffe+Shlosman+Velenik-15}.

Section~\ref{sec:Proof} is devoted to a proof of the convergence, focusing mainly on the necessary adaptations from~\cite{Ioffe+Shlosman+Velenik-15}.
Let us briefly sketch the main steps of the argument.

First, we introduce ``partition functions'' associated to our directed random walks with area-tilt:
\[
	\calG_{\beta,\lambda/N,+}^{n}[f](\sfu)
	=
	\sfE^{\sfu}_\beta \Bigl[
	\exp\Bigl( -\frac{2\lambda m^*_\beta}{N}\, \sfA(\underline{\sfS}[0,n]) \Bigr) f(\sfZ_n) \,,\,
	\underline{\sfS}[0,n]\subset \bbH_+
	\Bigr] ,
\]
where \(\sfE^{\sfu}_\beta\) denotes expectation with respect to the ERW \(\underline{\sfS}\) starting at \(\sfu\), \(\underline{\sfS}[0,n]\) denotes the first \(n\) steps of the trajectory of \(\underline{\sfS}\) and $\sfZ_n$ is the vertical coordinate of the endpoint of \(\underline{\sfS}[0,n]\).
It is easy to express the finite-dimensional distributions of \(\underline{\sfS}\) in terms of (ratio of) such partition functions (this is done in Section~\ref{sec:convFDD}).
This reduces the proof of convergence of the finite-dimensional distributions to the convergence of these partition functions.
The latter is done using a convergence theorem in~\cite{ethier2009markov}, as explained in Lemma~\ref{lem:TK-form}.

Finally, tightness can be proved by a straightforward adaptation of the approach we used in~\cite{Ioffe-Velenik-Wachtel-18},
but we provide an alternative approach relying on a suitable time-change, which reduces the analysis to a fixed number of steps; this allows us to conclude as in~\cite{Ioffe+Ott+Velenik+Wachtel-20} (this is explained in Section~\ref{sec:Tightness}).

\subsection{Open problems}

We list here some open problems that it would be interesting to address in the future.
\begin{itemize}
	\item In this paper, we have assumed from the start that \(h=\lambda/N\).
	One would expect, however, that the same limiting process should appear when \(h\downarrow 0\) as \(N\to\infty\) under only mild assumptions on the speed (obviously, \(h\) cannot tend to \(0\) too fast, otherwise one would recover the usual Brownian bridge as scaling limit; this would trivially be the case if \(h=\smo{N^{-2}}\) for instance).
	However, it seems that one should be allowed to let \(h\) go to \(0\) arbitrarily slowly.
	In fact, one should even be able to first take the limit as \(N\to\infty\) at fixed \(h>0\) and then, in a second step, let \(h\downarrow 0\).
	Namely, one can consider the Ising model in the upper half plane with \(-\) boundary condition and under a magnetic field \(h>0\).
	Such a model has an infinite interface \(\gamma\), separating the \(-\) phase region near the boundary from the \(+\) phase in the bulk.
	As \(h\downarrow 0\), the mean height of the interface diverges.
	However, if one scales it by a factor \(h^{1/3}\) vertically and by \(h^{2/3}\) horizontally, the rescaled interface should converge in the limit \(h\downarrow 0\) to the Ferrari--Spohn diffusion process on the line.
	\item We have dealt here with boundary condition forcing the interface to lie along the bottom wall of the box.
	As mentioned in the Introduction, an alternative geometric setting of interest corresponds to working in the box \(\Lambda_N\) with \(-\) boundary condition on all four sides of the box and a positive magnetic field \(h=\lambda/N\) with \(\lambda>\lambda_{\rm c}\).
	In this case, typical configurations exhibit the behavior described in the right panel of Fig.~\ref{fig:metastability}.
	It would be interesting to prove that the scaling limit of the boundary of the macroscopic droplet of \(+\) phase far away from the corners is the same Ferrari--Spohn diffusion we obtained in Theorem~\ref{thm:Main}.
	\item Another related problem would be to derive the scaling limit of the Ising model in \(\Lambda_N\) with \(\pm\) boundary condition and a magnetic field \(h>0\) in the upper half plane and \(h<0\) in the lower half plane.
	The limiting diffusion process (as \(N\to\infty\) and \(h\to 0\) in a suitable way and after a similar scaling as used here) should be a variant of the Ferrari--Spohn diffusion obtained in Theorem~\ref{thm:Main} (of course, now a diffusion on \(\bbR\) rather than \(\bbR_+\)).
	\item More general spatial dependence of magnetic fields might be of interest as well.
	In principle, \cite{Ioffe+Shlosman+Velenik-15,Ioffe-Velenik-Wachtel-18} cover a large family of self-potentials for effective models of one or finitely many ordered random walks. Scaling limits for effective models with magnetic fields (self-potentials) of the type \(h(x,y) = \frac{x}{N}\) acting in a strip of width \(N\) were derived in \cite{de2019interface}.
	The motivation for the latter work came from desire to understand the phenomenon of \emph{Uphill diffusions} as driven by non-equilibrium dynamics in the phase transition regime.
	See \cite{colangeli2018nonequilibrium,de2019note,DIMP20} for more detail.
	\item It would be interesting to extend the analysis made in the current paper to situations in which the unstable phase appears at the interface between two stable phases, rather than along the system's boundary (the phenomenon of \emph{interfacial adsorption}).
	\item Another related problem is the description of the scaling limit of the level lines of a two-dimensional SOS interface above a hard wall, as discussed in~\cite{Caputo+Lubetzky+Martinelli+Sly+Toninelli-16}, or of the level lines of SOS models coupled with bulk high and low density Bernoulli fields, as discussed in~\cite{Ioffe+Shlosman-08, Ioffe+Shlosman-19}.
	Note that, in these two problems, one has an unbounded collection of interacting random paths.
	The analysis of the scaling limit of such a system is still open, although progress has been made (see~\cite{Ioffe-Velenik-Wachtel-18} for the scaling limit in the case of finitely many paths, and~\cite{Caputo+Ioffe+Wachtel-19a,Caputo+Ioffe+Wachtel-19b} for preliminary results for infinitely many paths).
\end{itemize}

\section{Surface tension; random-line representation}
\label{sec:tools}

\subsection{Surface tension}
Recall the notation \(\Delta_N = \setof{x=(x_1,x_2)\in\bbZ^2}{-N\leq x_1,x_2 \leq N}\).
We write \(a\cdot b\) for the scalar product between two vectors \(a\) and \(b\) and \(\sfe_i ;\, i=1,2,\) for the coordinate vectors. Given any unit vector \(\uvec \in \bbS^1\subset \bbR^2\), the surface tension in direction \(\uvec\) is defined as
\begin{equation}
	\label{eq:DefSurfaceTension}
	\taub(\uvec) \defby \lim_{N\to\infty} -
	\frac{{\normsup{\uvec}}}{2N} \log \frac{\PF_{\Delta_N;\beta,0}^{\uvec}}{\PF_{\Delta_N;\beta,0}^+},
\end{equation}
where, for \(\uvec = (e_1, e_2)\), the notation \(\normsup{\uvec} \defby \max\{\abs{e_1}, \abs{e_2}\}\) stands for the supremum norm and where \(\PF_{\Delta_N;\beta,0}^{\uvec}\) is the partition function in \(\Delta_N\) with boundary condition \(\eta^{\uvec}\) given by
\[
	\eta^{\uvec}_i \defby
	\begin{cases}
		\phantom{-}1	&	\text{if } i\cdot \uvec \geq 0, \\
		-1				&	\text{if } i\cdot \uvec < 0.
	\end{cases}
\]
Note that, in the above notation, \(\PF_{\Delta_N;\beta,0}^{\sfe_2} = \PF_{{\Delta_N};\beta,0}^\pm\).
In fact, for the axis direction \(\sfe = \sfe_2\),
\eqref{eq:DefSurfaceTension} can be rewritten in terms of the half box \(\BOX_N\),
\begin{equation}\label{eq:st-BOXN}
	\taub
	\defby
	\taub(\sfe_2)
	=
	\lim_{N\to\infty} -\frac1{2N} \log \frac{\PF_{\BOXN;\beta,0}^{\pm}}{\PF_{\BOXN;\beta,0}^+} .
\end{equation}
Existence of the limit in~\eqref{eq:DefSurfaceTension} and its coincidence with the expression on the right-hand side of~\eqref{eq:st-BOXN} are well known~\cite{Frohlich+Pfister-87}.
The function \(\taub (\cdot ) \) can then be extended to \(\bbR^2\) by positive homogeneity:
\begin{equation}\label{eq:DefSurfaceTension2}
	\taub(0) \defby 0
	\qquad\text{and, for all } x\in\bbR^2\setminus\{0\},
	\qquad
	\taub(x) \defby \taub(x/\normII{x})\,\normII{x} .
\end{equation}
The homogeneous extension  of \(\taub\) is the support function of the so-called \emph{Wulff shape}
\begin{equation}\label{eq:WulffShape}
	\mathbf{K}_\beta \defby \bigcap_{n\in\bbS^{1}} \setof{t\in\bbR^2}{t\cdot n \leq \taub(n)} .
\end{equation}
The surface tension possesses the following  properties:
\begin{itemize}
	\item For all \(\beta>\betac\), \(x\mapsto\taub(x)\) is a norm.
	\item \textbf{Sharp triangle inequality}~\cite[Theorem~B]{Campanino-Ioffe-Velenik_2003}.
	There exists \(\kappa=\kappa(\beta) {> 0}\) such that, for all \(x,y\in\bbR^2\),
	\begin{equation}\label{eq:SharpTriangleInequality}
	\taub(x) + \taub(y) - \taub(x+y) \geq \kappa \bigl(\normII{x} + \normII{y} - \normII{x+y}\bigr) .
	\end{equation}
\end{itemize}

\subsection{Random-line representation}
\label{sub:random-lines}

In this section, we collect a number of results on the random-line representation, as developed in~\cite{Pfister+Velenik-97,Pfister-Velenik_1999}; proofs can be found in the latter works and precise references will be given below.

\medskip

Let \(\Lambda\Subset\bbZ^2\) be such that \(\hat{\Lambda}\) is a simply connected subset of \(\bbR^2\) (remember~\eqref{eq:Ahat}).
Fix
\begin{equation}\label{eq:DualBox}
\sfu,\sfv\in\Lambda^* \defby \setof{w\in\bbZ^{2,*}}{\exists z\in\Lambda, \normII{z-w} = \frac{1}{\sqrt{2}}}
= \hat{\Lambda}\cap\bbZ^{2,*}  .
\end{equation}
Let us denote by \(\calC_{\sfu,\sfv}\) the set of all \(\bbZ^{2,*}\)-edge-self-avoiding paths from \(\sfu\) to \(\sfv\) such that deforming the path at each vertex with at least 3 incoming edges according to the deformation rules on page~\pageref{deformationRules} does not split the path into several components.

One can define a weight \(q_{\Lambda,\beta}:\calC_{\sfu,\sfv}\to\bbR_{\geq 0}\) with the following properties (see~\cite{Pfister-Velenik_1999} for an explicit description):
\begin{enumerate}[label=\(\sfP_\arabic*\):, ref=\(\sfP_\arabic*\)]
	\item\label{property:SupportedInVol} \(q_{\Lambda,\beta}(\gamma) = 0\) when \(\gamma\not\subset\hat\Lambda\).
	\item\label{property:MonotVolWeight}
	\(\Lambda\subset\Lambda'\subset\bbZ^2\) implies that \(q_{\Lambda,\beta}(\gamma)\geq q_{\Lambda',\beta}(\gamma)\) for all \(\gamma\subset\hat\Lambda\)
	\cite[Lemma~6.3]{Pfister-Velenik_1999}.
	\item\label{property:BK} For any \(t\in\Lambda^*\),
	\[
	\sum_{\substack{\gamma\in\calC_{\sfu,\sfv}\\\gamma\ni t}} q_{\Lambda,\beta}(\gamma) \leq \Bigl( \sum_{\gamma\in\calC_{\sfu,t}} q_{\Lambda,\beta}(\gamma) \Bigr) \Bigl( \sum_{\gamma\in\calC_{t,\sfv}} q_{\Lambda,\beta}(\gamma) \Bigr) ,
	\]
	where \(\gamma\ni t\) means that the path \(\gamma\) visits the vertex \(t\)
	\cite[Lemma~6.5]{Pfister-Velenik_1999}.
	\item\label{property:MonotVolSum} For any \(\Lambda\subset\Lambda'\subset\bbZ^2\),
	\cite[(6.9)]{Pfister-Velenik_1999}
	\begin{equation}\label{eq:Prop4}
	\sum_{\gamma\in\calC_{\sfu,\sfv}} q_{\Lambda,\beta}(\gamma) \leq \sum_{\gamma\in\calC_{\sfu,\sfv}} q_{\Lambda',\beta}(\gamma) \leq \mathrm{e}^{-\taub(\sfv-\sfu)} ,
	\end{equation}
	where \(\taub\) is the surface tension defined in~\eqref{eq:DefSurfaceTension2}.
	\item\label{property:Mixing} There exist \(c_\beta\) and \( g_\beta \),
	such that, for any \(\Lambda\subset\Lambda'\subset\bbZ^2\) and any \(\gamma\subset\hat\Lambda\),
	\begin{equation}\label{eq:P5}
	\Babs{\frac{q_{\Lambda,\beta}(\gamma)}{q_{\Lambda',\beta}(\gamma)} - 1} \leq c_\beta \sum_{t\in\gamma} \sum_{t'\in\partial^{\rm ext}\Lambda'\setminus\Lambda} \mathrm{e}^{-g_\beta\normII{t-t'}} .
	\end{equation}
	Furthermore,
	\begin{equation}\label{eq:cbeta-nubeta}
		\lim_{\beta\to\infty} c_\beta = 0\quad  \text{ and }\quad
		\lim_{\beta\to\infty} g_\beta = \infty .
	\end{equation}
	(See \cite[Lemma~6.8]{Pfister-Velenik_1999}.)
	\item\label{property:Concat}
	Let \(\Lambda\subset\bbZ^2\) and \(\sfu,\sfv,t\in\Lambda^*\) be distinct dual vertices.
	Two paths \(\gamma_1\in\calC_{\sfu,t}\) and \(\gamma_2\in\calC_{t,\sfv}\) are \emph{compatible} if they are edge-disjoint and applying the deformation rule to their union leaves the two paths unchanged.
	We then denote by \(\gamma_1\circ\gamma_2\in\calC_{\sfu,\sfv}\) the concatenation of the two paths. This extends inductively to an arbitrary finite number of paths.
	One then has the property~\cite[Lemma~6.4]{Pfister-Velenik_1999}
	\begin{equation}\label{eq:LowerBndProductQ}
		q_{\Lambda,\beta}(\gamma_1\circ\dots\circ\gamma_n) \geq \prod_{i=1}^n q_{\Lambda,\beta}(\gamma_i).
	\end{equation}
	\item\label{property:CondWeights}
	Given a concatenation \(\gamma_1\circ \gamma_2\),
	we can define the conditional weight
	\[
	q_{\Lambda,\beta}(\gamma_1\given\gamma_2)
	\defby
	\frac{q_{\Lambda,\beta}(\gamma_1\circ\gamma_2)}{q_{\Lambda,\beta}
		(\gamma_2)}
		=
	q_{\Lambda\setminus\Delta(\gamma_2),\beta}(\gamma_1),
	\]
	where \(\Delta(\gamma_2)\) is the \emph{edge-boundary} of \(\gamma_2\) (essentially the edges incident to vertices of \(\gamma_2\), but with some adjustments to take into account the deformation rule; see~\cite[Lemma~6~.4]{Pfister-Velenik_1999} for a precise statement).
	Similarly, given a concatenation \(\gamma_1\circ\gamma_2\circ\gamma_3\), we set
	\[
	q_{\Lambda,\beta}(\gamma_2\given\gamma_1,\gamma_3)
	\defby
	q_{\Lambda\setminus\Delta(\gamma_1)\cup\Delta(\gamma_3),\beta}(\gamma_2).
	\]
\end{enumerate}

\medskip
We will use the following notations in special cases:
\begin{equation}\label{eq:q-weights}
q_{N,\beta} \equiv q_{\BOX_N,\beta},\qquad
q_{\beta}   \equiv q_{\bbZ^2,\beta}.
\end{equation}
We will also write \(\calC_N \equiv \calC_{\lc{N},\rc{N}}\).
Note, in particular, that any path \(\gamma\in \calC_N\) that does not stay inside \(\BOXs_N\) satisfies \(q_{N, \beta} (\gamma )=0\) (by~\eqref{property:SupportedInVol}).

The weight \(q_{N,\beta}\) is related to our problem through the following identity~{\cite[(2.10), (6.8), (6.11)]{Pfister-Velenik_1999}}: for any \(\gamma\in\calC_N\)
\begin{equation}\label{eq:q-mu-form}
	q_{N,\beta}(\gamma)
	=
	\frac{\PF_{\BOXN;\beta,0}^\pm}{\PF_{\BOXN;\beta,0}^+} \, \mu_{\BOXN;\beta,0}^\pm (\gamma)
\end{equation}
(\(\mu_{\BOXN;\beta,0}^\pm (\gamma)\) denoting, of course, the probability that \(\gamma\) is the realization of the interface connecting \(\lc{N}\) and \(\rc{N}\) when \(h=0\)).
In particular, returning to our geometric setting, property~\ref{property:MonotVolSum} above implies  that
\begin{equation}\label{eq:sub-ratio-pf}
\frac{\PF_{\BOXN;\beta,0}^{\pm}}{\PF_{\BOXN;\beta,0}^+} \leq \mathrm{e}^{-\taub(\rc{N}-\lc{N})}.
\end{equation}
In fact, Lemmas~3.2 and~3.3 in~\cite{Ioffe+Ott+Velenik+Wachtel-20} imply that there exist two constants
\(0 < \underline{c}_\beta \leq \overline{c}_\beta < \infty\), such that
\begin{equation}\label{eq:ratio-q-mu}
	\frac{\underline{c}_\beta}{N^{3/2}}
	\leq
	\frac{\mathrm{e}^{2\taub N}\PF_{\BOXN;\beta,0}^{\pm}}{\PF_{\BOXN;\beta,0}^+}
	=
	\frac{\mathrm{e}^{2\taub N}q_{N,\beta}(\gamma)}{\mu_{\BOXN;\beta,0}^\pm (\gamma)}
	\leq
	\frac{\overline{c}_\beta}{N^{3/2}} .
\end{equation}
for all \(N\) sufficiently large and all \(\gamma\in \calC_N\).

Note also that the following conditional version of~\eqref{eq:q-mu-form} holds:
\[
	\mu_{\BOXN;\beta,0}^\pm (\gamma_2\given \gamma_1 , \gamma_3)
	=
	\frac{q_{\beta , N} (\gamma_2\given \gamma_1 , \gamma_3)}{\sum_{\gamma^\prime} q_{\beta , N} (\gamma^\prime\given \gamma_1 , \gamma_3)} .
\]

\medskip
Let \(\gamma\) be one possible realization of the open contour in the box \(\BOX_N\). Denote by \(\Omega_{\BOX_N}^\pm[\gamma]\) the set of all configurations in \(\Omega_{\BOX_N}^\pm\) in which the open contour is \(\gamma\). Let
\begin{align}
\Delta^+_N[\gamma] &\defby \setof{i\in\BOX_N}{\forall\sigma\in\Omega_{\BOX_N}^\pm[\gamma],\,\sigma_i = 1},\notag\\
\Delta^-_N[\gamma] &\defby \setof{i\in\BOX_N}{\forall\sigma\in\Omega_{\BOX_N}^\pm[\gamma],\,\sigma_i = -1}\label{eq:DeltaMinus}
\end{align}
be the sets of all vertices of the box \(\BOX_N\) at which the spins take deterministically the value \(+1\), resp.\ \(-1\), when \(\gamma\) is the open contour.
We then denote by \(\BOX_N^+[\gamma]\), resp.\ \(\BOX_N^-[\gamma]\), the union of all the components of \(\BOX_N\setminus(\Delta^+_N[\gamma]\cup\Delta^-_N[\gamma])\) that are bordered by \(+\) spins, resp.\ \(-\) spins, in any configuration in \(\Omega_{\BOX_N}^\pm[\gamma]\).
\begin{figure}[t]
	\centering
	\includegraphics{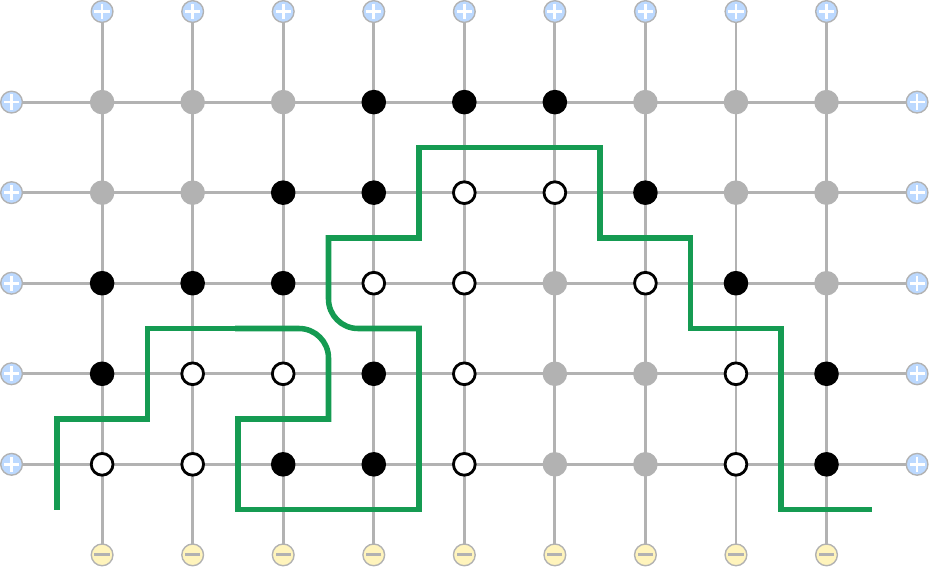}
	\caption{A realization of the contour \(\gamma\) and the corresponding sets \(\Delta^+[\gamma]\) (black nodes) and \(\Delta^-[\gamma]\) (white nodes).}
	\label{fig:DeltaPMgamma}
\end{figure}

\medskip
In view of the deformation rules given above, it is natural to define a compatible notion of connectivity for sets of vertices in \(\bbZ^2\).

\begin{definition}
	\label{def:s-path}
	We shall say that two vertices \(x,y\in\bbZ^2\) are s-connected if either \(\normII{y-x}=1\), or \(\normII{y-x}=\sqrt{2}\) and the two vertices are oriented NW-SE (see Fig.~\ref{fig:S-conn}). An s-path is then a sequence \(x_1,\dots,x_n\) of vertices in \(\bbZ^2 \) such that \(x_{k}\) and \(x_{k-1}\) are s-connected for each \(n\geq k\geq 2\).
	Observe that a contour (in particular, the open contour \(\gamma\)) can never cross an s-path along which the spins take a constant value.
\end{definition}
\begin{figure}[ht]
	\centering
	\resizebox{2cm}{!}{\includegraphics{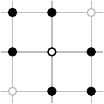}}
	\caption{The black nodes are the vertices that are s-connected to the central vertex (white node).}
	\label{fig:S-conn}
\end{figure}

It follows, for instance, from the skeleton calculus of~\cite{Campanino-Ioffe-Velenik_2003} that, for any \(\beta>\betac\), there exists \(c=c(\beta)>0\) such that
\begin{equation}\label{eq:expdecaySconn}
	\mu_{\BOXN;\beta,0}^+(\text{\(i\) and \(j\) are connected by an s-path of \(-\) spins}) \leq \mathrm{e}^{-c\normII{j-i}} .
\end{equation}

\section{Localization of \(\gamma\), absence of intermediate contours and tilted-area approximation}
\label{sec:NoIntermediateContours}

In this section, we first prove a rough upper bound which yields localization of the interface \(\gamma\) on macroscopic scale; this is the content of Lemma~\ref{lem:RoughLocalizationInterface} below.
This leads to a reduction to the so-called \emph{phase of small contours} in Subsection~\ref{sub:Ph-sc} and then to the crucial tilted-area representation~\eqref{eq:Comp-BN-bound} of interface probabilities.
The last result of this section is Proposition~\ref{prop:C-short} which warrants a linear (in \(N\)) upper bound on the interface length \(\abs{\gamma}\).

\subsection{Localization of \(\gamma\) on the macroscopic scale}
\label{sub:gamma-loc}

Consider the events
\[
	A_N^\delta
	\defby
	\bsetof{\sigma\in\Omega_{\BOX_N}^\pm}{\exists t\in \gamma\ \text{such that }t_2 >\delta N} .
\]
\begin{lemma}\label{lem:RoughLocalizationInterface}
Let \(\beta>\betac\) and \(1>\delta>0\). There exists a positive constant \(c(\beta)\)
such that, for any \(N\geq 1\) and any \(\lambda\geq 0\),
\[
	\mu_{\BOXN;\beta,\lambda/N}^\pm (A_N^\delta)
	\leq
	\mu_{\BOXN;\beta,0}^\pm (A_N^\delta)
	\leq
	\mathrm{e}^{-c(\beta)\delta^2 N} .
\]
\end{lemma}
\begin{proof}
First, observe that the event \(A_N^\delta\) occurs if and only if an s-path of \(-\)-spins connects
the vertex \((-N,-1)\) to a vertex of \(\setof{i=(i_1,i_2)\in\BOX_N}{i_2> \delta N}\) and is thus
clearly non-increasing.
Since we take \(\lambda \geq 0\), the FKG inequality thus implies that
\[
	\mu_{\BOXN;\beta,\lambda/N}^\pm (A_N^\delta) \leq \mu_{\BOXN;\beta,0}^\pm (A_N^\delta) .
\]
Using now~\ref{property:BK}, \ref{property:MonotVolSum} and the sharp triangle inequality~\eqref{eq:SharpTriangleInequality}, we have
\begin{align*}
	\mu_{\BOXN;\beta,0}^\pm (A_N^\delta)
	&\leq
	\sum_{\substack{t\in\BOX_N\\ {t_2>\delta N  } }}\mu_{\BOXN;\beta,0}^\pm ( \gamma\ni t ) \\
	&\leq
	\mathrm{e}^{(1+\smo{1}) \st(x_r -x_l )} \sum_{\substack{t\in\BOX_N\\
	{t_2 > \delta N }}} \mathrm{e}^{- \st(t-x_l) - \st(x_r-t) } \\
	&\leq
	\mathrm{e}^{-c(\beta)N\delta^2} .\qedhere
\end{align*}
\end{proof}

\subsection{Reduction to the phase of small contours}
\label{sub:Ph-sc}

Using Lemma~\ref{lem:RoughLocalizationInterface}, we can prove the first main result of this section:
all contours (except \(\gamma\)) have diameter at most \(\bgo{\log N}\).
\begin{definition}
		\label{def:restricted_contour_phase}
		Given \(\kappa\), denote by \(C_\kappa\) the event that all contours except \(\gamma\) have diameter at most \(\kappa\log N\).
		Define also the \emph{restricted contour phase}
		\begin{equation}\label{eq:kappa-phase}
		\mu_{\BOXN;\beta,\lambda/N}^{\pm,\kappa}(\cdot)
		\defby
		\mu_{\BOXN;\beta,\lambda/N}^\pm (\,\cdot \,|\, C_\kappa).
		\end{equation}
\end{definition}
\begin{proposition}\label{prop:no_large_contours}
	For any \(\beta>\betac\), there exists \(\kappa_0(\beta)<\infty\) such that,
	for all \(\kappa>\kappa_0(\beta)\) and for any \(\lambda\geq 0\) fixed,
	\[
	\lim_{N\to\infty}
	\mu_{\BOXN;\beta,\lambda/N}^\pm (C_\kappa) = 1 .
	\]
	Consequently, once \(\kappa>\kappa_0(\beta)\), one can construct a coupling \(\Psi_N^\kappa\) of the measures \(\mu_{\BOXN;\lambda/N}^{\pm}\) and \(\mu_{\BOXN;\lambda/N}^{\pm , \kappa}\) such that, if \((\sigma,\sigma^\prime) \sim \Psi_N^{\kappa}\),
	\begin{equation}\label{eq:kappa-couple}
		\Psi_N^{\kappa}(\sigma  \neq \sigma^\prime ) \xrightarrow{N\to\infty} 0.
	\end{equation}
\end{proposition}
\begin{remark}\label{rem:kappa}
	By~\eqref{eq:kappa-couple}, we are entitled not to distinguish between the original Gibbs measure \(\mu_{\BOXN;\lambda/N}^{\pm}\) and its \(\kappa\log N\) cut-off \(\mu_{\BOXN;\lambda/N}^{\pm, \kappa}\).
	We shall stress the super-index \(\kappa\) whenever this will facilitate references to computations that can be found in the literature.
\end{remark}
\begin{proof}[Proof of Proposition~\ref{prop:no_large_contours}]
We first write
\[
	\mu_{\BOXN;\beta,\lambda/N}^\pm (C_\kappa^{\sfc})
	\leq
	\sum_\gamma \mu_{\BOXN;\beta,\lambda/N}^\pm (\gamma)
	\Bigl\{
	\mu_{\BOX_N^+[\gamma];\beta,\lambda/N}^+ (C_\kappa^{\sfc})
	+ \mu_{\BOX_N^-[\gamma];\beta,\lambda/N}^- (C_\kappa^{\sfc})
	\Bigr\} .
\]
Now, observe that \(C_\kappa^{\sfc}\) occurs in \(\BOX_N^+[\gamma]\)
if and only if the latter set contains an \(s\)-path of \(-\)-spin of diameter
at least \(\kappa\log N\).
Since this event is obviously non-increasing, we conclude that
\[
	\mu_{\BOX_N^+[\gamma];\beta,\lambda/N}^+ (C_\kappa^{\sfc})
	\leq
	\mu_{\BOX_N^+[\gamma];\beta,0}^+ (C_\kappa^{\sfc}) ,
\]
and the latter probability tends to \(0\) as \(N\to\infty\) (for \(\kappa\) large enough)
by~\eqref{eq:expdecaySconn}.

Let us thus turn to the more delicate term \(\mu_{\BOX_N^-[\gamma];\beta,\lambda/N}^- (C_\kappa)\).
Thanks to Lemma~\ref{lem:RoughLocalizationInterface}, we can assume without loss of generality
that, for some fixed \(\delta>0\) (which we will choose small enough below), \(t_2 \leq \delta N\) for any \(t\in\gamma\).

Evidently,
\[
	\mu_{\BOX_N^-[\gamma];\beta,\lambda/N}^- (C_\kappa^{\sfc})
	\leq
	\sum_{D\geq\kappa\log N} \sum_{\hat\gamma:\,\diam\hat\gamma = D}
	\mu_{\BOX_N^-[\gamma];\beta,\lambda/N}^-(\hat\gamma \text{ is an external contour}) ,
\]
where \(\hat\gamma\) is an external contour if no contour in \(\BOX_N^-[\gamma]\) or \(\BOX_N^+[\gamma]\) surrounds \(\inte\hat{\gamma}\).

Let us introduce some notation.
First, given realizations of the open contour \(\gamma\) and a closed contour \(\hat\gamma\) in \(\BOX_N^-[\gamma]\), let us denote by \(\sigma^{\gamma,\hat\gamma}\) the configuration in \(\Omega_{\BOX_N}^\pm\) whose set of contours is exactly \(\{\gamma,\hat\gamma\}\).
Let also \(\Omega_{\BOX_N}^\pm[\gamma,\hat\gamma]\) be the set of all configurations in \(\Omega_{\BOX_N}^\pm\) having \(\gamma\) and \(\hat\gamma\) as two of their contours.
We then set
\begin{gather*}
	\Delta_N^-\equiv\Delta_N^-[\gamma,\hat\gamma] \defby \bsetof{i\in\BOX_N^-[\gamma]}{\forall\sigma\in\Omega_{\BOX_N}^\pm[\gamma,\hat\gamma],\,\sigma^{\vphantom{\hat\gamma}}_i=\sigma^{\gamma,\hat\gamma}_i=-1},\\
	\Delta_N^+\equiv\Delta_N^+[\gamma,\hat\gamma] \defby \bsetof{i\in\BOX_N^-[\gamma]}{\forall\sigma\in\Omega_{\BOX_N}^\pm[\gamma,\hat\gamma],\,\sigma^{\vphantom{\hat\gamma}}_i=\sigma^{\gamma,\hat\gamma}_i=1},
\end{gather*}
and define (see Fig.~\ref{fig:IntExtContour})
\[
	\ext\hat\gamma \defby \bsetof{i\in\BOX_N^-[\gamma]}{\sigma^{\gamma,\hat\gamma}_i = -1}\setminus\Delta_N^-,\qquad
	\inte\hat\gamma \defby \bsetof{i\in\BOX_N^-[\gamma]}{\sigma^{\gamma,\hat\gamma}_i = 1}\setminus\Delta_N^+.
\]
\begin{figure}[t]
	\centering
	\resizebox{\textwidth}{!}{\includegraphics{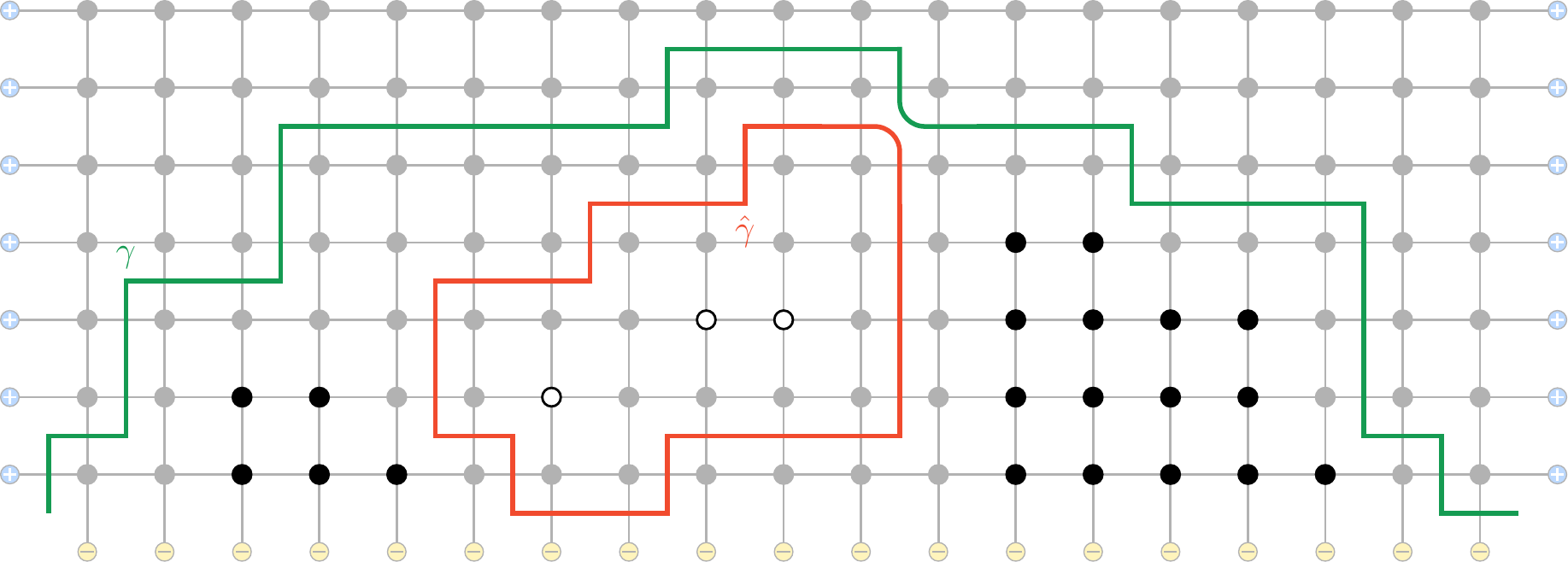}}
	\caption{The sets \(\ext\hat\gamma\) (black nodes) and \(\inte\hat\gamma\) (white nodes) associated to an (external) contour \(\hat\gamma\) in \(\BOX_N^-[\gamma]\).}
	\label{fig:IntExtContour}
\end{figure}%
Finally, we let \(\overline{\inte\hat\gamma}\defby \inte\hat\gamma \cup \Delta_N^+\).
An elementary but crucial observation is that, for any contour \(\hat\gamma\) of diameter \(D\) in \(\BOX_N^-[\gamma]\),
\begin{equation}\label{eq:Vol-delta}
	\abs{\overline{\inte \hat\gamma}} \leq D (D \wedge \delta N).
\end{equation}

Let \(G\) denote the event that the contour \(\hat{\gamma}\) is external.
Note that this is equivalent to saying that there is an s-path of \(-\) spins in \(\BOX_N^-[\gamma]\)
connecting \(\partial\inte\hat{\gamma}\) to  \(\partial\BOX_N\), which is a non-increasing event.
\begin{multline*}
	\mu_{\BOX_N^-[\gamma];\beta,(\lambda/N)}^- (\hat\gamma \text{ is an external contour})\\
	=
	\mathrm{e}^{-2\beta \abs{\hat\gamma}-(\lambda/N)\nabs{\Delta_N^-}+(\lambda/N)\nabs{\Delta_N^+}}\, \frac{\PF_{\ext\hat\gamma;\beta,\lambda/N}^-[G]\,
	\PF_{\inte\hat\gamma;\beta,\lambda/N}^+}{\PF_{\BOX_N^-[\gamma];\beta,\lambda/N}^-} ,
\end{multline*}
where \(\PF_{\ext\hat\gamma;\beta,\lambda/N}^-[G]\) denotes the partition function in
\(\ext\hat\gamma\) with \(-\) boundary condition restricted to configurations in \(G\).
Observe that
\begin{multline*}
	\frac{\PF_{\ext\hat\gamma;\beta,\lambda/N}^-[G]\, \PF_{\inte\hat\gamma;\beta,\lambda/N}^+}{\PF_{\BOX_N^-[\gamma];\beta,\lambda/N}^-}
	=
	\frac{\PF_{\ext\hat\gamma;\beta,\lambda/N}^-[G]\, \PF_{\inte\hat\gamma;\beta,\lambda/N}^-}{\PF_{\BOX_N^-[\gamma];\beta,\lambda/N}^-}
	\frac{\PF_{\inte\hat\gamma;\beta,\lambda/N}^+}{\PF_{\inte\hat\gamma;\beta,\lambda/N}^-}\\
	=
	\mathrm{e}^{{(\lambda/N)\nabs{\Delta_N^-}+(\lambda/N)\nabs{\Delta_N^+}}}\,
	\mu_{\BOX_N^-[\gamma];\beta,\lambda/N}^- \bigl( \{\sigma_i=-1\; \forall i\sim\hat\gamma\} \cap G \bigr) \,
	\frac{\PF_{\inte\hat\gamma;\beta,\lambda/N}^+}{\PF_{\inte\hat\gamma;\beta,\lambda/N}^-} ,
\end{multline*}
where \(i\sim\hat\gamma\) means that \(i\) is one of the spins along \(\hat\gamma\) whose
value is the same in all configurations in which \(\hat\gamma\) is an external contour.
The event \(\{\sigma_i=-1\; \forall i\sim\hat\gamma\}\cap G\) being non-increasing, we conclude that
\begin{align*}
	\mu_{\BOX_N^-[\gamma];\beta,\lambda/N}^- (\hat\gamma &\text{ is an external contour})\\
	&=
	\mathrm{e}^{-2\beta \abs{\hat\gamma}+2(\lambda/N)\nabs{\Delta_N^+}}\,
	\mu_{\BOX_N^-[\gamma];\beta,{\lambda /N}}^- (\sigma_i=-1\; \forall i\sim\hat\gamma; G) \,
	\frac{\PF_{\inte\hat\gamma;\beta,\lambda/N}^+}{\PF_{\inte\hat\gamma;\beta,\lambda/N}^-}\\
	&\leq
	\mathrm{e}^{-2\beta \abs{\hat\gamma}+2(\lambda/N)\nabs{\Delta_N^+}}\,
	\mu_{\BOX_N^-[\gamma];\beta,0}^- (\sigma_i=-1\; \forall i\sim\hat\gamma; G) \,
	\mathrm{e}^{2\lambda\abs{\inte\hat\gamma}/N}\\
	&=
	\mu_{\BOX_N^-[\gamma];\beta,0}^- (\hat\gamma \text{ is an external contour}) \,
	\mathrm{e}^{2\lambda\abs{\overline{\inte\hat\gamma}}/N} .
\end{align*}
Therefore, in view of \eqref{eq:Vol-delta},
\[
	\mu_{\BOX_N^-[\gamma];\beta,\lambda/N}^- (C_\kappa^{\sfc})
	\leq
	\sum_{D\geq\kappa\log N} \! \mathrm{e}^{2\lambda D(D\wedge \delta N)/N} \! \sum_{\hat\gamma:\,\diam\hat\gamma = D} \! \mu_{\BOX_N^-[\gamma];\beta,0}^- (\hat\gamma \text{ is an external contour}) .
\]
The sum over \(\hat\gamma\) is bounded above by \(\mathrm{e}^{-c D}\) for some \(c=c(\beta)>0\) by the usual skeleton upper bounds~\cite{Campanino-Ioffe-Velenik_2003} (provided that \(\kappa\) be large enough).
We thus finally obtain
\[
	\mu_{\BOX_N^-[\gamma];\beta,\lambda/N}^- (C_\kappa)
	\leq
	\sum_{D\geq\kappa\log N} \mathrm{e}^{-(c-2\lambda\delta) D},
\]
which also vanishes as \(N\to\infty\), provided that we choose \(\delta<c/(2\lambda)\).
\end{proof}

\subsection{Tilted-area approximation of \(\mu_{\BOXN;\beta,\lambda/N}^{\pm}(\gamma)\)}
\label{sub:taa}

The reduction to the phase of small contours paves the way for our second main result in this section.
Given an interface \(\gamma\), define
\begin{equation}\label{eq:psi-N}
	\psi_N (\gamma)
	\defby
	\sum_{i\in \BOX_N^+[\gamma]}
	\Bigl\{
	\mu_{\BOXN;\beta,0}^{\pm,\kappa} (\sigma_i \,|\, \gamma) -
	\mu_{\BOXN;\beta,0}^{+,\kappa} (\sigma_i)
	\Bigr\}
	+
	\sum_{i\in \BOX_N^-[\gamma]}
	\Bigl\{
	\mu_{\BOXN;\beta,0}^{\pm,\kappa} (\sigma_i \,|\, \gamma) -
	\mu_{\BOXN;\beta,0}^{-,\kappa} (\sigma_i)
	\Bigr\}.
\end{equation}

By the exponential relaxation of finite-volume expectations towards their pure phase values (see, \textit{e.g.}, \cite{Ioffe-Schonmann_1998}), the correction term \(\psi_N\) in \eqref{eq:psi-N} satisfies the following bound: There exists \(\nu_\beta <\infty\) such that
\begin{equation}\label{eq:psi-N-bound}
	\abs{\psi_N (\gamma)} \leq \nu_\beta \abs{\gamma},
\end{equation}
for all \(N\) and for any realization of the interface \(\gamma\).

\begin{proposition}\label{prop:Comp-BN-bound}
	The following asymptotic formula holds uniformly in interfaces \(\gamma\) as soon as \(N\) is sufficiently large:
	\begin{equation}\label{eq:Comp-BN-bound}
		\mu_{\BOXN;\beta,\lambda/N}^{\pm , \kappa} (\gamma) \,\bigl(1+\smo{1}\bigr)
		\propto
		\exp\Bigl[ - 2 \frac{\lambda m^*_\beta}{N} \abs{\BOX_N^-[\gamma]} + \frac{\lambda \psi_N(\gamma)}{N}  \Bigr] \;
		\mu_{\BOXN;\beta,0}^{\pm ,\kappa} (\gamma) .
	\end{equation}
\end{proposition}
In view of Remark~\ref{rem:kappa}, it should not cause ambiguity to drop the small-contour upper
index \(\kappa\) on either side of \eqref{eq:Comp-BN-bound}, which we shall do without further comment in the sequel.
\begin{proof}
For any given \(\lambda\),
\[
	\mu_{\BOXN;\beta,\lambda/N}^{\pm, \kappa} (\gamma)
	\propto
	\mu_{\BOXN;\beta,0}^{\pm,\kappa} \Bigl( \exp\Bigl[ \tfrac{\lambda}{N} \sum_{i\in\BOX_N} \sigma_i \Bigr] \Bigm| \gamma \Bigr) \;
	\mu_{\BOXN;\beta,0}^{\pm,\kappa} (\gamma) .
\]
Next, given a magnetic field \(g = g_N \defby \lambda/N\), one can proceed as in~\cite[Section~2.3]{Ioffe-Schonmann_1998} and derive, uniformly in the interface \(\gamma\),
the following expansion of conditional expectations:
\begin{align}
	\log \mu_{\BOXN;\beta,0}^{\pm,\kappa} \bigl( &\mathrm{e}^{ g \sum_{i\in \BOX_N}\sigma_i} \bigm| \gamma \bigr)  +\smo{1} \notag\\
	&\quad=
	g \, \sum_{i\in \BOX_N} \mu_{\BOXN;\beta,0}^{\pm,\kappa} (\sigma_i \,|\, \gamma) + \frac{g^2}{2}\sum_{i,j\in\BOX_N}
	\mu_{\BOXN;\beta,0}^{\pm,\kappa} (\sigma_i ;\sigma_j \,|\, \gamma) \notag\\
	&\quad=
	c_N + g m^*_\beta \abs{\BOX_N} + \frac{g^2}{2} \chi^*_\beta  \abs{\BOX_N}
	- 2 g m^*_\beta \abs{\BOX_N^-[\gamma]} + g \psi_{N} (\gamma),
	\label{eq:exp-g}
\end{align}
where, \(\mu_{\BOXN;\beta,0}^{\pm,\kappa} (\sigma_i ;\sigma_j \,|\, \gamma)\) denotes the covariance between \(\sigma_i\) and \(\sigma_j\) under \(\mu_{\BOXN;\beta,0}^{\pm,\kappa} (\,\cdot \,|\, \gamma)\), \(m^*_\beta\) and \(\chi^*_\beta\) are, respectively, the spontaneous magnetization and the
susceptibility of the infinite-volume plus state, \(c_N\) is a boundary term of order \(N\) coming from the \(\pm\) boundary condition on \(\partial\BOX_N\) and the function \(\psi_N(\gamma)\) is precisely the one defined in~\eqref{eq:psi-N}.
\end{proof}
Next, we will show that
one can restrict attention to interfaces $\gamma$ satisfying \linebreak\(\abs{\gamma}, \psi_N(\gamma) \leq CN\).

\subsection{Control of interface lengths \(\abs{\gamma}\)}
\label{sub:i-length}

\begin{proposition}\label{prop:C-short}
	For any \(\beta>\betac\), and \(\lambda\geq 0\) fixed, there exists \( C = C(\beta,\lambda) < \infty\) such that
	\begin{equation}\label{eq:C-short}
		\lim_{N\to\infty} \mu_{\BOXN;\beta,\lambda/N}^\pm (\abs{\gamma} \leq CN) = 1 .
	\end{equation}
\end{proposition}
\begin{proof}
Recall the Definition~\ref{def:restricted_contour_phase} of the restricted phase \(\mu_{\BOXN;\beta,\lambda/N}^{\pm, \kappa}\) of \(\kappa\log N\)-contours and the statement of Proposition~\ref{prop:no_large_contours}.
Clearly,
\[
	\mu_{\BOXN;\beta,\lambda/N}^\pm (\abs{\gamma} > C N)
	\leq
	\mu_{\BOXN;\beta,\lambda/N}^\pm (C_\kappa^\comp)
	+
	\mu_{\BOXN;\beta,\lambda/N}^{\pm , \kappa}  (\abs{\gamma} > C  N).
\]
We already know that the first term in the right-hand side vanishes as \(N\to\infty\), so we only discuss the second one.
Using the usual skeleton calculus for Ising phase separation lines (which is based on a certain BK-type property, see~\cite{Pfister-Velenik_1999, Campanino-Ioffe-Velenik_2003}),
we infer the existence of \(\alpha = \alpha(\beta) > 0\) and \(C_0 = C_0(\beta) < \infty\) such that
\begin{equation}\label{eq:Skel-UB}
	\mu_{\BOXN;\beta,0}^{\pm, \kappa} (\abs{\gamma} \geq CN) \leq \mathrm{e}^{-\alpha C N}
\end{equation}
simultaneously for all \(C\geq C_0\) as soon as \(N\) is sufficiently large. Our target~\eqref{eq:C-short} thus follows from~\eqref{eq:Comp-BN-bound} and the  estimate~\eqref{eq:psi-N-bound}, once \(C\) is taken large enough.
\end{proof}
\begin{remark}\label{rem:psiN}
	By~\eqref{eq:psi-N-bound} and in view of Proposition~\ref{prop:C-short}, we can restrict
	attention to the case in which \(\frac{\abs{\psi_N (\gamma)}}{N} \leq C\nu_\beta\).
	Accordingly, as long as we discuss events whose \(\mu_{\BOXN;\beta,\lambda/N}^{\pm}\)-probability
	asymptotically vanishes, we may ignore the term \(\frac{{\psi_N(\gamma)}}{N}\) on the right-hand side
	of~\eqref{eq:Comp-BN-bound}.
\end{remark}

\section{Rough entropic repulsion bound}
\label{sec:rough-ER}

Define the rectangular boxes and, for further reference, the vertical semi-strips
\begin{gather}\label{eq:boxes-strips}
	\calB_{M,R} \defby \setof{i=(i_1,i_2) \in {\BOX_N}}{\abs{i_1}\leq M; {0 \leq }i_2\leq R}\\
	\label{eq:strip}
	\calS_{L,R} \defby \setof{i=(i_1,i_2) \in \bbZ^2}{ L \leq {i_1}\leq R ; i_2\geq 0}\\
	\intertext{\rm and}
	\calS_{L} \defby \setof{i=(i_1,i_2) \in \bbZ^2}{ |i_1|\leq L}.
\end{gather}
For horizontal shifts of rectangular boxes, we shall use \(\calB_{M,R}(\ell) \defby (\ell,0) + \calB_{M,R}\).
Furthermore, for \(\sfu,\sfv\in\bbH_+\), we use \(\calS_{\sfu,\sfv}\) for the semi-strip in~\eqref{eq:strip} through the horizontal projections of \(\sfu\) and \(\sfv\).

Let us say that a contour \(\gamma\) does not hit a subset \(A \subset \bbZ^2\) (which we denote \(\gamma \cap A =\emptyset\))  if \(\gamma \cap \hat{A} = \emptyset\) (recall~\eqref{eq:Ahat} for the definition of \(\hat{A}\)).
Our goal in this section is to prove the following
\begin{proposition}\label{prop:ER}
	For all \(\epsilon>0\) small enough,
	\[
		\mu_{\BOXN;\beta,\lambda/N}^{\pm} (\gamma\cap\calB_{{3} N^{1-5\epsilon},N^\epsilon}\neq\emptyset)
		\leq
		N^{-\frac\epsilon2}.
	\]
\end{proposition}

\begin{proof}
To shorten notation, let \(\Delta_1 \defby \calB_{N^{2/3 -{\epsilon}}, N^{\epsilon}}\) and \(\Delta_2 \defby \calB_{2N^{2/3-\epsilon}, N^{1/3+\epsilon}}\).

Observe first that the event \(\{\gamma\cap\Delta_1\neq\emptyset\}\) is equivalent to the existence of an s-path of \(+\) spins connecting \((-N-1, -1)\) to \(\Delta_1\); in particular, this is an increasing event.
It then follows from the FKG inequality (by fixing to \(+1\) all spins in \(\BOX_N\setminus\Delta_2\)) that
\[
	\mu_{\BOXN;\beta,\lambda/N}^{\pm} (\gamma\cap\Delta_1\neq\emptyset)
	\leq
	\mu_{\Delta_2;\beta,\lambda/N}^{\pm} (\gamma\cap\Delta_1 \neq\emptyset).
\]
We can now get rid of the magnetic field: since \(\abs{\Delta_2} {<5}N\),
\[
	\mu_{\Delta_2;\beta,\lambda/N}^{\pm} (\gamma\cap\Delta_1\neq\emptyset)
	\leq
	\mathrm{e}^{
		{10} \lambda} \, \mu_{\Delta_2;\beta,0}^{\pm} (\gamma\cap\Delta_1\neq\emptyset).
\]
Our next step is to replace the box \(\Delta_2\) by the larger box \(\Delta_3= \calB_{2N^{2/3-\epsilon},2N^{2/3-\epsilon}}\) (we want to import results from~\cite{Ioffe+Ott+Velenik+Wachtel-20} which are stated for half square boxes).
By the BK-type Property~\ref{property:BK} of the path weights and the sharp triangle inequality~\eqref{eq:SharpTriangleInequality}, there exists \(c=c(\beta)>0\) such that
\begin{equation}\label{eq:StripM}
	\mu_{\Delta_3;\beta,0}^{\pm} (\gamma\not\subset\calB_{2N^{2/3-\epsilon},\frac12 N^{1/3+\epsilon}})
	\leq
	\mathrm{e}^{-cN^{3\epsilon}},
\end{equation}
for all \(N\) large enough.
Thus, writing \(A\) for the event that there is an s-path of \(+\) spins crossing \(\{(i_1,i_2):\ |i_1|\leq 2N^{2/3-\epsilon}, \frac12 N^{1/3+\epsilon} \leq  i_2\leq  N^{1/3+\epsilon} \}\) from left to right, we have \(\mu_{\Delta_3;\beta,0}^{\pm}(A) = 1+\smo{1}\) (as a direct consequence of the spatial Markov property and exponential decay of \(-\) connections under \(+\) boundary conditions).

In particular, as a consequence of monotonicity (FKG),
\begin{equation}\label{eq:boxtostrip}
	\mu_{\Delta_2;\beta,0}^{\pm} (\gamma\cap\Delta_1\neq\emptyset)
	=
	\bigl(1+\smo{1}\bigr)\, \mu_{\Delta_3;\beta,0}^{\pm} (\gamma\cap\Delta_1\neq\emptyset).
\end{equation}
Indeed, on the one hand, one has \(\mu_{\Delta_2;\beta,0}^{\pm} (\gamma\cap\Delta_1\neq\emptyset)
\geq \mu_{\Delta_3;\beta,0}^{\pm} (\gamma\cap\Delta_1\neq\emptyset)\) as \(\{\gamma\cap\Delta_1\neq\emptyset\}\) is increasing.
On the other hand, denoting \(X\) the upper-most path realizing \(A\),
\begin{align*}
	\mu_{\Delta_3;\beta,0}^{\pm} (\gamma\cap\Delta_1\neq\emptyset) &\geq \mu_{\Delta_3;\beta,0}^{\pm} (\gamma\cap\Delta_1\neq\emptyset, A)\\
	&= \sum_{\eta} \mu_{\Delta_3;\beta,0}^{\pm} (X=\eta)\mu_{\eta^-;\beta,0}^{\pm} (\gamma\cap\Delta_1\neq\emptyset)\\
	&\geq \mu_{\Delta_2;\beta,0}^{\pm} (\gamma\cap\Delta_1\neq\emptyset)\mu_{\Delta_3;\beta,0}^{\pm} (A),
\end{align*}
where the sum is over \(\eta\) realizing \(A\), \(\eta^-\) is the volume below \(\eta\), and we used FKG and the spatial Markov property.

At this stage, one can use an entropic repulsion bound derived in~\cite{Ioffe+Ott+Velenik+Wachtel-20}:
\begin{equation}\label{eq:ER-bound}
	\mu_{\Delta_3;\beta,0}^{\pm} (\gamma\cap\Delta_1\neq\emptyset)
	\leqs
	N^{-\frac13+\frac{7\epsilon}2}.
\end{equation}
This is a simple re-running of the proof of~\cite[Lemma 3.6]{Ioffe+Ott+Velenik+Wachtel-20} with our choice of boxes (the claim in~\cite{Ioffe+Ott+Velenik+Wachtel-20} is actually stronger, as it applies to the FK-percolation cluster that contains the interface \(\gamma\)).

Collecting all the above estimates, we conclude that there exists \(c=c(\beta,\lambda)\) such that
\[
	\mu_{\BOXN;\beta,\lambda/N}^{\pm} (\gamma\cap\Delta_1\neq\emptyset)
	\leq
	c N^{-\frac13+\frac{7\epsilon}2}.
\]
Observe that the exact same bound holds for any horizontal translate of \(\Delta_1\) such that the corresponding translate of \(\Delta_2\) is contained inside \(\BOX_N\).
Therefore, using a union bound, we can apply the above inequality to \(3N^{1/3 - 4\epsilon}\) successive translates of \(\Delta_1\) to obtain the desired claim.
\end{proof}

\section{Renewal structure and Random Walk representation of interfaces}
\label{sec:OZTheory}

In this Section we introduce and develop a factorization of the interface weights \(q_{N,\beta}\) and \(q_\beta\), as defined in~\eqref{eq:q-weights}.

\subsection{Irreducible decomposition of the interface}

Let us recall the notions and basic facts related to the irreducible decomposition of interfaces~\cite{Campanino-Ioffe-Velenik_2003,Campanino-Ioffe-Velenik_2008,Ott-Velenik_2017,Ioffe+Ott+Velenik+Wachtel-20}.
We shall try to keep  the notation close to  the one employed in \cite{Ioffe+Ott+Velenik+Wachtel-20}.

First of all, one defines cones and the associated diamonds:
\begin{gather*}
	\fcone \defby \setof{i=(i_1,i_2)\in\bbZ^{2}}{i_1 \geq \abs{i_2}}, \quad
	\bcone \defby -\fcone,\\
	\calD(\sfu,\sfv) \defby (\sfu+\fcone)\cap (\sfv+\bcone).
\end{gather*}
Let \(\eta\) be a connected subgraph of \(\bbZ^{2,*}\), specifically a portion of the interface \(\gamma\) or, more generally, any path satisfying the deformation rules.
We will say that \(\sfu\in\eta\) is a \emph{cone-point} of \(\eta\) if
\begin{equation*}
	\eta\subset \sfu+ (\bcone\cup \fcone).
\end{equation*}
We denote \( \CPts(\eta)\) the set of cone-points of \(\eta\).

\begin{definition}\label{def:path-types}
We will say that \(\eta\) is:
\begin{itemize}
	\item \emph{Forward-confined} if there exists \(\sfu \in {\eta}\) such that \({\eta}\subset \sfu+\fcone\). When it exists, such an \(\sfu\) is unique; we denote it \(\fend({\eta})\).
	\item \emph{Backward-confined} if there exists \(\sfv\in {\eta}\) such that \({\eta}\subset \sfv+\bcone\). When it exists, such a \(\sfv\) is unique; we denote it \(\bend({\eta})\).
	\item \emph{Diamond-confined} if it is both forward- and backward-confined.
	\item \emph{Irreducible} if it is diamond-confined and
	it does not contain cone-points other than its end-points \(\bend (\eta)\) and
	\(\fend (\eta)\).
\end{itemize}
\end{definition}
In the sequel, we shall employ the following notation:
\begin{align*}
	\SetRootMarkForwCont(\sfu) &\defby \{\eta \text{ forward-confined with \(\fend(\eta) =\sfu\)}\}, \\
	\SetRootMarkBackCont(\sfv) &\defby \{\eta \text{ backward-confined with \(\bend(\eta) = \sfv\)}\}, \\
	\SetRootDiaCont (\sfu,\sfv) &\defby \SetRootMarkBackCont(\sfv)\cap \SetRootMarkForwCont(\sfu) =
	\{\eta \text{ diamond-confined with \(\fend(\eta) =\sfu,\, \bend(\eta) = \sfv\)}\}, \\
	\SetRootDiaCont^{\textnormal{irr}} (\sfu,\sfv) &\defby \{\eta\in\SetRootDiaCont (\sfu,\sfv)
	\text{ and is irreducible}\}.
\end{align*}
Note that the concatenation  \(\eta_1\circ\eta_2\) is well defined whenever \(\eta_1\in \SetRootMarkBackCont (\sfu )\) and  \(\eta_2\in \SetRootMarkForwCont (\sfu )\).
In the sequel,
\[
	\SetRootMarkForwCont \defby \bigcup_\sfu \SetRootMarkForwCont(\sfu) ,\,
	\SetRootMarkBackCont \defby \bigcup_\sfv \SetRootMarkBackCont(\sfv) ,\,
	\SetRootDiaCont \defby \bigcup_{\sfu,\sfv} \SetRootDiaCont(\sfu,\sfv)
	\text{ and }
	\SetRootDiaCont^{\textnormal{irr}} \defby \bigcup_{\sfu,\sfv} \SetRootDiaCont^{\textnormal{irr}}(\sfu,\sfv).
\]

\smallskip
In our context, the irreducible decomposition of the interface \(\gamma\) is quantified by the mass-gap upper bound (see~\eqref{eq:massgap-mu} below), which was established in~\cite{Campanino-Ioffe-Velenik_2003} in the case of the infinite-volume weights \(q_\beta\).
The proof there was based on upper bounds and surcharge inequalities, which were used to control the geometry of contours on large finite scales.
By the domain-monotonicity property~\ref{property:MonotVolSum}, these estimates extend to the finite-volume weights \(q_{N,\beta}\).
By routine adjustments, it then follows that there exist positive finite constants \(\epsilon=\epsilon_\beta\) and \(\phi=\phi_\beta\) such that the following lower bound on the
density of \(\CPts(\gamma)\) holds in our context for all \(N\) large enough:
\begin{equation}\label{eq:massgap-mu}
	\mu_{\BOXN;\beta,0}^{\pm} ( \abs{\CPts(\gamma)} \leq \epsilon_\beta N )
	 \leq
	 \mathrm{e}^{- 2\phi_\beta N} .
\end{equation}
Choosing \(\delta\) in Lemma~\ref{lem:RoughLocalizationInterface}	to be sufficiently small (for instance \(\delta < \frac{\phi_\beta}{\lambda}\)), we may assume, in view of \eqref{eq:Comp-BN-bound}, that the following version of the \emph{irreducible} decomposition of~\cite{Campanino-Ioffe-Velenik_2003} holds,
up to exponentially small \(\mu_{\BOXN;\beta,\lambda/N}^{\pm}\)-probabilities, for the interface \(\gamma\):
\begin{equation}\label{eq:IRD1}
	\gamma = \gamma_L \circ \gamma_1 \circ \dots \circ \gamma_m \circ \gamma_R ,
\end{equation}
where \(\gamma_L\in\SetRootMarkBackCont\), \(\gamma_R\in\SetRootMarkForwCont\) and \(\gamma_1,\dots,\gamma_m\in\SetRootDiaCont^{\textnormal{irr}}\), whereas \(m\geq \epsilon_\beta N\).
The expression \emph{irreducible decomposition} implies that~\eqref{eq:IRD1} cannot be further refined, that is, the boundary pieces \(\gamma_L\) and \(\gamma_R\) do not contain any non-trivial cone-points.

The following  refinement of~\eqref{eq:massgap-mu} was established in~\cite{Ioffe+Ott+Velenik+Wachtel-20} (see Subsection~\ref{sub:ltPotts} for more detail):
\begin{proposition}\label{prop:ref-Potts-mg}
	There exists \(\phi=\phi_\beta>0\) such that, for any \(\rho\in (0,1)\) and \(d\in\bbN\) fixed,
	\begin{equation}\label{eq:massgap-mep}
		\mu_{\BOXN;\beta,0}^{\pm} ( \CPts(\gamma) \cap \calS_{\ell,\ell+d N^{\rho}} = \emptyset )
		\leq
		\mathrm{e}^{- 2\phi_\beta d N^{\rho }} ,
	\end{equation}
	for each \(\ell\in \{-N,\dots,N-dN^{\rho }\}\). (The semi-strip \(\calS_{\ell,\ell+d N^{\rho }}\) is defined in~\eqref{eq:strip}.)
\end{proposition}

Below, we will need twice a lower bound on the probability that a portion of the interface remains inside a narrow tube. We provide here a result that is sufficient for our purposes (and essentially follows the corresponding argument for random walks in~\cite{Hryniv-Velenik_2004}).
\begin{lemma}\label{lem:tube}
	The following is true for all integers \(C\) and \(M\) large enough.
	Let \(\sfu=(u_1,u_2),\sfv=(v_1,v_2)\) be such that \(v_1-u_1 > C M^2 > C^2\abs{v_2-u_2}\). Let \(\sfe\) denote a unit vector normal to the line through \(\sfu\) and \(\sfv\) and
	\[
		\calT \defby \calD(\sfu,\sfv) \cap \setof{x\in\bbR^2}{\abs{(x-\sfu)\cdot \sfe} \leq M}.
	\]
	Then, there exists \(c=c(\beta)\) such that, for all \(M\) large enough,
	\[
	\sum_{\gamma'\in\calC_{\sfu,\sfv}} \IF{\gamma' \subset \hat\calT} q_{\beta}(\gamma') \geq \frac{1}{c} \mathrm{e}^{- c \normII{\sfv-\sfu}/M^2 -\taub(\sfv-\sfu)}.
	\]
\end{lemma}
\begin{proof}
	For \(k\) an integer, let
	\[
	\Pi_k \defby \setof{i=(i_1,i_2)\in\calT}{i_1 = u_1 + k M^2,\abs{(i-\sfu)\cdot \sfe} \leq M/2 }.
	\]
	Let \(K \defby \lfloor (v_1-u_1)/M^2 \rfloor-1\). Then,
	\[
	\sum_{\gamma'\in\calC_{\sfu,\sfv}} \IF{\gamma' \subset \hat\calT} q_\beta(\gamma')
	\geq
	\sum_{t_1,\dots,t_K} \sum_{\gamma'_0,\dots,\gamma'_{K}} q_\beta(\gamma'_0\circ\dots\circ\gamma'_{K}) ,
	\]
	where the first sum is over vertices \(t_1,\dots,t_K\) such that \(t_k\in\Pi_k\) for all \(1\leq k\leq K\) and the second sum is over paths \(\gamma'_0,\dots,\gamma'_{K}\) such that each path \(\gamma'_k\) belongs to \(\SetRootDiaCont(t_{k},t_{k+1})\) and stays inside \(\hat\calT\) (with the convention that \(t_{0} = \sfu\) and \(t_{K+1} = \sfv\)); see Fig.~\ref{fig:tube}.
	
	By property~\ref{property:Concat}, \(q_\beta(\gamma'_0\circ\dots\circ\gamma'_{K+1}) \geq \prod_{k=0}^K q_\beta(\gamma'_k)\).
	By the invariance principle established in~\cite{greenberg2005invariance}, each of the sums over \(\gamma_k\), \(0\leq k\leq K\), is bounded below by \(\tilde cM^{-1}\mathrm{e}^{-\taub(t_{k+1}-t_k)}\) for some constant \(\tilde c=\tilde c(\beta)>0\).
	Finally, each sum over \(t_k\) contributes a factor at least \(M\).
	Therefore, provided that \(M\) was chosen large enough,
	\[
	\sum_{\gamma'\in\calC_{\sfu,\sfv}} \IF{\gamma' \subset \hat\calT_M} q_\beta(\gamma')
	\geq
	\Bigl(\frac{\tilde c}{M}\Bigr)^{K+1} M^{K+1} \mathrm{e}^{-\sum_{k=0}^K \taub(t_{k+1}-t_k)}
	\geq
	\mathrm{e}^{-c K} \mathrm{e}^{- \taub(\sfv-\sfu)} ,
	\]
	where we used \(\sum_{k=0}^K \taub(t_{k+1}-t_k) = \taub(\sfv-\sfu) + \bgo{K}\). Indeed, denoting by \(\theta_k\) the angle between the vectors \(t_{k+1}-t_k\) and \(\sfv-\sfu\) and parametrizing the surface tension by the angle: \(\taub((t_{k+1}-t_k)/\normII{t_{k+1}-t_k}) \equiv \taub(\theta_k)\), we have
	\[
		\taub(t_{k+1}-t_k)
		= \taub(\theta_k) \normII{t_{k+1}-t_k}
		= \bigl[\taub(0) + \taub'(0)\theta_k + \bgo{\theta_k^2}\bigr] \normII{t_{k+1}-t_k}.
	\]
	The claim follows by observing that \(\sum_k \normII{t_{k+1}-t_k} = \normII{\sfv-\sfu}+\bgo{K}\), \(\theta_k = \bgo{1/M}\) and \(\sum_k \theta_k \normII{t_{k+1}-t_k} = \sum_k \sin(\theta_k) \normII{t_{k+1}-t_k} + \smo{K} = \smo{K}\), since the last sum vanishes, which is geometrically evident.
\end{proof}

\begin{figure}[th]
	\centering
	\resizebox{\textwidth}{!}{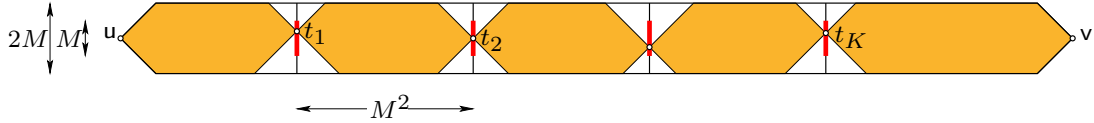}
	\caption{
		Sketch (not-to-scale) of the construction in the proof of Lemma~\ref{lem:tube}. We restrict the sum to paths staying inside the shaded area. The probabilistic cost of this confinement is constant in each of the \(K+1\) intervals, thanks to the invariance principle.}
	\label{fig:tube}
\end{figure}%

\subsection{The set \(\calC_N^{\rm reg}\) and rough estimates under \(\mu_{\BOXN;\beta,\lambda /N}^{\pm}\)}
\label{sub:Creg}

In this subsection, we develop rough estimates under the measure \(\mu_{\BOXN;\beta,\lambda/N}^{\pm}\) that enable a reduction to a certain subset \(\calC_N^{\rm reg}\subset\calC_N\) of interfaces which we call \emph{regular} interfaces.
The definition of \(\calC_N^{\rm reg}\) is adjusted to our needs and depends on three scales \(\kappa\), \(\iota\) and \(\epsilon\) satisfying
\begin{equation}\label{eq:kappa-scales-g}
	1 >\kappa  > \iota  > \frac23 > \frac13 > \epsilon > 0.
\end{equation}
It should be kept in mind that \((N^{\frac23},N^{\frac13})\) is the natural (and diffusive) scale when considering the statistical properties of the interface \(\gamma\) under \(\mu_{{\BOXN};\beta,\lambda/N}^\pm\).
The relatively rough estimates we derive here are adjusted to this scale.

Following~\eqref{eq:kappa-scales-g}, set
\begin{equation}\label{eq:dlrh-N-g}
	d_N \defby d_+ \sqrt[3]{N},\
	\ell_N \defby -N^{\kappa} + N^{\iota},\
	r_N \defby N^{\kappa} - N^{\iota}
	\text{ and }
	h_N \defby c_+ \sqrt[3]{N}.
\end{equation}
\begin{definition}\label{def:Creg}
	Fix \(C\) sufficiently large.
	We say that a realization of the interface \(\gamma\) belongs to the set
	\[
		\calC_N^{\sf reg}
		=
		\calC_N^{\sf reg} (c_+,d_+,\kappa,\iota,\epsilon) \subset \calC_N
	\]
	if the following four conditions~\eqref{eq:C0}--\,\eqref{eq:C3} are satisfied by its irreducible decomposition~\eqref{eq:IRD1}:
	\begin{equation}
		\label{eq:C0}\tag{\(\sfC^{\sf reg}_{0,N}\)}
		\abs{\CPts(\gamma)} \geq \frac1C N\quad \text{ and }\quad \abs{\gamma}\leq C N .
	\end{equation}
	Furthermore (see Fig.~\ref{fig:regularInterfaces}),
	\begin{equation}
		\label{eq:C1}\tag{\(\sfC_{1,N}^{\sf reg }\)}
		\gamma\text{ stays above the rectangular box \(\calB_{N^\kappa,N^\epsilon}\), that is, }
		\gamma\cap \calB_{N^\kappa,N^\epsilon} = \emptyset .
	\end{equation}
	Next, there is a cone-point of \(\gamma\) in every semi-strip \(\calS_{\ell,\ell+d_N}\) of width \(d_N\):
	\begin{equation}
		\label{eq:C2}\tag{\(\sfC^{\sf reg}_{2 , N}\)}
		\CPts(\gamma)\cap \calS_{\ell,\ell+d_N} \neq \emptyset
		\text{ for any }
		\ell\in \{-N,\dots,N-d_N\} .
	\end{equation}
	Finally, \(\gamma\) puts at least one cone-point in each of the two shifted rectangular boxes \(\calB_{N^{\iota},2h_N} (\ell_N)\) and \(\calB_{N^{\iota},2 h_N} (r_N)\) (that is, the left-most and
	right-most sub-boxes of width \(2N^{\iota}\) and height \(2h_N\) sitting on the bottom of the strip \(\calS_{N^\kappa} = \calS_{-N^{\kappa},N^\kappa}\)):
	\begin{equation}
		\label{eq:C3}\tag{\(\sfC^{\sf reg}_{3,N}\)}
		\CPts(\gamma) \cap \calB_{N^{\iota},2h_N} (\ell_N) \neq \emptyset
		\text{ and } \CPts(\gamma) \cap \calB_{N^{\iota},2h_N} (r_N) \neq \emptyset .
	\end{equation}
\end{definition}
\begin{figure}
	\centering
	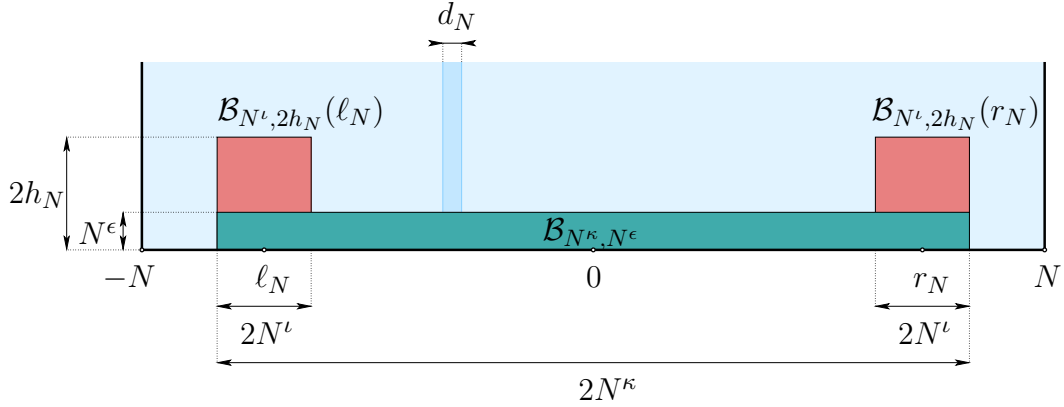
	\caption{Sketch (not-to-scale) of the construction in the definition of regular interfaces: the latter must have at least one cone-point in each of the semi-infinite vertical strip \(\calS_{\ell,\ell+d_N}\) of width \(d_N\) (a piece of one of them is drawn in light blue) and in each of the boxes \(\calB_{N^{\iota},2h_N} (\ell_N)\) and \(\calB_{N^{\iota},2 h_N} (r_N)\) (in light red) and must avoid the box \(\calB_{N^\kappa,N^\epsilon}\) (in teal).}
	\label{fig:regularInterfaces}
\end{figure}%

\medskip
Here is our main reduction statement:
\begin{theorem}
	\label{thm:ERW-g}
	For any \(\lambda>0\), there exists a choice of \(c_+,d_+\) and of scales \(\kappa,\iota,\epsilon\)
	satisfying~\eqref{eq:kappa-scales-g} such that
	\begin{equation}\label{eq:ERW}
		\lim_{N\to\infty}
		\mu_{\BOXN;\beta,\lambda /N}^{\pm} \bigl( \gamma\not\in\calC_N^{\sf reg} (c_+,d_+,\kappa,\iota,\epsilon) \bigr)
		= 0 .
	\end{equation}
\end{theorem}

\subsection{Proof of Theorem~\ref{thm:ERW-g}}

Property~\eqref{eq:C0} is already established (see Proposition~\ref{prop:C-short} and the paragraph
before~\eqref{eq:IRD1}).

The same is true regarding Property~\eqref{eq:C1}: this is the content of the entropic repulsion estimate in Proposition~\ref{prop:ER}.

Property~\eqref{eq:C2} is a simple consequence of the following upper and lower  bounds on the area \(\abs{\BOX_N^-[\gamma]}\), which we shall prove in Subsection~\ref{sub:UB-Area-g}:
\begin{lemma}\label{lem:UB-Area-g}
	There exists \(C = C_\beta \) such that
	\begin{equation}\label{eq:g-lb}
	\sum_{\gamma\in\calC_N} \IF{\nabs{\BOX_N^-[\gamma]} \leq C_\beta N^{4/3}} q_{N,\beta}(\gamma)
	\geq
	\mathrm{e}^{-\lambda m_\beta^* C_\beta N^{1/3} - 2\taub N} .
	\end{equation}
	In particular,
	\begin{equation}\label{eq:UB-Area-g}
	\mu_{\BOXN;\beta,\lambda/N}^{\pm} \bigl( \abs{\BOX_N^-[\gamma]} >
	C  N^{4/3} \bigr) \leq \mathrm{e}^{-\lambda m_\beta^* C N^{1/3}} .
	\end{equation}
	for any \(C > C_\beta\).
\end{lemma}
Indeed, let us write \(\calE \defby \bigcup_\ell \bsetof{\gamma}{\CPts(\gamma)\cap\calS_{\ell,\ell+d\sqrt[3]{N}} = \emptyset}\), where the union is over \(\ell\in \{-N,\dots,N-d\sqrt[3]{N}\}\).
It then follows from Lemma~\ref{lem:UB-Area-g} that
\begin{align*}
	\mu_{\BOXN;\beta,\lambda /N}^{\pm} (\calE)
	&\leq
	\mu_{\BOXN;\beta,\lambda /N}^{\pm} (\abs{\BOX_N^-[\gamma]} > 4C_\beta N^{4/3}) +
	\mu_{\BOXN;\beta,\lambda /N}^{\pm} (\calE \given \abs{\BOX_N^-[\gamma]} \leq
	4C_\beta  N^{4/3}) \\
	&\leq
	\mathrm{e}^{-\lambda m_\beta^* C_\beta N^{1/3}} +
	\mu_{\BOXN;\beta,\lambda /N}^{\pm} (\calE \given \abs{\BOX_N^-[\gamma]} \leq 4C_\beta  N^{4/3})	.
\end{align*}
We can bound the second term in the right-hand side using~\eqref{eq:Comp-BN-bound}, which can be recorded as follows:
\begin{equation}\label{eq:Comp-BN-bound-q}
	\mu_{\BOXN;\beta,\lambda/N}^{\pm} (\gamma)\, \bigl(1+\smo{1}\bigr)
	\propto
	\exp\Bigl[ - 2 \frac{\lambda m^*_\beta }{N} \abs{\BOX_N^-[\gamma]} + \frac{\lambda \psi_N(\gamma)}{N} \Bigr] \; q_{N,\beta}(\gamma).
\end{equation}
We thus have
\begin{multline*}
	\mu_{\BOXN;\beta,\lambda/N}^{\pm} (\calE \given \abs{\BOX_N^-[\gamma]} \leq 4C_\beta  N^{4/3})\\
	\leq
	\frac{
		\sum_{\gamma\in\calC_N} \IF{\calE,\nabs{\BOX_N^-[\gamma]} \leq 4C_\beta  N^{4/3}}
		\exp\Bigl[ - 2 \frac{\lambda m^*_\beta }{N} \abs{\BOX_N^-[\gamma]} + \frac{\lambda \psi_N(\gamma)}{N} \Bigr] \; q_{N,\beta}(\gamma)
	}
	{
		\sum_{\gamma\in\calC_N} \IF{\nabs{\BOX_N^-[\gamma]} \leq 4C_\beta  N^{4/3}}
		\exp\Bigl[ - 2 \frac{\lambda m^*_\beta }{N} \abs{\BOX_N^-[\gamma]} + \frac{\lambda \psi_N(\gamma)}{N} \Bigr] \; q_{N,\beta}(\gamma)
	} .
\end{multline*}
Thanks to~\eqref{eq:psi-N-bound} and~\eqref{eq:C0}, the exponential term in the numerator can be bounded by a constant, while the corresponding term in the denominator can be bounded below by \(\exp(-c N^{1/3})\) for some constant \(c=c(\beta,\lambda)>0\).
Therefore,
\[
	\mu_{\BOXN;\beta,\lambda/N}^{\pm} (\calE \given \abs{\BOX_N^-[\gamma]} \leq 4C_\beta  N^{4/3})
	\leq
	\mathrm{e}^{c N^{1/3}} \mu_{\BOXN;\beta,0}^{\pm} (\calE \given \abs{\BOX_N^-[\gamma]} \leq 4C_\beta  N^{4/3}) .
\]
It thus follows from the refinement~\eqref{eq:massgap-mep} of~\eqref{eq:massgap-mu}, as
formulated in Proposition~\ref{prop:ref-Potts-mg}, and from	the bounds~\eqref{eq:Prop4} and~\eqref{eq:g-lb} that there exists \(\phi=\phi_\beta>0\) such that, for any \(d\in\bbN\) sufficiently large,
\[
	\mu_{\BOXN;\beta,\lambda /N}^{\pm} (\calE)
	\leq
	\mathrm{e}^{- \phi_\beta \sqrt[3]{N}} ,
\]
which is~\eqref{eq:C2}.

\subsubsection{Proof of Lemma~\ref{lem:UB-Area-g}}
\label{sub:UB-Area-g}

We start again with~\eqref{eq:Comp-BN-bound}:
\[
	\mu_{\BOXN;\beta,\lambda/N}^{\pm} (\gamma)\, \bigl(1+\smo{1}\bigr)
	\propto
	\exp\Bigl[ - 2 \frac{\lambda m^*_\beta }{N} \abs{\BOX_N^-[\gamma]} + \frac{\lambda \psi_N(\gamma)}{N} \Bigr] \; q_{N,\beta}(\gamma).
\]
By~\eqref{eq:C0}, we may restrict attention to \(\abs{\gamma} \leq CN\).
Hence the term \(\frac{\psi_N(\gamma)}{N} = \bgo{1}\).
Therefore, by virtue of~\eqref{eq:q-mu-form}, the upper bound~\eqref{eq:UB-Area-g} is indeed a consequence of the lower bound~\eqref{eq:g-lb}.

Now, the lower bound~\eqref{eq:g-lb} is a rather standard assertion of the Ornstein--Zernike theory: Consider
\[
	\sfu_N \defby (-N + \lfloor 4\sqrt[3]{N} \rfloor, \lfloor 2\sqrt[3]{N} \rfloor)
	\text{ and }
	\sfv_N \defby (N - \lfloor 4\sqrt[3]{N} \rfloor, \lfloor 2\sqrt[3]{N} \rfloor) .
\]
By construction, \(\lc{N} \in \sfu_N + \bcone\) and \(\rc{N} \in \sfv_N + \fcone\).
On the one hand, it is not difficult to check that there exists \(\kappa=\kappa_\beta\) such that
\begin{multline}\label{eq:cl-rl}
	\sum_{\gamma_L:\, \lc{N} \mapsto \sfu_N} q_{N,\beta} (\gamma_L)
	\IF{\gamma_L\subset \sfu_N + \bcone}
	\geq
	\mathrm{e}^{-\kappa_\beta \sqrt[3]{N}}\\
	\text{ and }
	\sum_{\gamma_R:\, \sfv_N \mapsto \rc{N}} q_{N,\beta} (\gamma_R)
	\IF{\gamma_R\subset \sfv_N + \fcone} \geq \mathrm{e}^{-\kappa_\beta \sqrt[3]{N} } .
\end{multline}
On the other hand, since \(q_{N,\beta} (\gamma') \geq q_\beta(\gamma')\) by property~\ref{property:MonotVolWeight}, it follows from Lemma~\ref{lem:tube} that there exists \(\kappa_\beta\) such that, for all \(N\) large enough,
\begin{equation}\label{eq:uN-vN-part}
	\sum_{\gamma'\in\SetRootDiaCont(\sfu_N,\sfv_N)} q_{N,\beta} (\gamma')
	\IF{\gamma'\subset\bbZ\times [\sqrt[3]{N}, 3\sqrt[3]{N}]}
	\geq
	\mathrm{e}^{-\kappa_\beta \sqrt[3]{N} -2\taub N} .
\end{equation}
By property~\ref{property:Concat},
\[
	q_{N,\beta} (\gamma_L\circ\gamma'\circ\gamma_R) \geq
	q_{N,\beta} (\gamma_L)
	q_{N,\beta} (\gamma')
	q_{N,\beta} (\gamma_R),
\]
uniformly in all triples \((\gamma_L, \gamma', \gamma_R)\) contributing to~\eqref{eq:cl-rl} and~\eqref{eq:uN-vN-part}.
Our claim~\eqref{eq:g-lb} follows.\qed
\begin{figure}
	\centering
	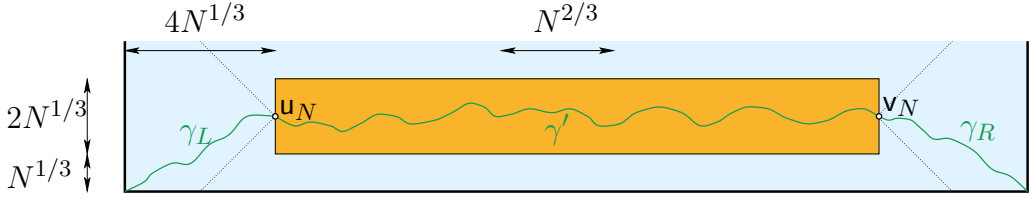
	\caption{Sketch (not-to-scale) of the construction in the proof of Lemma~\ref{lem:UB-Area-g}.
		The lower bound is obtained by restricting to paths \(\gamma_L\circ\gamma'\circ\gamma_R\) with cone-points at the vertices \(\sfu_N\) and \(\sfv_N\), with \(\gamma'\) staying inside the yellow tube.}
	\label{fig:Lemma5_5}
\end{figure}%

\subsubsection{Property~\eqref{eq:C3}}

Recall~\eqref{eq:kappa-scales-g} and~\eqref{eq:dlrh-N-g}.
In view of the symmetry, it is sufficient to derive an upper bound on
\begin{equation}\label{eq:left-box-g}
	\mu_{\BOXN;\beta,\lambda/N}^{\pm} (\CPts(\gamma) \cap \calB_{N^{\iota},2h_N} (\ell_N) = \emptyset) .
\end{equation}
By the already established reduction to~\eqref{eq:C2}, we may assume that \(\CPts(\gamma)\) contains a vertex in any semi-strip \(\calS_{\ell,\ell+d_N}\) of width \(d_N\).

We proceed with a disjoint partition of the non-intersection event in question.
If \(\CPts(\gamma) \cap \calB_{N^{\iota},2h_N} (\ell_N) = \emptyset\), then let \(\sfu\) be the first vertex of \(\CPts(\gamma)\) left of the strip \(\calS_{-N^\kappa,-N^\kappa+2N^\iota}\) that sits below the height \(2h_N\).
Similarly, let \(\sfv\) be the first vertex of \(\CPts(\gamma)\) right of the strip \(\calS_{-N^\kappa,-N^\kappa+2N^\iota}\) that sits below the height \(2h_N\).
We use \(\calE_N(\sfu,\sfv)\) to denote the corresponding event (see Figure~\ref{fig:Creg3}) and we write \(\gamma = \gamma_L \circ \gamma^\prime  \circ \gamma_R\), where \(\gamma^\prime \defby \gamma \cap \calD(\sfu,\sfv) \in \SetRootDiaCont(\sfu,\sfv)\) is the portion of \(\gamma\) between \(\sfu\) and \(\sfv\).
\begin{figure}
	\centering
	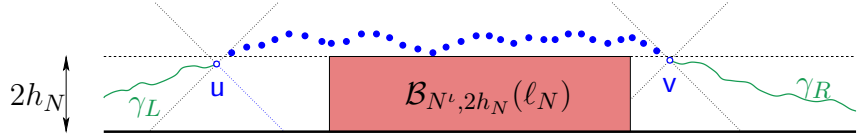
	\caption{Sketch (not-to-scale) of the construction in the definition of the event \(\calE_N(\sfu,\sfv)\).
	All cone-points of the path located between \(\sfu\) and \(\sfv\) (blue dots) have to stay above the level \(2h_N\), while the cone-points \(\sfu\) and \(\sfv\) have to be below this level.}
	\label{fig:Creg3}
\end{figure}%
In this way,
\begin{equation}\label{eq:uv-above-sum-g}
	\mu_{\BOXN;\beta,\lambda/N}^{\pm} \bigl( \CPts(\gamma) \cap \calB_{N^{\iota},2h_N}(\ell_N) = \emptyset \bigr)
	=
	\sum_{\sfu,\sfv} \mu_{\BOXN;\beta,\lambda/N}^{\pm} \bigl( \calE_N(\sfu,\sfv) \bigr) .
\end{equation}
By construction, there exist \(\ell,r\in\bbN\) such that
\begin{align*}
	\sfu\in\calB_L(\ell) &\defby \calB_{d_N,2h_N}(-N^\kappa - (2\ell-1) d_N)\\
	\shortintertext{ and }
	\sfv\in\calB_R ( r ) &\defby \calB_{d_N,2h_N}(-N^\kappa + 2N^\iota + (2r-1) d_N) .
\end{align*}
We claim that there exists \(\kappa_\beta < \infty\), such that, for any \(\ell,r\in\bbN\),
\begin{multline}\label{eq:lr-bound-box-g}
	\maxtwo{\sfu\in\calB_L(\ell)}{\sfv\in\calB_R(r)} \mu_{\BOXN;\beta,\lambda/N}^{\pm} \bigl( \calE_N(\sfu,\sfv) \bigr) \\[-1em]
	\leq
	\exp\Bigl\{ - 2 \lambda m^*_\beta c_+ \frac{N^\iota+(\ell+r)d_N}{N^{2/3}} +
	\kappa_\beta \frac{N^\iota+(\ell+r)d_N}{ N^{2/3}} \Bigr\} .
\end{multline}
A substitution of~\eqref{eq:kappa-scales-g} and~\eqref{eq:dlrh-N-g}, with \(c_+\) chosen to be sufficiently large,  into~\eqref{eq:lr-bound-box-g} implies the result:
\[
	\mu_{\BOXN;\beta,\lambda/N}^{\pm}\bigl(
		\CPts (\gamma) \cap \calB_{N^{\iota},2h_N} (\ell_N) = \emptyset
	\bigr)
	=
	\smo{1} .
\]
It remains to explain \eqref{eq:lr-bound-box-g}.

Let us start by formulating an approximate domain-decoupling property of \(\psi_N\) in~\eqref{eq:psi-N}.
We shall give a general formulation and then apply it for \(\gamma = \gamma_L\circ \gamma^\prime \circ \gamma_R \in \calE_N(\sfu,\sfv)\).
Consider \(\eta  = \eta_L\circ \eta^\prime  \circ \eta_R\) with \(\eta_L\in\SetRootMarkBackCont\), \(\eta^\prime  \in \SetRootDiaCont\) and \(\eta_R\in\SetRootMarkForwCont\).
Set \(\sfw = \fend(\eta^\prime )\) and \(\sfz  = \bend (\eta^\prime)\).
Fix any path \(\eta_0^\prime \in\SetRootDiaCont(\sfw,\sfz)\) and define \(\eta_0 = \eta_L \circ \eta_0^\prime \circ \eta_R\).
Then,
\begin{align}\label{eq:psi-N-sum}
	\psi_N (\eta )
	= \bgo{\abs{\eta^\prime }}
	&+ \sum_{i\in\BOX_N^+[\eta]\setminus\calS_{\sfw,\sfz}}
	\Bigl\{
		\mu_{\BOXN;\beta,0}^{\pm} (\sigma_i \given \eta_0) -
		\mu_{\BOXN;\beta,0}^{+} (\sigma_i)
	\Bigr\}\notag\\
	&+ \sum_{i\in\BOX_N^-[\eta]\setminus\calS_{\sfw,\sfz}}
	\Bigl\{
		\mu_{\BOXN;\beta,0}^{\pm} (\sigma_i \given \eta_0 ) -
		\mu_{\BOXN;\beta,0}^{-} (\sigma_i)
	\Bigr\} .
\end{align}
\eqref{eq:psi-N-sum} is just an expression of the exponential spatial-relaxation properties of finite-volume expectations for the Ising model at any fixed sub-critical temperature.

Let \(\sfA (\gammap) \defby \abs{\Delta^-_N[\gamma'] \cap \calS_{\sfu,\sfv}}\) (remember~\eqref{eq:DeltaMinus}) be the area between \(\gammap\) and \(\calL\) inside the semi-strip \(\calS_{\sfu,\sfv}\).
In view of~\eqref{eq:psi-N-sum} and~\eqref{eq:C0}, we have (remember the definition of conditional weights in~\ref{property:CondWeights})
\begin{equation}\label{eq:g-cond}
	\mu_{\BOXN;\beta,\lambda /N}^{\pm} \bigl( \gammap \bgiven \gamma_L,\gamma_R \bigr)
	\leq
 	\frac{\mathrm{e}^{C_1}}{\calZ(\gamma_L,\gamma_R)}\, q_{N,\beta}\bigl(\gammap \bgiven \gamma_L,\gamma_R \bigr) \exp\Bigl\{ -\frac{2\lambda m_\beta^*}{N} \sfA(\gammap)\Bigr\} ,
\end{equation}
where \(\calZ (\gamma_L,\gamma_R)\) is a provisional notation for the following partition function
\begin{equation}\label{eq:cal-Z}
	\calZ(\gamma_L,\gamma_R) \defby \sum_{\etap\in\SetRootDiaCont(\sfu,\sfv)} q_{N,\beta} \bigl( \etap \bgiven \gamma_L,\gamma_R\bigr) \exp\Bigl\{ -\frac{2\lambda m_\beta^*}{N} \sfA(\etap)\Bigr\} .
\end{equation}
Recall~\eqref{eq:dlrh-N-g}.
By~\eqref{eq:C2}, we may assume that, for any \(\ell\), \(\CPts(\gammap)\cap\calS_{\ell,\ell+d_N} \neq \emptyset\). This means that if we take \(d_+ \leq c_+/3\), then
\begin{equation}\label{eq:A-g}
	\sfA(\gammap) \geq (2h_N - 2d_N) \bigl( 2 N^\iota + (\ell-1)d_N + (r-1) d_N \bigr)
	\geq
	h_N \bigl( N^\iota + (\ell+r) d_N \bigr) .
\end{equation}
Moreover, by the definition of conditional weights and the GKS inequality,
\begin{equation}\label{eq:sum-g}
	\sum_{\gamma'\in\SetRootDiaCont(\sfu,\sfv)} q_{N,\beta} ( \gamma' \given \gamma_L,\gamma_R )
	\leq
	\mathrm{e}^{-\taub(\sfv-\sfu)} .
\end{equation}
Note that both estimates above are uniform in \(\gamma_L,\gamma_R\).

Let us now derive matching uniform lower bounds on \(\calZ(\gamma_L,\gamma_R)\).
We proceed in the spirit of the proof of Lemma~\ref{lem:UB-Area-g} and rely on the Ornstein--Zernike theory as developed in~\cite{Campanino-Ioffe-Velenik_2003,greenberg2005invariance}.
Consider
\[
	\sfw = \lfloor (-N^\kappa - 2 \ell d_N + N^{\frac23},2\sqrt[3]{N}) \rfloor
	\text{ and }
	\sfz = \lfloor (-N^\kappa + 2 N^\iota + 2 r d_N - N^{\frac23},2\sqrt[3]{N}) \rfloor .
\]
Let \(\calG_N (\sfu,\sfv) \subset \SetRootDiaCont(\sfu,\sfv)\) be the set of diamond-confined paths \(\eta = \eta^{\sfu\sfw} \circ \eta^{\sfw\sfz} \circ \eta^{\sfz\sfv}\)
such that the sub-paths \(\eta^{\sfu\sfw} \in \SetRootDiaCont(\sfu,\sfw)\), \(\eta^{\sfw\sfz}\in\SetRootDiaCont(\sfw,\sfz)\) and \(\eta^{\sfz\sfv}\in\SetRootDiaCont(\sfz,\sfv)\) each stay inside the strip of width \(\sqrt[3]{N}\) around the corresponding segments \([\sfu,\sfw]\), \([\sfw,\sfz]\) and \([\sfz,\sfv]\) (see Fig.~\ref{fig:5_46_LB}).

\begin{figure}
	\centering
	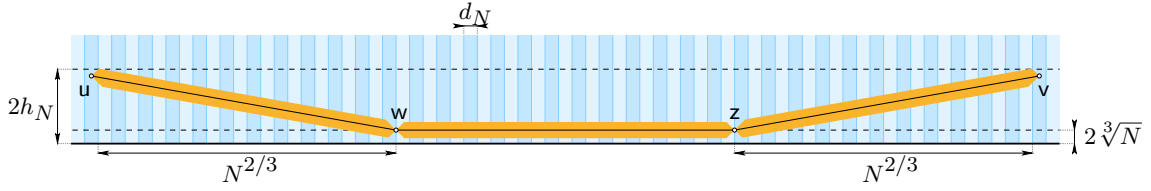
	\caption{Sketch (not-to-scale) of the construction of the set \(\calG_N (\sfu,\sfv)\). We only consider paths having cone-points at \(\sfu\), \(\sfw\), \(\sfz\) and \(\sfv\) and staying inside the intersection of the relevant cones and tubes of width \(\sqrt[3]{N}\) centered on the segments \(\sfu\sfw\), \(\sfw\sfz\) and \(\sfz\sfv\).}
	\label{fig:5_46_LB}
\end{figure}%

Then,
\begin{equation}\label{eq:A-bound-eta}
\sfA (\eta^{\sfu\sfw}), \sfA (\eta^{\sfz\sfv}) \leq 3h_N N^{\frac23}
\text{ and }
\sfA (\eta^{\sfw\sfz})\leq 3\sqrt[3]{N} \bigl( 2 N^\iota + \max\{ 2 (\ell+r) d_N - 2N^{\frac23}, 0\} \bigr) .
\end{equation}
By properties~\ref{property:CondWeights} and~\ref{property:MonotVolWeight}, \(q_{N,\beta} ( \gamma' \given \gamma_L,\gamma_R ) \geq q_\beta(\gamma')\).
We can thus deduce from Lemma~\ref{lem:tube} that
\begin{equation}\label{eq:GN-lb}
	\sum_{\gamma'\in\calG_N(\sfu,\sfv)} q_{N,\beta} ( \gamma' \given \gamma_L,\gamma_R )
	\geq
	\mathrm{e}^{-\taub(\sfv-\sfu) - CN^{\iota-2/3}} .
\end{equation}
Hence, there exists \(\kappa_\beta <\infty\) such that
\begin{equation}\label{eq:LowerBndCReg3}
	\calZ(\gamma_L,\gamma_R)
	\geq
	\sum_{\gamma'\in\calG_N(\sfu,\sfv)} \mathrm{e}^{2(\lambda/N) m^*_\beta \sfA(\gamma')}
	q_{N,\beta} ( \gamma' \given \gamma_L,\gamma_R )
	\geq
	\mathrm{e}^{-\taub(\sfv-\sfu) - \kappa_\beta \frac{ N^\iota + (\ell+r) d_N}{N^{2/3}}} .
\end{equation}
Combined with~\eqref{eq:A-g} and~\eqref{eq:sum-g}, this implies~\eqref{eq:lr-bound-box-g}.
\qed

\subsection{Factorization of the irreducible decomposition}\label{sec:FactorizationIrred}

By convention, paths contain at least two vertices.
Consider~\eqref{eq:IRD1} or, more generally, paths \(\eta\subset\BOXs_{N}\) admitting the irreducible decomposition
\begin{equation}\label{eq:IRD1-gen}
	\eta = \eta_L \circ \eta_1 \circ \dots \circ \eta_k \circ \eta_R .
\end{equation}
We neither exclude \(k=0\), nor the possibility that \(\eta_L\) or \(\eta_R\) be empty.
The weights
\[
	q_{N,\beta} (\eta_L \circ \eta_1 \circ \dots \circ \eta_k \circ \eta_R)
\]
make perfect sense in the context of dual high-temperature models (see~\cite{Pfister-Velenik_1999} for an explicit description).
The fluctuation theory developed in~\cite{Campanino-Ioffe-Velenik_2003} is based on the interpretation of the infinite-volume weights \(q_\beta (\eta_L \circ \eta_1 \circ \dots \circ \eta_k \circ \eta_R)\) in terms of the action of Ruelle operators for full shifts of countable type.
An adjustment to finite volumes gives rise to a representation of interfaces in terms of effective random walks with exponentially mixing steps.
Such a representation becomes rather complex when combined with a tilt by magnetic fields as we consider here.
A representation in terms of usual random walks with independent steps would be technically much more convenient.

Let us formulate the corresponding result.
The price of factorization is irreducibility. Instead of~\eqref{eq:IRD1-gen}, consider
\begin{equation}\label{eq:D1}
	\eta = \omega_L \circ \omega_1 \circ \dots \circ \omega_M \circ \omega_R \defby \uomega ,
\end{equation}
where \(\omega_L\in\SetRootMarkBackCont\), \(\omega_R\in \SetRootMarkForwCont\) and \(\omega_1,\dots,\omega_M\in\SetRootDiaCont\) are not required to be irreducible.
In particular, the representation~\eqref{eq:D1} is not uniquely defined.
In fact, there are exponentially many (in the number \(k\) of irreducible pieces in~\eqref{eq:IRD1-gen})  different decompositions~\eqref{eq:D1} of a particular path \(\eta\) to be taken into account (see Fig.~\ref{fig:5_55}).

\begin{figure}[ht]
	\centering
	\includegraphics{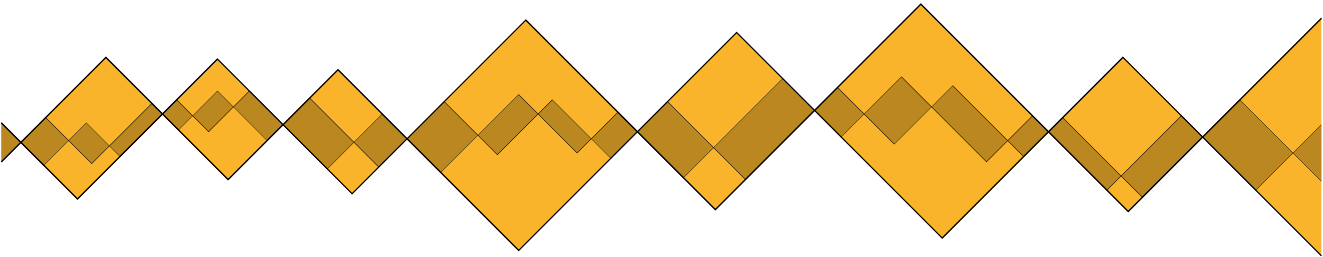}
	\caption{In yellow, the sequence of diamonds corresponding to one possible decomposition \(\omega_L \circ \omega_1 \circ \dots \circ \omega_M \circ \omega_R\). The diamonds corresponding to the original (irreducible) decomposition are drawn in a darker shade.}
	\label{fig:5_55}
\end{figure}%

Our main objective is to develop a factorization of the weights \(q_{N,\beta}(\gamma)\) associated to regular interfaces \(\gamma \in \calC_N^{\mathsf reg}\), but the whole theory applies in the full generality of concatenated paths \(\eta\) as suggested by~\eqref{eq:IRD1-gen} and, in fact, the factorization in question will be formulated in this framework in order to streamline the way it is used.

\begin{definition}
	\label{def:gam-chi}
	We shall record~\eqref{eq:D1} as a compatibility property \(\uomega\sim\eta\).
	Note that \(\uomega\sim \gamma\) implies that each of \(\omega_L,\omega_1,\dots,\omega_M,\omega_R\) are concatenations of irreducible paths appearing on the right-hand side of the irreducible decomposition~\eqref{eq:IRD1}.
	For any path \(\omega\) (for instance for \(\omega = \omega_L\)
	or \(\omega = \omega_i\) in~\eqref{eq:D1}), we shall use \(\frc(\omega) \defby \abs{\CPts(\omega)}\) to denote the number of cone-points of \(\omega\).
\end{definition}

In order to formulate the factorization result, we need one more notation:
By its very nature, \emph{any} path \(\eta = (\sfu_1,\dots,\sfu_\ell)\) always has the initial point \(\fend(\eta)\defby\sfu_1\) and the terminal point \(\bend(\eta)\defby\sfu_\ell\).
This notation is compatible with the one introduced for various confined paths in Definition~\ref{def:path-types}.
\begin{definition}
	\label{def:Displacement}
	Given a path \(\eta\), \(\sfX(\eta) \defby \bend(\eta) - \fend(\eta)\) denotes the displacement along \(\eta\).
	In the coordinate representation, we shall write \(\sfX(\eta) = (\theta(\eta), \zeta(\eta))\), where \(\theta(\eta) \defby \sfX(\eta)\cdot\sfe_1\) and \(\zeta(\eta) \defby \sfX(\eta)\cdot\sfe_2\).
\end{definition}

\begin{figure}[ht]
	\centering
	\includegraphics{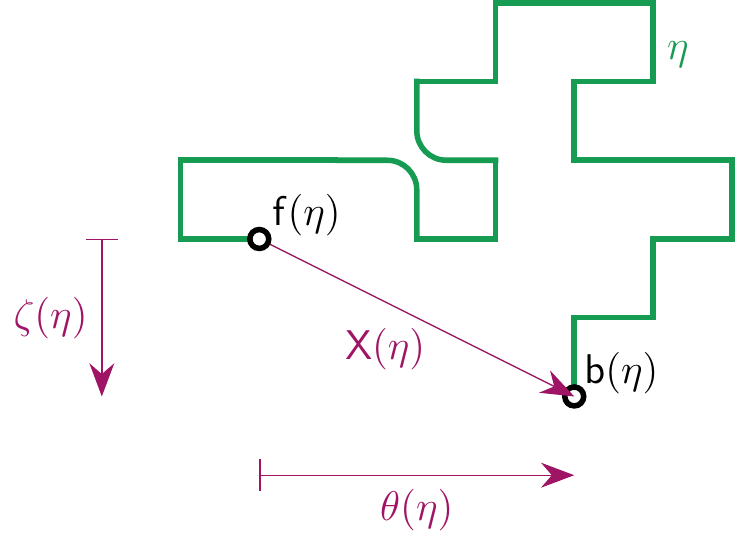}
	\caption{Notations for the displacement along a path \(\eta\).}
	\label{fig:displacement}
\end{figure}

Given non-negative weights \(\rho_{N,\beta}^L\) on \(\SetRootMarkBackCont\), \(\rho_{N,\beta}^R\) on \(\SetRootMarkForwCont\) and \(\sfp_{N,\beta}\) on \(\SetRootDiaCont\), consider the induced weights \(\WNbeta\) on the set of concatenations \(\uomega\) in~\eqref{eq:D1}.
\begin{equation}\label{eq:WN-beta}
	\WNbeta (\uomega) \defby \rho_{N,\beta}^L(\omega_L) \rho_{N,\beta}^R(\omega_R) \prod_{i=1}^M \sfp_{N,\beta}(\omega_i).
\end{equation}
We are ready to formulate our factorization result:
\begin{theorem}\label{thm:fact}
	There exist a number \(\nu_\beta >0\) and, for all \(N\) sufficiently large, weights \(\rho_{N,\beta}^L, \rho_{N,\beta}^R\) and \(\sfp_{N,\beta}\) such that:
	\begin{enumerate}[label=\(\sfW_\arabic*\)., ref=\(\sfW_\arabic*\)]
	\item\label{property:W1} For any \(\eta\subset\BOXs_{N}\) admitting an irreducible decomposition as in~\eqref{eq:IRD1-gen},
	\begin{equation}\label{eq:fact}
		\mathrm{e}^{\taub \theta(\eta)} q_{N,\beta}(\eta)
		=
		\sum_{\uomega\sim\eta} \WNbeta(\uomega) .
	\end{equation}
	In particular, \( \WNbeta(\eta ) = \mathrm{e}^{\taub \theta(\eta)} q_{N,\beta}(\eta)\) if \(\eta\) is irreducible.
	\item\label{property:W2} Recall the notation \(\frc(\omega)\) for the number of cone-points of \(\omega\). For any \(\omega\in\SetRootDiaCont\),
	\begin{equation}\label{eq:W-bound-cp}
		\sfp_{N,\beta}(\omega)
		\leq
		\mathrm{e}^{-\nu_\beta \frc(\omega)} \mathrm{e}^{\taub \theta(\omega)} q_{N,\beta}(\omega).
	\end{equation}
	The same holds for the weights \(\rho_{N,\beta}^L\) of  \(\omega\in \SetRootMarkBackCont\) and the weights \(\rho_{N,\beta}^R\) of  \(\omega\in \SetRootMarkForwCont\).
	\item\label{property:W3} There exist \(\bbZ^2\)-shift-invariant (infinite-volume) weights \(\sfp_\beta\) such that the following bulk relaxation property of the weights \(\sfp_{N,\beta}\) holds:
	For \(\omega \subset \BOXs_N\),
	\begin{equation}\label{eq:pnb-pb}
		\Babs{ \frac{\sfp_{N,\beta}(\omega)}{\sfp_\beta(\omega)} - 1 }
		\leq
		\exp\Bigl\{ \sum_{t\in\omega} \mathrm{e}^{-\nu_\beta \dd_2 (t,\partial\BOXs_N)} \Bigr\} ,
	\end{equation}
	where \(\dd_2(t,\partial\BOXs_N) \defby \min_{s\not\in\BOXs_N} \normII{t-s}\).
	Moreover, the weights \(\{\sfp_\beta\}\) are probabilistic in the following sense: For any \(\sfu\),
	\begin{equation}\label{eq:pbet-weights}
		\sum_{\sfw} \sum_{\omega\in\SetRootDiaCont(\sfu,\sfw)} \sfp_\beta(\omega) = 1 .
	\end{equation}
	\end{enumerate}
\end{theorem}
For maximal effect, the factorization properties \ref{property:W1}--\ref{property:W3} above should be combined with properties~\ref{property:SupportedInVol}--\,\ref{property:Mixing} of the path weights \(q_{N,\beta}\).

The proof of Theorem~\ref{thm:fact} is relegated to the Appendix~\ref{app:renewal_perco}.
In particular, the weights \(\{\sfp_{N,\beta}\}\) are defined in~\eqref{eq:pnbeta} there.
It boils down to a rather simple percolation argument based on a form of high-temperature expansion for one-dimensional systems with exponentially decaying interactions.
This expansion relies on the domain-monotonicity property~\ref{property:MonotVolWeight} of contour weights.

\begin{remark}\label{rem:ZeroExp}
We shall use several times the observation that the expectation of \(\zeta(\eta)\) (remember Definition~\ref{def:Displacement}) under \(\sfp_\beta\) necessarily vanishes. Indeed, were it not the case, a simple large deviations bound relying on~\eqref{eq:WN-beta}, \eqref{eq:fact} and~\eqref{eq:W-bound-cp} would imply that \(\mathrm{e}^{2\taub N}\sum_{\eta\in\calC_N} q_{N,\beta}(\eta)\) decays exponentially fast in \(N\), which would contradict~\eqref{eq:ratio-q-mu}.
\end{remark}

\subsubsection{Interfaces in low-temperature Potts models}
\label{sub:ltPotts}

A more robust (but slightly more involved) version of Theorem~\ref{thm:fact}, which does not require domain monotonicity and applies, for instance, to a wide range of random-cluster models, is presented in~\cite{Ott-Velenik_2017} and was very recently used in~\cite{Ioffe+Ott+Velenik+Wachtel-20} to study scaling limits of Potts/FK interfaces above a wall.
We shall rely on these results in the following way: The zero-field measure \(\mu_{\BOXN;\beta,0}^{\pm}\) was analyzed in~\cite{Ioffe+Ott+Velenik+Wachtel-20} on the more general level of nearest-neighbor Potts models below the critical temperature, in their random-cluster representation.
As is the case here, the crucial role was played by the irreducible decomposition of the corresponding percolation clusters.
In particular, the notion of cone-points was well defined.
In fact, under the Edwards--Sokal coupling, all the cone-points of the underlying random-cluster component associated to \(\gamma\) in~\cite{Ioffe+Ott+Velenik+Wachtel-20} belong to the set \(\CPts(\gamma)\).
We shall therefore routinely import from~\cite{Ioffe+Ott+Velenik+Wachtel-20} various upper bounds on the probability of atypically small densities of cone-points or atypically large sizes of irreducible components in~\eqref{eq:IRD1}.
For instance, Proposition~\ref{prop:ref-Potts-mg} is an immediate consequence of the computations employed in the proof of Lemma~2.2 in~\cite{Ioffe+Ott+Velenik+Wachtel-20}.

\subsection{Random-walk representation and effective reduction of the interface}\label{sec:RWRepr}

Theorem~\ref{thm:fact} paves the way for a probabilistic analysis in terms of random walks under linear area tilts.
Suppose we have a collection \( \omega_1, \dots , \omega_k\) of paths, satisfying the property

\[
	\bend (\omega_1) = \fend (\omega_2), \dots , \bend(\omega_{\ell-1}) = \fend(\omega_\ell).
\]
Then we can form their concatenation  \(\uomega = \omega_1 \circ \dots \circ \omega_k\) in an obvious way. In the following, we will naturally restrict ourselves to the admissible concatenations, meaning that the path $\uomega$ satisfies all the restrictions above.
\begin{definition}\label{def:S-walk}
	Given a concatenation \(\uomega = \omega_1 \circ \dots \circ \omega_\ell\), define \(\sfS_0 \defby \fend(\omega_1)\), \(\sfS_1 \defby \fend(\omega_2)\), \dots , \(\sfS_\ell \defby \bend(\omega_\ell)\)
	and set
	\begin{equation}\label{eq:S-walk}
		\sfS (\uomega) \defby \{\sfS_0, \sfS_1, \dots, \sfS_\ell\} .
	\end{equation}
	We shall think about \(\sfS(\uomega)\) in terms of an effective random walk through the vertices of \(\uomega\) with steps \(\sfX_i \defby \sfS_i - \sfS_{i-1} = \sfX(\omega_i)\).
\end{definition}
In particular, if \(\eta\) admits the irreducible representation~\eqref{eq:IRD1-gen}, then the induced trajectories of the effective random walk \(\sfS(\uomega)\) correspond to various compatible concatenations
\(\uomega\sim\eta\) with weights \(\WNbeta\) defined in~\eqref{eq:WN-beta}.

\medskip
Let \(\underline{\Omega}_N\) be the set of all compatible concatenations \(\uomega\sim \gamma\) of the interface \(\gamma\) with at least one cone-point.
\begin{definition}
	We define two probability measures \(\PNbeta\) and \(\PNbetal\) on \(\uOmega_N\) by
	\begin{equation}\label{eq:PNb-PNbl}
		\PNbeta(\uomega) \propto \WNbeta(\uomega)
		\quad\text{and}\quad
	\PNbetal(\uomega) \propto \exp\Bigl[ - 2 \frac{\lambda m^*_\beta}{N} \abs{\BOX_N^-[\uomega]} + \frac{\lambda \psi_N(\uomega)}{N} \Bigr] \PNbeta(\uomega) ,
	\end{equation}
	where the weights \(\WNbeta\) were defined in~\eqref{eq:WN-beta} and, following~\eqref{eq:Comp-BN-bound},
	\[
		\BOX_N^-[\uomega] \defby \BOX_N^-[\gamma]
		\quad\text{ and }\quad
		\psi_N(\uomega) \defby \psi_N(\gamma)
	\]
	if \(\uomega \sim \gamma\).
	\end{definition}

Recall the Definition~\ref{def:Creg} of the set \(\calC_N^{\sf reg}\) of regular interfaces.
Define
\begin{equation}\label{eq:Om-reg}
	\uOmega_N^{\sf reg} = \uOmega_N^{\sf reg} (c_+, d_+, \kappa, \iota, \epsilon) \defby \bigcup_{\gamma\in\calC_N^{\sf reg}} \setof{\uomega}{\uomega\sim\gamma} .
\end{equation}
In view of~\eqref{eq:fact}, Theorem~\ref{thm:ERW-g} implies that
\begin{theorem}\label{thm:ERW-om}
	For any \(\lambda\geq 0\), there exists a choice of parameters in~\eqref{eq:kappa-scales-g} and~\eqref{eq:dlrh-N-g} such that
 	\begin{equation}\label{eq:ERW-2}
	 	\lim_{N\to\infty} \PNbetal \bigl( \uomega \notin \uOmega_N^{\sf reg}(c_+,d_+,\kappa,\iota,\epsilon) \bigr) = 0 .
 	\end{equation}
\end{theorem}

\subsection{Coupling with random walks driven by infinite-volume weights}
\label{sub:cpl}

Our proof of the invariance principle in Section~\ref{sec:Proof} is based on the coupling between the interface \(\gamma\) and effective random walks under linear area tilts, which is stated in Theorem~\ref{thm:coupling} below.
We need to define some additional notation.

\smallskip
Let \(\sfu,\sfv \in \bbH_+^*\) and let \(\uOmega^{\sfu,\sfv}\) be the set of all concatenations
\(\uomega \sim \eta \in \SetRootDiaCont(\sfu,\sfv)\).
Recall that, in general, since we have compromised on irreducibility, there are many different concatenations \(\uomega\) which are compatible with a path \(\eta \in \SetRootDiaCont(\sfu,\sfv)\).
Define the following weights on \(\uomega = (\omega_1, \dots, \omega_M) \in \uOmega^{\sfu,\sfv}\) and, accordingly, define the probability measures
\begin{equation}\label{eq:Wbeta}
	\Wbeta(\uomega) \defby \prod_{i=1}^M \sfp_\beta(\omega_i)
	\quad\text{ and }\quad
	\Pbeta{\sfu,\sfv}(\uomega) \propto \Wbeta(\uomega) .
\end{equation}
As we are interested in \(\uomega \subset \bbH_+^*\), we define the conditional measure
\begin{equation}\label{eq:Pbetap}
	\Pbetap{\sfu,\sfv}(\cdot) = \Pbeta{\sfu,\sfv}(\cdot \given \uomega \subset \bbH_+^*) .
\end{equation}
As was already mentioned, for any diamond-confined \(\omega \in \SetRootDiaCont\) lying above \(\calL\), one can unambiguously define, using the deformation rules as specified in Subsection~\ref{sub:random-lines}, the area \(\sfA (\omega)\) between \(\omega\) and \(\calL\) inside the semi-strip \(\calS_{\fend(\omega),\bend(\omega)}\).
If \(\uomega = \omega_1 \circ \dots \circ \omega_M \subset \bbH_+^*\), we set
\begin{equation}\label{eq:A-uomega}
	\sfA (\uomega) \defby \sum_{i=1}^M \sfA(\omega_i) .
\end{equation}
Finally, we define
\begin{equation}\label{eq:Pbetapl}
	\Pbetapl{\sfu,\sfv}(\uomega)
	\propto
	\exp\Bigl\{ -\frac{2\lambda m^*_\beta}{N} \sum_{i=1}^M \sfA (\omega_i) \Bigr\}
	\Pbetap{\sfu,\sfv}(\uomega) .
\end{equation}
More generally, let \(\xi_L\) and \(\xi_R\) be two probability measures on \(\bbH_+^*\) such that
\(\xi_L\otimes\xi_R\) is supported on the set
\[
	\setof{ (\sfu,\sfv)}{\sfv\in \sfu + \fcone} .
\]
Then, let us define
\begin{equation}\label{eq:Pbetapl-xi}
	\Pbeta{\xi_L,\xi_R} \defby \sum_{\sfu,\sfv} \xi_L(\sfu) \xi_{R}(\sfv)\, \Pbeta{\sfu,\sfv}
	\quad\text{ and }\quad
	\Pbetapl{\xi_L,\xi_R} \defby \sum_{\sfu,\sfv} \xi_L(\sfu) \xi_{R}(\sfv)\, \Pbetapl{\sfu,\sfv} .
\end{equation}

Let \(\gamma\in\calC_N\) be an interface as sampled from \(\mu_{\BOXN;\beta,\lambda/N}^{\pm}\).
Given \(\sfu,\sfv\) and \(\uomega\in\uOmega^{\sfu,\sfv}\), let us say that \(\{\uomega \subset \gamma\}\) if \(\sfu\) and \(\sfv\) are cone-points of \(\gamma\) and \(\uomega \sim \gamma \cap \calD(\sfu,\sfv)\).
\begin{theorem}\label{thm:coupling}
 	Recall~\eqref{eq:kappa-scales-g} and~\eqref{eq:dlrh-N-g}.
 	For all \(N\) sufficiently large, there exist
	\begin{enumerate}[label=(\roman*)]
		\item probability measures \(\xi_L^N\) on \(\calB_{N^{\iota},h_N \log N}(\ell_N)\)
		and \(\xi_R^N\) on \(\calB_{N^{\iota},h_N \log N}(r_N)\);
 		\item a coupling \(\Phi_N\) between \(\mu_{\BOXN;\beta,\lambda/N}^{\pm}\) and \(\Pbetapl{\xi_L^N,\xi_R^N}\) such that
	 	\begin{equation}\label{eqPhiN}
		 	\lim_{N\to\infty} \Phi_N (\uomega\not\subset\gamma) = 0.
	 	\end{equation}
	\end{enumerate}
\end{theorem}
Consequently, one is entitled to explore \(\gamma\) on the \((N^{\frac23},N^{\frac13})\)-scale using the effective-random-walk measure \(\Pbetapl{\xi_L^N,\xi_R^N}\).
This is precisely what we shall do in Section~\ref{sec:Proof}.

\begin{proof}[Proof of  Theorem~\ref{thm:coupling}]
The claim of Theorem~\ref{thm:coupling} is a straightforward consequence of Theorem~\ref{thm:ERW-om}, the bulk-relaxation property~\eqref{eq:pnb-pb} of the finite-volume weights \(\{\sfp_N\}\) and of the approximate domain-decoupling property of \(\psi_N\) as formulated in~\eqref{eq:psi-N-sum}.

Indeed, let \(\uomega \in \uOmega_N^{\sf reg}(c_+,d_+,\kappa,\iota,\epsilon)\).
By~\eqref{eq:W-bound-cp}, we may restrict attention to \(\max_i \frc(\omega_i) \leq C_\beta\log N\).

By properties~\eqref{eq:C2} and~\eqref{eq:C3}, we can choose the leftmost vertex \(\sfu\) of \(\sfS(\uomega)\) inside \(\calB_{N^{\iota},2h_N + C_\beta\log N d_N}(\ell_N)\) and the right-most vertex \(\sfv\) inside \(\calB_{N^{\iota},2h_N + C_\beta\log N d_N}(r_N)\). Indeed, there is at least one cone-point in the box \(\calB_{N^{\iota},2h_N}(\ell_N)\), so the leftmost vertex \(\sfu\) of \(\sfS(\uomega)\) is at most at the height $2h_N + C_\beta\log N d_N$, which is of the order of $h_N \log N$.
Let \(\uomega^\prime = \uomega \cap \calD(\sfu,\sfv) \in \SetRootDiaCont(\sfu,\sfv)\) be the portion of \(\uomega\) between \(\sfu\) and \(\sfv\).
Let us write \(\uomega = \uomega_L \circ \uomega^\prime  \circ \uomega_R\).

\smallskip
By property~\eqref{eq:C1}, \(\uomega^\prime \) stays above the rectangular box \(\calB_{N^\kappa,N^\epsilon}\) and, consequently,
by the bulk relaxation property~\eqref{eq:pnb-pb}, we can write \(\Wbeta(\uomega^\prime )\) instead of \(\WNbeta (\uomega^\prime )\).

\smallskip
By construction,
\begin{equation}\label{eq:Box-sum}
	\abs{\BOX_N^-[\uomega ]} = \sfA(\uomega_L) + \sfA(\uomega^\prime ) + \sfA(\uomega_R)
	\text{ and, moreover, }
	\abs{\uomega^\prime } = \sfo\bigl(N^\kappa(\log N)^2\bigr),
\end{equation}
where the latter bound is uniform in all choices of \(\sfu\) and \(\sfv\) in question, since the probability of finding an irreducible piece of diameter larger than \((\log N)^2\)
vanishes with \(N\).

Substituting~\eqref{eq:psi-N-sum} and~\eqref{eq:Box-sum} into~\eqref{eq:PNb-PNbl}, we can \emph{factor out}, up to a uniform factor of order \((1+\smo{1})\),  the terms
\[
	\Wbeta(\uomega') \exp\Bigl\{ -\frac{2\lambda m^*_\beta}{N} \sfA(\uomega') \Bigr\}
\]
in the expression for \(\PNbetal (\uomega_L \circ \uomega' \circ \uomega_R \given	\uOmega_N^{\sf reg})\).

The rest follows from~\eqref{eq:ERW-2} and an application of the total probability formula for  \(\PNbetal(\cdot \given \uOmega_N^{\sf reg})\) with respect to the partition induced by the choice of the vertices \(\sfu\) and \(\sfv\) as above.
\end{proof}

\section{Proof of the main result}
\label{sec:Proof}

At this stage, the proof boils down to a derivation of a uniform invariance principle for the diffusive rescaling of the family \(\{\Pbetapl{\xi_L^N,\xi_R^N}\}\) appearing in the statement of Theorem~\ref{thm:coupling}.

From now on, the parameters \(\kappa\) and \(\iota\) introduced in~\eqref{eq:kappa-scales-g} are chosen in  the following way:
\begin{equation}\label{eq:kappa-choice}
	\kappa \defby \frac{2}{3} + 2\delta \quad \text{and}\quad \iota \defby \frac{2}{3} + \delta ,
\end{equation}
where \(\delta>0\) is chosen small enough to ensure that the following two technical conditions are satisfied:
\begin{equation}\label{eq:d-fix}
	3\delta < \epsilon \quad\text{and}\quad 6\delta <\frac13 .
\end{equation}
As will be explained below, \eqref{eq:d-fix} helps to secure a simple transition from the polymer chains $\uomega$ to the trajectories of the effective random walk $\underline{\sfS}$.
\smallskip

\noindent
By Theorems~\ref{thm:coupling} and~\ref{thm:ERW-g}, it suffices to prove that both
\begin{enumerate}[label=\textsf{\alph*}.]
	\item convergence of finite-dimensional distributions and
	\item tightness on fixed intervals $[-T, T]$
\end{enumerate}
hold \emph{uniformly} in all sequences of measures \(\{\Pbetapl{\sfu_N,\sfv_N}\}\) with (note that we redefine $\ell_N, r_N$ as compared to~\eqref{eq:dlrh-N-g})
\begin{equation}\label{eq:uv-param}
	\begin{split}
		\sfu_N &\defby (\ell_N, u_N) \in \calB_{N^{\iota},h_N \log N}(-N^\kappa + N^\iota) \setminus \calB_{N^{\iota}, N^\epsilon}(-N^\kappa + N^\iota) ,\\
		\sfv_N &\defby (r_N, v_N) \in \calB_{N^{\iota},h_N \log N}(N^\kappa - N^\iota) \setminus \calB_{N^{\iota}, N^\epsilon}(N^\kappa - N^\iota) .
	\end{split}
\end{equation}

\subsection{Transition from the polymer chains $\uomega$ to the effective random walks (ERW) $\underline{\sfS}$}
\label{sub:PC-RW}

Recall that $\Pbetapl{\sfu,\sfv}$ is a probability distribution on \(\uOmega^{\sfu,\sfv}\), the set of all concatenations of paths \(\eta \in \SetRootDiaCont(\sfu,\sfv)\).
We define the length of $\uomega\sim\eta$ by
\[
\abs{\uomega} \defby \abs{\eta} .
\]
Following Definition~\ref{def:S-walk}, we associate with each polymer chain $\uomega\in\uOmega^{\sfu,\sfv}$ an ERW
\[
\underline{\sfS} = \{ \sfu = \sfS_0, \sfS_1 , \dots , \sfS_\ell = \sfv \}
\]
from $\sfu$ to $\sfv$.
It will be convenient to record the vertices of \(\underline{\sfS}\) as $\sfS_i \defby (\sfT_i, \sfZ_i)$ and the corresponding increments as
\[
\sfX_i = (\theta_i, \zeta_i) \defby \sfS_i - \sfS_{i-1} .
\]
Observe that, by Remark~\ref{rem:ZeroExp}, \(\sfE_\beta(\zeta_i)=0\), where \(\sfE_\beta\) denotes expectation with respect to the measure \(\sfp_\beta\).

\smallskip
In the sequel, the size of the maximal step will be denoted by
\begin{equation}\label{eq:gap-size}
	\gap(\underline{\sfS}) \defby \max_i \normII{\sfX_i} .
\end{equation}
With an abuse of notation, we continue to use $\uOmega^{\sfu,\sfv}$ for the collection of such walks and
$\Pbetapl{\sfu,\sfv}$ for the induced distribution.
Note that the number $\ell$ of steps is not fixed under $\Pbetapl{\sfu,\sfv}$, that is, we are in the realm of  two-dimensional renewals.

There is a convenient definition of the area: For  $\underline{\sfS} = \{ \sfS_0, \sfS_1, \dots, \sfS_\ell \}$, we set
\[
\sfA(\underline{\sfS}) \defby \sum_{i=1}^\ell \theta_i \sfZ_{i-1}.
\]

The choice of parameters in~\eqref{eq:kappa-choice} an~\eqref{eq:d-fix} implies the following controls.
\begin{lemma}\label{lem:controls}
	Let \(\delta>0\) and \(\eta>3\delta\) be small enough to ensure that \(2/3 + 3\delta + \eta < 1\) (which is an eligible choice by the second condition in~\eqref{eq:d-fix}.).
	The following statements hold for all $N$ large enough with $\Pbetapl{\sfu_N,\sfv_N}$-probability tending to one uniformly in sequences $\{\sfu_N, \sfv_N\}$ satisfying~\eqref{eq:uv-param}.
	\begin{description}
		\item[\sl 1. Control of the area]
		There exists \(C_1\) such that
		\begin{equation}\label{eq:A-control}
			\frac{\nabs{\BOX_N^-[\uomega]}}{N} \leq C_1 N^{3\delta}
		\end{equation}
		\item[\sl 2. Control of maximal gaps]
		\begin{equation}\label{G-control}
			\gap(\underline{\sfS}) \leq N^{\eta} .
		\end{equation}
		\item[\sl 3. Control of the length $\abs{\uomega}$]
		There exists \(C_2\) such that
		\begin{equation}\label{L-control}
			\abs{\uomega} \leq C_2 N^{2/3 + 2\delta}.
		\end{equation}
		\item[\sl 4. Control of mismatched area between {$\BOX_N^-[\uomega]$} and $\sfA(\underline{\sfS})$]
		\begin{equation}\label{eq:MA-control}
			\frac1N \babs{\nabs{\BOX_N^-[\uomega]} - \sfA(\underline{\sfS})} \leq N^{2/3 + 3 \delta + \eta - 1} = \smo{1}.
		\end{equation}
	\end{description}
\end{lemma}
\begin{proof}
	\noindent\textsl{1. Control of the area.}
	First observe that
	\[
	\Pbetapl{\sfu_N,\sfv_N}\bigl(\nabs{\BOX_N^-[\uomega]} > C_1N^{1+3\delta}\bigr)
	\leq
	\frac{\exp\Bigl\{ -2\lambda m^*_\beta C_1 N^{3\delta} \Bigr\}}{\Ebetap{\sfu_N,\sfv_N}\Bigl[\exp\Bigl\{ -\frac{2\lambda m^*_\beta}{N}\nabs{\BOX_N^-[\uomega]}\Bigr\}\Bigr]} ,
	\]
	so we only need a lower bound on the denominator. This can be achieved using an argument completely similar to the one in the proof of~\eqref{eq:LowerBndCReg3}, so we only sketch it. We consider the set of polymer chains whose associated ERW trajectories \(\underline{\sfS}\) visit the two vertices
	\[
	\bigl(\ell_N+(h_N\log N)^2,h_N\bigr)
	\text{ and }
	\bigl(r_N-(h_N\log N)^2,h_N\bigr)
	\]
	and stay inside tubes of width \(h_N\), as we did in the lower bound~\eqref{eq:LowerBndCReg3}. Each such polymer chain \(\uomega\) satisfies \(\nabs{\BOX_N^-[\uomega]} \leq 3N^{1+2\delta}\). The proof thus reduces to proving that the \(\Pbetap{\sfu_N,\sfv_N}\)-probability of this set of polymer chains is bounded below by \(\exp\{-\sfO(N^{2\delta})\}\), which can be done by proceeding analogously to the proof of Lemma~\ref{lem:tube}.

	\noindent\textsl{2. Control of maximal gaps.}
	It follows from the first point that
	\begin{align*}
		\Pbetapl{\sfu_N,\sfv_N} \bigl( \gap(\underline{\sfS}) > N^{\eta} \bigr)
		&\leq
		e^{\sfO(N^{3\delta})}\, \Pbetap{\sfu_N,\sfv_N} \bigl( \gap(\underline{\sfS}) > N^{\eta} \bigr) \\
		&\leq
		e^{\sfO(N^{3\delta})}\, \Pbeta{\sfu_N,\sfv_N} \bigl( \gap(\underline{\sfS}) > N^{\eta} \bigr) \\
		&\leq
		e^{\sfO(N^{3\delta})}\, \sfp_\beta \bigl( \normII{X_1} > N^{\eta} \bigr)
		\leq
		e^{\sfO(N^{3\delta}) - \sfO(N^{\eta})},
	\end{align*}
	which tends to \(0\) since \(\eta>3\delta\).

	\noindent\textsl{3. Control of the length $\abs{\uomega}$.}
	Once more, it follows from the first point that
	\begin{align*}
		\Pbetapl{\sfu_N,\sfv_N} \bigl( \nabs{\uomega} > C_2 N^{2/3+2\delta}\bigr)
		&\leq
		e^{\sfO(N^{3\delta})}\, \Pbeta{\sfu_N,\sfv_N} \bigl( \nabs{\uomega} > C_2 N^{2/3+2\delta}\bigr),
	\end{align*}
	and the result now follows from standard estimates for bulk connectivities (for instance through skeleton calculus).

	\noindent\textsl{4. Control of mismatched area between {$\BOX_N^-[\uomega]$} and $\sfA(\underline{\sfS})$.}
	Note that the contribution to the mismatched area from any given increment of \(\underline{\sfS}\) is smaller than \(\theta_i^2 \leq \theta_i N^\eta\) by the second point above (remember that the corresponding path is constrained to lie inside a diamond). Therefore,
	\[
	\babs{\nabs{\BOX_N^-[\uomega]} - \sfA(\underline{\sfS})}
	\leq
	N^\eta \sum_{i=1}^\ell \theta_i
	\leq
	N^\eta (2 N^{2/3+2\delta}+1)
	\leq
	N^{2/3 + 3 \delta + \eta} .\qedhere
	\]
\end{proof}

\subsection{FS scaling of the interfaces}

Let us define the quantity
\begin{equation}\label{eq:chi-fun}
	\chi_\beta \defby \frac{\sfE_\beta(\zeta^2)}{\sfE_\beta(\theta)} = \frac{\Var_\beta(\zeta)}{\sfE_\beta(\theta)},
\end{equation}
which is precisely the curvature of the Wulff shape $\mathbf{K}_\beta$ at the point \(\taub\sfe_1\)
(see Appendix~\ref{sec:Curvature}).

Let $\sfu,\sfv\in \bbH_+$ and let
\[
\underline{\sfS} = \{ \sfS_0, \sfS_1, \dots, \sfS_n \} \in \uOmega^{\sfu,\sfv} .
\]
For $N\in\bbN$, let us define the linear interpolation (LI) through the rescaled vertices of $\underline{\sfS}$:
\begin{equation}\label{eq:IN}
	\fri_N[\underline{\sfS}] \defby \text{LI}\Bigl\{\Bigl(\frac{\sfT_0}{N^{2/3}},\frac{\sfZ_0}{N^{1/3}\sqrt{\chi_\beta}} \Bigr), \dots, \Bigl(\frac{\sfT_n}{N^{2/3}},\frac{\sfZ_n}{N^{1/3}\sqrt{\chi_\beta}}\Bigr)\Bigr\} .
\end{equation}
If $\sfu = (s N^{2/3}, r\sqrt{\chi_\beta}N^{1/3})$ and $\sfv = (t N^{2/3}, y\sqrt{\chi_\beta}N^{1/3})$,
we shall think of $\fri_N[\underline{\sfS}]$ as a random non-negative function on the rescaled \emph{time} interval $[s,t]$ and, with a slight abuse of notation, write $\fri_N(\tau)$, $\tau\in [s,t]$.

\subsection{ERW partition functions}
Let us define the following set of partition functions and probability measures related to the effective random walks $\underline{\sfS}$.
We shall use the same symbols $\sfP$ and $\sfE$ for probability measures and the corresponding expectations, but shall employ the symbol $\calG$ for partition functions in order to avoid confusion with the symbol $\calZ$ used for partition functions phrased in terms of interfaces/polymer chains and their weights.
\begin{description}
	\item[\textit{$n$-step partition functions {$\calG_{\beta,\lambda/N,+}^{\sfu;n}[f]$.}}]
	Given a function $f$ on $\bbN$ and a vertex $\sfu\in \bbH_+$, we define
	\begin{align}
		\calG_{\beta,\lambda/N,+}^{n}[f](\sfu)
		&=
		\calG_{\beta,\lambda/N,+}^{\sfu;n}[f] \notag\\
		&\defby
		\sfE^{\sfu}_\beta \Bigl[
		\exp\Bigl( -\frac{2\lambda m^*_\beta}{N} \sum_{i=1}^{n} \sfZ_{i-1}\theta_{i} \Bigr) f(\sfZ_n) \,,\,
		\underline{\sfS}[0,n]\subset \bbH_+
		\Bigr] \\
		\label{eq:G-nstep}
		&=
		\sfE^{\sfu}_\beta \Bigl[
		\exp\Bigl( -\frac{2\lambda m^*_\beta}{N}\, \sfA(\underline{\sfS}[0,n]) \Bigr) f(\sfZ_n) \,,\,
		\underline{\sfS}[0,n]\subset \bbH_+
		\Bigr] ,
	\end{align}
	where \(\sfE^{\sfu}_\beta\) denotes expectation with respect to the ERW \(\underline{\sfS}\) with increments of law \(\sfp_\beta\) starting at \(\sfu\) and \(\underline{\sfS}[0,n]\) denotes the first \(n\) steps of the trajectory of \(\underline{\sfS}\).
	\smallskip
	\item[\textit{$n$-step partition functions $\calG_{\beta,\lambda/N,+}^{\sfu,\sfv;n}$}.]
	Given $\sfu , \sfv\in \bbH_+$, define
	\begin{align*}
		\calG_{\beta,\lambda/N,+}^{\sfu,\sfv;n}
		&\defby
		\sfE^{\sfu}_\beta \Bigl[
		\exp\Bigl( -\frac{2\lambda m^*_\beta}{N} \sum_{i=1}^{n} \sfZ_{i-1}\theta_{i} \Bigr) \,,\, \sfS_n = \sfv \,,\, \underline{\sfS}[0,n]\subset \bbH_+
		\Bigr] \\
		\label{eq:G-nstep-v}
		&=
		\sfE^{\sfu}_\beta \Bigl[
		\exp\Bigl( -\frac{2\lambda m^*_\beta}{N}\, \sfA(\underline{\sfS}[0,n]) \Bigr) \,,\, \sfS_n = \sfv \,,\, \underline{\sfS}[0,n]\subset \bbH_+
		\Bigr] .
	\end{align*}
\end{description}
Of course,
\begin{equation}\label{eq:G-pf-rel}
	\calG_{\beta,\lambda/N,+}^{\sfu;n}[f] = \sum_{\sfv=(v_1,v_2)} \calG_{\beta,\lambda/N,+}^{\sfu,\sfv;n} f(v_2) .
\end{equation}
For \(1\leq n_1<n_2<\cdots<n_m < n\), we will also use the notation
\begin{multline}
\label{eq:def_part_functions_fin_moments}
	\calG_{\beta,\lambda/N,+}^{\sfu,\sfv;n} \bigl(\prod_{i=1}^m f_{n_i} \bigr)
	\defby \\\sum_{\sfv^1, \dots, \sfv^{m}} \calG_{\beta,\lambda/N,+}^{\sfu,\sfv^1;l_1} f_{n_1}(v^1_2) \calG_{\beta,\lambda/N,+}^{\sfv^1,\sfv^2;l_2} f_{n_2}(v^2_2) \cdots \calG_{\beta,\lambda/N,+}^{\sfv_{m-1},\sfv_m;l_m} f_{n_m}(v^m_2) \calG_{\beta,\lambda/N,+}^{\sfv_{m},\sfv;l_{m+1}} ,
\end{multline}
where \(l_i=n_i-n_{i-1}\), \(n_0=0, n_{m+1} = n\).

\subsection{Mixing}
\label{sub:mixing}

Fix $T >0$. Recall our notation~\eqref{eq:boxes-strips} for lattice semi-strips.
Let $\sfu,\sfu^\prime,\sfv$ and $\sfv^\prime$ be such that
\begin{equation}\label{eq:sp-uv}
	\sfu, \sfu^\prime \in \calS_{-\infty,-TN^{2/3}}
	\quad \text{and} \quad
	\sfv, \sfv^\prime \in \calS_{TN^{2/3},\infty} .
\end{equation}
Let $\underline{\sfS} = \{ \sfs_0, \dots, \sfs_n \} \in \uOmega^{\sfu,\sfv}$ and $\underline{\sfS}^\prime = \{ \sfs_0^\prime, \dots, \sfs_m^\prime \} \in \uOmega^{\sfu^\prime,\sfv^\prime}$.
We shall say that $\underline{\sfS}\stackrel{T}{\sim} \underline{\sfS}^\prime$ if there exist $\ell,\ell^\prime$ and $k$ such that
\begin{equation}\label{eq:Tsim}
	\sfs^{\vphantom{\prime}}_\ell, \sfs^\prime_{\ell^\prime} \in \calS_{-\infty,-TN^{2/3}},\;
	\sfs^{\vphantom{\prime}}_{\ell+k}, \sfs^\prime_{\ell^\prime+k} \in \calS_{TN^{2/3},\infty}
	\text{ and }
	\sfs^{\vphantom{\prime}}_{\ell+j} = \sfs^\prime_{\ell^\prime+j} \text{ for all } j=0, \dots, k.
\end{equation}
By an adjustment of the techniques developed in~\cite{Ioffe+Shlosman+Velenik-15,Ioffe-Velenik-Wachtel-18}, we obtain the following

\begin{proposition}\label{prob:coupl-simT}
	Fix $0 < c < C < \infty$.
	Then, there exist \(K_0 = K_0(\beta,\lambda)\) and $\nu = \nu (\beta,\lambda) > 0$ such that the following holds:  for all $N$ large enough, there exists a coupling $\Phi$ between $\Pbetapl{\sfu,\sfv}$ and $\Pbetapl{\sfu^\prime,\sfv^\prime}$ such that, as soon as $N$ is sufficiently large,
	\begin{equation}\label{eq:coupl-simT}
		\Phi\bigl( \underline{\sfS}\stackrel{T}{\sim} \underline{\sfS}^\prime \bigr)
		\geq
		1 - \mathrm{e}^{-\nu K} ,
	\end{equation}
	for all $K > K_0$, $T > 0$ and all $\sfu,\sfu^\prime,\sfv,\sfv^\prime$ such that
	\begin{equation}\label{eq:sp-uv-K}
		\sfu, \sfu^\prime \in \calS_{-\infty,-(T+K)N^{2/3}}
		\quad \text{and} \quad
		\sfv, \sfv^\prime \in \calS_{(T+K)N^{2/3},\infty}
	\end{equation}
	and, in addition,
	\begin{equation}\label{eq:uv-heights}
		\sfu\cdot\sfe_2, \sfu^\prime\cdot\sfe_2, \sfv\cdot\sfe_2, \sfv^\prime\cdot\sfe_2 \in (cN^{1/3},CN^{1/3}) .
	\end{equation}
\end{proposition}
\begin{proof}
The claim can be proved using a variant of the proof of the corresponding claim in~\cite[Proposition~5]{Ioffe+Shlosman+Velenik-15}.
Minor adaptations are needed because we deal here with a directed random walk in \(\bbZ^2\) rather than with the space-time path of a random walk in \(\bbZ\) as in the latter paper.
Since the changes are mostly straightforward, we only briefly sketch the argument for the sake of the reader.

The proof relies on a coupling argument. Let \(\underline{\sfS}\) and \(\underline{\sfS}'\) be independent processes distributed according to \(\Pbetapl{\sfu,\sfv}\) and \(\Pbetapl{\sfu',\sfv'}\), respectively.

We start by decomposing the interval \([T N^{2/3}, (T+K) N^{2/3}]\) into consecutive disjoint intervals \(I_1, I_2, \dots\) of length \(N^{2/3}\).
The first observation is that, in view of our target estimate, we can assume that there exists \(\epsilon\in(0,1/3)\) such that both trajectories have at least \(\epsilon N^{2/3}\) steps in each interval, with the left-most (resp.\ right-most) step occurring at a distance at most \(\epsilon N^{2/3}\) from the left end (resp.\ right end) of the interval. Indeed, as we did in the proof of Lemma~\ref{lem:controls}, at a cost at most \(e^{N^{3\delta}}\), we can set \(\lambda=0\) and ignore the positivity constraint, which reduces the estimate to a straightforward large deviations upper bound once \(\epsilon\) is small enough.

Fix \(\upsilon>0\) sufficiently large. Then, it is easy to prove that in the vast majority of the intervals \(I_k\), both \(\underline{\sfS}\) and \(\underline{\sfS}'\) visits points with height below \(\upsilon N^{1/3}\). Indeed, when this is not the case the total area \(\sfA\) has to be much larger than \(N\), which is easily seen to be very unlikely.

This implies that there is a positive fraction of the intervals \(I_{2k}\) such that the above occurs both in \(I_{2k-1}\) and in \(I_{2k+1}\). It is then easy to prove that, in such an interval \(I_{2k}\), there is a uniformly positive probability that both \(\underline{\sfS}\) and \(\underline{\sfS}'\) never reach a height above \(2\upsilon N^{1/3}\). Let us say that an interval where this occurs is \(\upsilon\)-good. We conclude that there is a positive fraction of \(\upsilon\)-good intervals. (See~\cite[Lemma~3]{Ioffe+Shlosman+Velenik-15} for the detailed argument.)

Finally, there exists \(p>0\), independent of $N$, such that, above each \(\upsilon\)-good interval, uniformly in what occurs outside of them, there is a probability at least \(p\) that the two trajectories \(\underline{\sfS}\) and \(\underline{\sfS}'\) meet. The proof of the corresponding statement when \(\lambda=0\) can be done using the second moment method (similarly as in~\cite[Appendix~A]{Ioffe+Shlosman+Velenik-15}). The extension to \(\lambda>0\) then follows from the fact that, in an \(\upsilon\)-good interval, the scaled area \(\frac1N\sfA\) is uniformly bounded above by a constant (depending on \(\upsilon\)); see~\cite[Proposition~6]{Ioffe+Shlosman+Velenik-15}.

It follows that, up to an event of probability exponentially small in \(K\), the two trajectories \(\underline{\sfS}\) and \(\underline{\sfS}'\) meet above \([T N^{2/3}, (T+K) N^{2/3}]\). Of course, the same is true over the interval \([-(T+K) N^{2/3}, -T N^{2/3}]\). When this happens, it is straightforward to couple the two trajectories over the interval \([-T N^{2/3}, T N^{2/3}]\) and the conclusion follows.
\end{proof}

\subsection{Scaling and Trotter--Kurtz formula}
In the sequel, we rely on~\cite[Section~1.6]{ethier2009markov} (specifically on Theorem~1.6.5 therein)  and we refer to~\cite[Section~1.2]{Ioffe+Shlosman+Velenik-15} for a concise description of the functional-analytic setup we employ.

By~\eqref{eq:IN}, it is natural to consider the following spatial rescaling:
\begin{equation}\label{eq:sp-scal}
	x\mapsto r \defby \frac{x}{N^{1/3}\sqrt{\chi_\beta}} .
\end{equation}
We also define the rescaling operator \(\scalingop f (x) = f\bigl(xN^{-1/3}\chi_\beta^{-1/2} \bigr)\).
Let $f$ be a smooth test function with compact support in $(0,\infty)$.
Let us consider the operator
\[
\transferOp_N f(r) \defby \sfE_\beta \Bigl[
\exp\Bigl( -\frac{2\lambda m^*_\beta}{N} \theta r N^{1/3}\sqrt{\chi_\beta} \Bigr)
f\Bigl(r + \frac{\zeta}{N^{1/3}\sqrt{\chi_\beta}}\Bigr) \bIF{r + \frac{\zeta}{N^{1/3}\sqrt{\chi_\beta}} \geq 0}
\Bigr].
\]
Fix $r>0$. Then, a second-order Taylor expansion yields, using the fact that the constraint \(r + \frac{\zeta}{N^{1/3}\sqrt{\chi_\beta}} \geq 0\) in the definition of \(\transferOp_n\) can be removed (since \(f\) is supported on \((0, \infty)\)) and the fact that $\sfE_\beta(\zeta) = 0$
,
\begin{align}
	\label{eq:TK-gen}
	\lim_{N\to\infty} N^{2/3} \Bigl\{
	\sfE_\beta \Bigl[
	&\exp\Bigl( -\frac{2\lambda m^*_\beta}{N} \theta r N^{1/3}\sqrt{\chi_\beta} \Bigr)
	f\Bigl(r + \frac{\zeta}{N^{1/3}\sqrt{\chi_\beta}}\Bigr)
	\Bigr]
	- f(r)
	\Bigr\} \notag\\
	&=
	\lim_{N\to\infty} N^{2/3} \bigl\{
	\transferOp_N f(r) - f(r)
	\bigr\} \\
	&=
	\sfE_\beta(\theta)  \Bigl( \tfrac{1}{2} f''(r) - 2\lambda m^*_\beta \sqrt{\chi_\beta} r f(r) \Bigr)
	\defby
	\sfE_\beta(\theta)\,\calL_\beta f(r) \notag.
\end{align}
The last equality in~\eqref{eq:TK-gen} defines the Airy operator
\begin{equation}\label{eq:AiryOp}
	\calL_\beta f(r) = \tfrac{1}{2} f''(r) - 2\lambda m^*_\beta \sqrt{\chi_\beta} r f(r),
\end{equation}
which is a singular Sturm--Liouville operator on $\bbR_+$ with Dirichlet boundary conditions at zero.

Let $f$ belong to the domain \(\calD(\calL_\beta)\) of the self-adjoint extension of the operator $\calL_\beta$, see~\cite[equ.~(1.4)]{Ioffe+Shlosman+Velenik-15}.
It then follows from~\cite[Theorem~1.6.5]{ethier2009markov} that
\begin{equation}\label{eq:uni-lim}
	\lim_{N\to\infty} \transferOp_N^{\lfloor t N^{2/3}/\sfE_\beta(\theta) \rfloor} f = \mathrm{e}^{t\calL_\beta} f ,
\end{equation}
uniformly in $t$-s from bounded intervals of $\bbR_+$.

Going back to~\eqref{eq:G-nstep}, we conclude that
\begin{lemma}
	\label{lem:TK-form}
	Let $f\in\calD(\calL_\beta)$. Then,
	\begin{equation}\label{eq:FS-lim}
		\lim_{N\to\infty}
		\calG_{\beta,\lambda/N,+}^{\lfloor tN^{2/3}/\sfE_\beta(\theta) \rfloor}[\scalingop f] = \mathrm{e}^{t\calL_\beta} f ,
	\end{equation}
	uniformly in $t$ on bounded intervals of $\bbR_+$.
\end{lemma}
\begin{proof}
	This is an immediate consequence of~\eqref{eq:uni-lim}, once one observes that
	\[
		\calG_{\beta,\lambda/N,+}^{k}[\scalingop f] = \transferOp_N^kf.\qedhere
	\]
\end{proof}

\subsection{Time-change and tightness}
\label{sec:Tightness}
Fix $\epsilon >0$ small.
Recall our notation~\eqref{eq:boxes-strips} and consider $\sfu$ and $\sfv$ living in the shifted boxes
\begin{align}
	\label {eq:uv-ref}
	\sfu &\in (-LN^{2/3},cN^{1/3}) + \calB_{N^{1/3+\epsilon},CN^{1/3}} , \notag\\
	\sfv &\in ( LN^{2/3},cN^{1/3}) + \calB_{N^{1/3+\epsilon},CN^{1/3}} .
\end{align}
The rescaled area $\sfA(\underline{\sfS})/N$ has uniform tails with respect to the family of probability measures $\{\Pbetapl{\sfu,\sfv}\}$ with $\sfu,\sfv$ satisfying~\eqref{eq:uv-ref}.
Therefore, any asymptotic event which has vanishing $\Pbetap{\sfu,\sfv}$-probability as $N\to\infty$ has also vanishing $\Pbetapl{\sfu,\sfv}$-probability and can be ignored.
Let us now apply this observation to two particular families of events.

\smallskip\noindent
\textbf{Time-change:}
Let $\underline{\sfS}\in \uOmega^{\sfu,\sfv}$.
We shall write $n(\underline{\sfS})$ for the number of steps.
Then, there exists \(c>0\) such that
\begin{align}
	\Pbetap{\sfu,\sfv} \Bigl( \Babs{n(\underline{\sfS}) - \frac{2LN^{2/3}}{\sfE_\beta(\theta)}}
	\geq N^{1/3 + \epsilon_1} \Bigr)
	&\leq
	N^c
	\Pbeta{\sfu,\sfv} \Bigl( \Babs{n(\underline{\sfS}) - \frac{2LN^{2/3}}{\sfE_\beta(\theta)}}
	\geq N^{1/3 + \epsilon_1} \Bigr) \\
	&\leq \mathrm{e}^{-c_1 N^{2\epsilon_1}} ,
	\label{eq:step-bound}
\end{align}
uniformly in $\sfu,\sfv$ satisfying~\eqref{eq:uv-ref}. The first inequality follows from the fact that in our setting the removal of the positivity constraint results in the power-law correction. The second inequality in~\eqref{eq:step-bound} follows from the local limit theorem and standard large deviation bounds, since the increments of \(\underline{\sfS}\) have exponential tails.

\smallskip
As before, let us record $\sfu$ and $\sfv$ as
\[
\sfu = (s N^{2/3}, r\sqrt{\chi_\beta}N^{1/3})
\quad\text{ and }\quad
\sfv = (t N^{2/3}, y\sqrt{\chi_\beta}N^{1/3}).
\]
For $\underline{\sfS}\in \uOmega^{\sfu,\sfv}$, instead of~\eqref{eq:IN}, consider the following fixed-time-step interpolation $\frj_N$:
\begin{equation}\label{eq:JN}
	\frj_N[\underline{\sfS}]
	\defby
	\operatorname{LI}\Bigl\{
	(s,r),
	\Bigl( s + \frac{t-s}{n(\underline{\sfS})},\frac{\sfZ_1}{N^{1/3}\sqrt{\chi_\beta}} \Bigr),
	\Bigl( s + 2\frac{t-s}{n(\underline{\sfS})},\frac{\sfZ_2}{N^{1/3}\sqrt{\chi_\beta}} \Bigr),
	\dots, (t,y)
	\Bigr\} .
\end{equation}
Under $\Pbetap{\sfu,\sfv}$, both $\frj_N$ and $\fri_N$ are random functions defined on $[s,t]$.
They are related by a time-change:
\begin{equation}\label{eq:tc-phi}
	\frj_N(\tau) = \fri_N\bigl(\varphi(\tau)\bigr),
\end{equation}
where $\varphi$ is the linear interpolation through the points,
\begin{equation}\label{eq:phi-form}
	\varphi\Bigl(s + \frac{(t-s)k}{n}\Bigr) = \frac{\sfT_k}{N^{2/3}} .
\end{equation}
Clearly,
\begin{gather}
	\label{eq:mod-cont}
	\delta_\varphi \defby \max_{\tau\in (s,t)} \abs{\varphi(\tau) - \tau} \leq \max_k \frac{1}{N^{2/3}}\Babs{\frac{k}{n} (\sfT_n-\sfT_0) - (\sfT_k-\sfT_0)} \\
	\shortintertext{and}
	\max_{\tau\in [s,t]}
	\abs{\frj_N(\tau) - \fri_N(\tau)} \leq \max_{\abs{\eta-\tau}\leq \delta_\varphi} \abs{\frj_N(\tau) - \frj_N(\eta)} .
\end{gather}
By~\eqref{eq:step-bound}, we may restrict attention to
\begin{equation}\label{eq:n-range}
	\Babs{n(\underline{\sfS}) - \frac{2LN^{2/3}}{\sfE_\beta(\theta)}} \leq N^{1/3 + \epsilon_1} .
\end{equation}
At this stage, usual stretched exponential moderate deviations upper bounds for the \(n\)-step trajectories $\sfT[0,n]$ and $\sfZ[0,n]$ under $\sfP_\beta$, with $n$ in the range described in~\eqref{eq:n-range}, imply that there exist $\epsilon_2,\epsilon_3,c_2 > 0$ such that
\begin{equation}\label{eq:dist-bound}
	\Pbetap{\sfu,\sfv} \Bigl( \max_{\tau\in [s,t]} \abs{\frj_N(\tau) - \fri_N(\tau)} \geq N^{-\epsilon_2} \Bigr)
	\leq \mathrm{e}^{-c_2 N^{\epsilon_3}} ,
\end{equation}
uniformly in $\sfu,\sfv$ satisfying~\eqref{eq:uv-ref}.

\smallskip\noindent
\textbf{Tightness:}
Uniform tightness on $\sfC\bigl([-T, T],\bbR_+\bigr) \subset \sfC\bigl([s, t],\bbR_+\bigr)$ of the family $\{\Pbetap{\sfu,\sfv}\}$, with $L > T+1$ and $\sfu,\sfv$ satisfying~\eqref{eq:uv-ref},
for the fixed-time-step linear interpolation $\frj_N$ in~\eqref{eq:JN},
was explained in~\cite[end of Section~5.6]{Ioffe+Ott+Velenik+Wachtel-20}.
By~\eqref{eq:dist-bound}, it extends to $\{\Pbetapl{\sfu,\sfv}\}$.
Alternatively, one can carefully follow the route which was paved in~\cite[Section~5.2]{Ioffe-Velenik-Wachtel-18} using the (functionally independent) finite-dimensional distribution statement~\eqref{eq:fin-dist} below.~

\subsection{Proof of the invariance principle to FS diffusions}
\label{sec:convFDD}
In the sequel, we take $L = T+K$.
The invariance principle will be recovered in the limit $\lim_{K\to\infty}\lim_{N\to\infty}$.

In view of the uniform mixing estimates stated in Proposition~\ref{prob:coupl-simT}, it suffices to prove convergence of finite-dimensional distributions for the following family $\{\hPbetapl{f,g}\}$ of convex combinations of $\Pbetapl{\sfu,\sfv}$, with $\sfu,\sfv$ satisfying~\eqref{eq:uv-ref}, (recall~\eqref{eq:def_part_functions_fin_moments}),
\begin{equation}\label{eq:Phat-fg}
	\hPbetapl{f,g}(\,\cdot\,)
	\defby
	\frac{
		\sum_{\sfu,\sfv} F_N(\sfu) \phi_{-L,N}(\sfu)
		\sum_n \calG_{\beta,\lambda/N,+}^{\sfu,\sfv;n} (\, \cdot\,  )
		G_N(\sfv) \phi_{L,N}(\sfv)
	}
	{
		\sum_{\sfu,\sfv} F_N(\sfu) \phi_{-L,N}(\sfu)
		\sum_n \calG_{\beta,\lambda/N,+}^{\sfu,\sfv;n}
		G_N(\sfv) \phi_{L,N}(\sfv)
	}.
\end{equation}
Above, $f,g$ are two non-negative non-trivial test functions with compact support in $(c,C)$ and
\begin{equation}\label{eq:FG}
	F_N(\sfu) \defby f\Bigl( \frac{\sfu\cdot\sfe_2}{N^{1/3}\sqrt{\chi_\beta}} \Bigr)
	\text{ and }
	G_N(\sfv) \defby g\Bigl( \frac{\sfv\cdot\sfe_2}{N^{1/3}\sqrt{\chi_\beta}} \Bigr) .
\end{equation}
Furthermore,
\begin{equation}\label{eq:phi}
	\phi_{-L,N}(\sfu) \defby \1_{\{ \sfu\cdot\sfe_1 = -\lfloor LN^{2/3} \rfloor \}}
	\text{ and }
	\phi_{L,N}(\sfv) \defby \1_{\{ \nabs{\sfv\cdot\sfe_1 - LN^{2/3}} \leq N^{1/3 + \epsilon} \}} .
\end{equation}
Fix $\epsilon' < \epsilon$. By Lemma~\ref{lem:TK-form},
\begin{multline}
	\label{eq:fin-dist-pf}
	\lim_{N\to\infty}
	\frac{1}{N^{1/3}\sqrt{\chi_\beta}}
	\sum_{\sfu,\sfv} F_N(\sfu) \phi_{-L,N}(\sfu)
	\frac{1}{N^{1/3}\sqrt{\chi_\beta}}
	\calG_{\beta,\lambda/N,+}^{\sfu,\sfv;n} G_N(\sfv) \phi_{L,N}(\sfv) \\
	=
	\int_{\bbR_+}\int_{\bbR_+} f(r)\mathrm{e}^{2L\calL_\beta}(r,y) g(y) \dd r\dd y ,
\end{multline}
uniformly in
\begin{equation}\label{eq:ep-epp-in}
	\Babs{n - \frac{2L N^{2/3}}{\sfE_\beta(\theta)}}
	\leq N^{1/3 + \epsilon} - N^{1/3 + \epsilon'}.
\end{equation}
On the other hand, by a slight modification of~\eqref{eq:step-bound},
\begin{equation}\label{eq:out-bound}
	\lim_{N\to\infty}
	\frac{1}{N^{1/3}\sqrt{\chi_\beta}}
	\sum_{\sfu,\sfv} F_N(\sfu) \phi_{-L,N}(\sfu)
	\frac{1}{N^{1/3}\sqrt{\chi_\beta}}
	\calG_{\beta,\lambda/N,+}^{\sfu,\sfv;n} G_N(\sfv) \phi_{L,N}(\sfv)
	= 0 ,
\end{equation}
uniformly in
\begin{equation}\label{eq:ep-epp-out}
	\Babs{n - \frac{2L N^{2/3}}{\sfE_\beta(\theta)}}
	\geq N^{1/3 + \epsilon} + N^{1/3 + \epsilon'} .
\end{equation}
Therefore, the denominator in~\eqref{eq:Phat-fg} satisfies
\begin{multline}
	\label{eq:denom-lim}
	\lim_{N\to\infty}
	\frac{1}{N^{1/3 + \epsilon}}
	\frac{1}{N^{1/3}\sqrt{\chi_\beta}}
	\sum_{\sfu,\sfv} F_N(\sfu) \phi_{-L,N}(\sfu)
	\frac{1}{N^{1/3}\sqrt{\chi_\beta}}
	\sum_n \calG_{\beta,\lambda/N,+}^{\sfu,\sfv;n} G_N(\sfv) \phi_{L,N}(\sfv) \\
	=
	\int_{\bbR_+}\int_{\bbR_+} f(r) \mathrm{e}^{2L\calL_\beta}(r,y) g(y) \dd r\dd y .
\end{multline}
Let $-T \leq \tau_1 < \tau_2 < \dots < \tau_k\leq T$.
Let $h_1, \dots, h_k$ be smooth test functions with compact support in $(c,C)$.
Literally repeating the above argument for the numerator in~\eqref{eq:Phat-fg} of the form
\[
\sum_{\sfu,\sfv} F_N(\sfu)
\phi_{-L,N}(\sfu)
\sum_n \calG_{\beta,\lambda/N,+}^{\sfu,\sfv;n} \biggl( \prod_{\ell=1}^k h_\ell\bigl(\frj_N(\tau_\ell)\bigr) \biggr)
G_N(\sfv) \phi_{L,N}(\sfv) ,
\]
(notice that the scaling operator is implicitly present in \(\frj_N\)) and setting $\sfK_\tau \defby \mathrm{e}^{\tau\calL_\beta}$, we obtain
\begin{multline}
	\label{eq:fin-dist}
	\lim_{N\to\infty}
	\hPbetapl{f,g}\biggl( \prod_{\ell=1}^k h_\ell\bigl(\frj_N(\tau_\ell)\bigr) \biggr) = \\
	\frac{
		\int_{\bbR_+}\!\!\!\cdots \!\int_{\bbR_+} f(r)
		\sfK_{\tau_1+L}(r,y_1) h_1(y_1) \sfK_{\tau_2-\tau_1}(y_1,y_2) h_2(y_2)
		\cdots \sfK_{L-\tau_k}(y_k,y) g(y) \dd r \prod_\ell\dd y_\ell \dd y
	}
	{
		\int_{\bbR_+}\int_{\bbR_+} f(r) \sfK_{2L}(r,y) g(y) \dd r\dd y
	} .
\end{multline}
At this point, we can take the limit $L\to\infty$ and proceed as in~\cite[(3.37)]{Ioffe-Velenik-Wachtel-18}.
\section{Acknowledgments}

The authors are grateful to Shirshendu Ganguly and Reza Gheissari for sending them their preprint~\cite{GG20}.
They also thank the referees for their careful reading and suggestions that have improved the presentation of this work.
The research of D. Ioffe was partially supported by	Israeli Science Foundation grant 765/18.
S. Ott thanks the university Roma Tre for its hospitality and is supported by the Swiss NSF through an early PostDoc.Mobility Grant.
The research of S. Shlosman was partially supported by the Russian Science Foundation (project No.  20-41-09009).
The research of Y. Velenik was partially supported by the Swiss NSF through the NCCR SwissMAP.

\appendix
\renewcommand{\theequation}{\thesection.\arabic{equation}}
\renewcommand{\thesubsection}{\thesection.\arabic{subsection}}
\renewcommand{\thesubsubsection}{\thesection.\arabic{subsection}.\arabic{subsubsection}}

\section{Renewal structure via percolation estimates}
\label{app:renewal_perco}

\subsection{The underlying percolation picture}

Consider an irreducible  decomposition
\begin{equation}\label{eq:kappa-konc}
	\eta = \eta _1 \circ \eta_2 \circ \dots \circ \eta _m ,
\end{equation}
where, as usual, \(\eta_1\in\SetRootMarkBackCont\), \(\eta_m\in\SetRootMarkForwCont\) and, for \(1<i<m\), \(\eta_i\in\SetRootDiaCont\).
Set \(I_\eta \defby \{1,\dots,m\}\) and let  \(\INts(\eta)\) be the set of connected sub-intervals \(I \defby \{\ell,\ell+1,\dots,r\} \subseteq I_\eta\) (\(\ell<r\)).

Consider now Bernoulli percolation \(\frn\) on \(\Theta_\eta \defby \{0,1\}^{\INts(\eta)} =
\{0,1\}^{\binom{m}{2}}\), where an interval \(I\subset \INts(\eta)\) is independently open with a certain probability \(\calP_{N,\beta}(I\given\eta)\) specified below in~\eqref{eq:P-kappa}.
Let us use \(\calP_{N,\beta}(\cdot\given\eta)\) for the corresponding product measure on \(\Theta_\eta\).

Any realization \(\frn \in \Theta_\eta\) of open sub-intervals induces a splitting of
\(\{1,\dots,m\}\) into maximal connected components in the following way:
\begin{enumerate}
	\item If \(I = \{\ell,\ell+1,\dots,r\} \in \frn\), then  \(I\) is connected.
	\item If \(I_1\) and \(I_2\) are connected and such that \(I_1\cap I_2\neq\emptyset\), then \(I_1\cup I_2\) is connected.
\end{enumerate}
Let us denote
\begin{equation}\label{eq:Cl-kappa}
	\underline{\sf Cl}(\frn) \defby \{J_1,\dots,J_{\ell(\frn)}\}
\end{equation}
the collection of all maximal connected components of \(\{1,\dots,m\}\) as induced by \(\frn\).

\begin{figure}[ht]
	\centering
	\includegraphics[width=6cm]{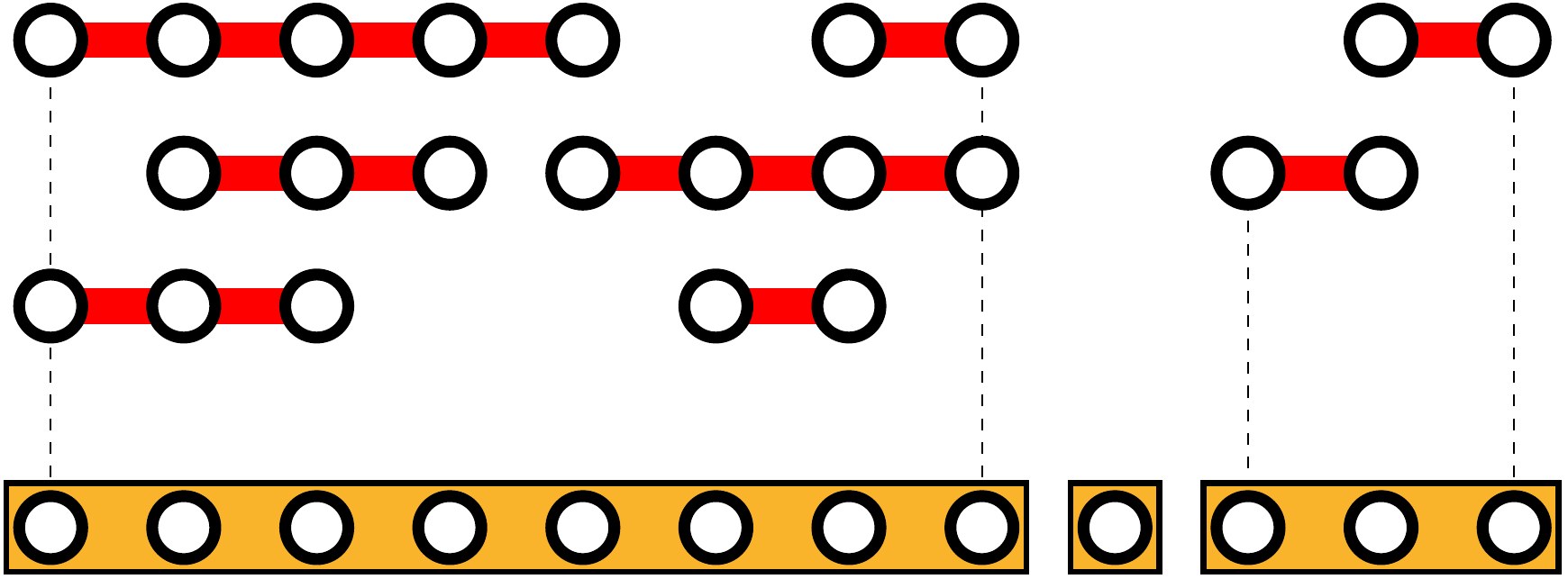}
	\caption{The set of open sub-intervals of a configuration \(\frn\) (top 3 lines; the vertical position is irrelevant and only introduced for readability) and the corresponding set \(\underline{\sf Cl}(\frn)\) of maximal connected components (bottom line).}
	\label{fig:intervals}
\end{figure}

Define the event \(\Theta^{\sf cd}_\eta \subset \Theta_\eta\) by
\begin{equation}\label{eq:frc-kappa}
	\Theta^{\sf cd}_\eta \defby \{I_\eta \text{ is connected}\} = \setof{\frn}{\underline{\sf Cl}(\frn) = \{I_\eta\}} .
\end{equation}
The quantity
\begin{equation}\label{eq:P-kappa-cd}
	\calP_{N,\beta}(\eta) \defby \calP_{N,\beta}(\frn\in\Theta^{\sf cd}_\eta \given \eta)
	=
	\sum_{\frn\in\Theta^{\sf cd}_\eta} \prod_{I\in\frn} \calP_{N,\beta} (I \given \eta) \prod_{I\not\in\frn} \bigl( 1 - \calP_{N,\beta}(I \given \eta) \bigr)
\end{equation}
is, thereby, well defined.

We are ready to define the weights \(\sfp_{N,\beta}\) appearing in~\eqref{eq:WN-beta}:
\begin{definition}\label{def:pnbeta}
	For any \(\eta \subset \bbH_+^*\) such that \(\eta\in\SetRootMarkBackCont \cup \SetRootMarkForwCont\),
	\begin{equation}\label{eq:pnbeta}
		\sfp_{N,\beta}(\eta) \defby \mathrm{e}^{\taub \theta(\eta)} q_{N,\beta}(\eta) \calP_{N,\beta}(\eta) .
	\end{equation}
\end{definition}

\subsection{Conditional weights and expansions}

The proof of Theorem~\ref{thm:fact} is a kind of high-temperature expansion.
Consider \(\eta \) in~\eqref{eq:kappa-konc}.
Let us rewrite the weight \(q_{N,\beta}(\eta)\) as follows:
\[
	q_{N,\beta}(\eta)
	=
	q_{N,\beta}(\eta_1 \given \eta_2 \circ \dots \circ \eta_m) \,
	q_{N,\beta}(\eta_2 \circ \dots \circ \eta_m) .
\]
Next, we factor the conditional weight above as
\begin{multline}
	q_{N,\beta}(\eta_1 \given \eta_2 \circ \dots \circ \eta_m) \\
	=
	q_{N,\beta}(\eta_1) \frac{q_{N,\beta}(\eta_1 \given \eta_2)}{q_{N,\beta}(\eta_1)}
	\frac{q_{N,\beta}(\eta_1 \given \eta_2 \circ \eta_3)}{q_{N,\beta}(\eta_1 \given \eta_2)} \dots
	\frac{q_{N,\beta}(\eta_1 \given \eta_2 \circ \dots \circ \eta_m)}{q_{N,\beta}(\eta_1 \given \eta_2 \circ \dots \circ \eta_{m-1})} .
	\label{eq:cw-1}
\end{multline}
Proceeding with the very same telescopic factorization of \(q_{N,\beta}(\eta_2 \circ \dots \circ \eta_m)\), we eventually arrive to the following construction:
For \(r > \ell+1\) and \(I = \{\ell,\ell+1,\dots,r\}\), set
\begin{equation}\label{eq:rho-N-beta}
	\rho_{N,\beta} (I \given \eta)
	\defby
	\frac
	{q_{N,\beta}(\eta_\ell \given \eta_{\ell+1} \circ \dots \circ \eta_{r})}
	{q_{N,\beta}(\eta_\ell \given \eta_{\ell+1} \circ \dots \circ \eta_{r-1})} - 1 .
\end{equation}
For two point intervals \(I = \{\ell,\ell+1\}\), set
\begin{equation}\label{eq:rho-N-beta-2}
	\rho_{N,\beta}(I \given \eta)
	\defby
	\frac{q_{N,\beta}(\eta_\ell \given \eta_{\ell+1})}{q_{N,\beta}(\eta_\ell)} - 1 .
\end{equation}
Then,
\begin{equation}\label{eq:IRD2-weights}
	q_{N,\beta}(\eta) = q_{N,\beta}(\eta_1) \cdots q_{N,\beta}(\eta_m)\!\!
	\prod_{I\in\INts(\eta)} \bigl( 1 + \rho_{N,\beta}(I \given \eta) \bigr) .
\end{equation}
For any interval \(I = \{\ell,\dots,r\} \in \INts(\eta)\), denote by \(\abs{I} \defby 1+(r-\ell)\) the number of points in \(I\).
Note that the weights \(\rho_{N,\beta}\) are \emph{non-negative} by the domain-monotonicity property~\ref{property:MonotVolWeight} and, in view of~\eqref{eq:P5} and~\eqref{eq:cbeta-nubeta}, satisfy
\begin{equation}\label{eq:rho-decay}
	\rho_{N,\beta} (I \given \eta)
	\leq
	c_\beta \mathrm{e}^{- g_\beta (\abs{I} - 2)},
\end{equation}
with
\begin{equation}\label{eq:cbeta-nubeta-w}
	\lim_{\beta\to\infty} c_\beta = 0
	\quad\text{ and }\quad
	\lim_{\beta\to\infty} g_\beta = \infty .
\end{equation}

\subsubsection{Probabilities \(\calP_{N,\beta}(\cdot \given \eta)\)}
\label{subsub:ubPN}

Let us define the probabilities \(\calP_{N,\beta}(I \given \eta)\) appearing in~\eqref{eq:P-kappa-cd}:
\begin{equation}\label{eq:P-kappa}
	\calP_{N,\beta} (I \given \eta) \defby \frac{\rho_{N,\beta}(I \given \eta)}{1 + \rho_{N,\beta}(I \given \eta)} \stackrel{\eqref{eq:rho-decay}}{\leq} c_\beta \mathrm{e}^{- g_\beta (\abs{I} - 2)}.
\end{equation}

\subsubsection{Expansion of \(\prod_{I\in \INts(\eta)} \bigl( 1 + \rho_{N,\beta}(I \given \eta) \bigr)\)}

Let us expand the product in the right-hand side of~\eqref{eq:IRD2-weights}:
\begin{equation}\label{eq:IRD2-weights-e}
	q_{N,\beta}(\eta) = q_{N,\beta}(\eta_1) \cdots q_{N,\beta}(\eta_m)
	\sum_{\frn\in\Omega_\eta} \prod_{I:\,\frn(I) = 1} \rho_{N,\beta}(I \given \eta) .
\end{equation}
Given \(\frn\in\Omega_\eta\), one can encode the corresponding decomposition into maximal connected sub-intervals \(\underline{\sf Cl} = \underline{\sf Cl}(\frn)\) in~\eqref{eq:Cl-kappa} as
\begin{equation}\label{eq:kappa-coarse}
	\eta = \omega_1 \circ \dots \circ \omega_\ell ,
\end{equation}
where \(\omega_k \defby \omega_{J_k}\) is the concatenation of consecutive paths \(\eta_i\) with indices \(i\) in \(J_k\subseteq I_\eta\).
Clearly, the paths \(\omega_k\) belong to \(\SetRootDiaCont\), but, in general, are not irreducible.

Rearranging the sum in the right-hand side of~\eqref{eq:IRD2-weights-e} according to the values of \(\underline{\sf Cl} (\frn)\) or, equivalently, according to the values in the right-hand side in~\eqref{eq:kappa-coarse} and using the reduced notation \(\Theta_J \defby \Theta_{\omega_J}\),
\begin{align}
	\sum_{\frn\in\Theta_\eta} \prod_{I:\,\frn(I) = 1} \rho_{N,\beta}(I \given \eta)
	&=
	\sum_{\underline{\sf Cl}} \prod_{J \in \underline{\sf Cl}}
	\Bigl(
		\sum_{\frn\in\Theta_J^{\rm cd}} \prod_{I:\,\frn(I) = 1} \rho_{N,\beta} (I \given \omega_J )
	\Bigr) \nonumber \\
	&=
	\sum_{\underline{\sf Cl}} \prod_{J \in \underline{\sf Cl}} \calP_{N,\beta}(\omega _J)
	\prod_{I\in\INts(\omega_J)} \bigl( 1 + \rho_{N,\beta}(I \given \omega_J) \bigr) .
	\label{eq:rho-sum}
\end{align}
Recall Definition~\ref{def:pnbeta}.
In view of~\eqref{eq:IRD2-weights} and the additivity of displacement, \(\theta(\eta) = \sum_k \theta(\omega_k)\), \eqref{eq:rho-sum} implies that
\begin{equation}\label{eq:qbeta-form}
	q_{N,\beta}(\eta) \mathrm{e}^{\taub \theta(\eta)}
	=
	\sum_{\omega_1 \circ \dots \circ\, \omega_\ell = \eta} \prod_j \sfp_{N,\beta}(\omega_j) ,
\end{equation}
as claimed in~\eqref{eq:fact} (Property~\ref{property:W1}) of Theorem~\ref{thm:fact}.

\subsection{Properties of \(\lbr \sfp_{N, \beta}\rbr\)}

We still have to justify~\eqref{eq:W-bound-cp} (Property~\ref{property:W2}) and~\eqref{eq:pnb-pb} (Property~\ref{property:W3}) of Theorem~\ref{thm:fact}.

\subsubsection{Property~\ref{property:W2}}

In terms of the representation~\eqref{eq:pnbeta}, Property~\ref{property:W2} reads as follows:
Given an irreducible decomposition \(\eta = \kappa_1 \circ \dots \circ \kappa_n\), we need to check that
\begin{equation}\label{eq:W2-bound}
	\calP_{N,\beta}(\eta) \leq \mathrm{e}^{- \nu_\beta (n-1)} .
\end{equation}
We use~\eqref{eq:rho-decay} and~\eqref{eq:cbeta-nubeta-w} as an input.
The proof is performed in two steps:

\noindent
\step{1}
There is a very easy proof of~\eqref{eq:W2-bound} when \(\beta\) is large enough:

For \(a\in\{1,\dots,n-1\}\), one can define random variables \(\xi_a\) on \(\Theta_\eta\) in the following way:
Given \(\frn\in\Theta_\eta\), set
\begin{equation}\label{eq:xi-a}
	\xi_a = \xi_a(\frn) \defby \max\bsetof{\ell > 0}{\frn\bigl(\{a,\dots,a+\ell\}\bigr) = 1}.
\end{equation}
If there is no such \(\ell\), we set \(\xi_a \defby 0\).
Under \(\calP_{N,\beta}(\cdot \given \eta)\), the random variables \(\{\xi_a\}\) are independent and, by virtue
of~\eqref{eq:P-kappa}, satisfy the following upper bound:
For \(r\in\bbN\),
\begin{equation}\label{eq:ub-geqr}
	\calP_{N,\beta}(\xi_a \geq r \given \eta)
	\leq
	\sum_{\ell = r}^\infty \calP_{N,\beta} \bigl(\frn\bigl(\{a,\dots,a+\ell\}\bigr) = 1 \bgiven \eta \bigr)
	\leq
	c_\beta \frac{\mathrm{e}^{-g_\beta (r-1)}}{1-{\rm e }^{-g_\beta}} .
\end{equation}
Therefore, the expectation
\begin{equation}\label{eq:Exp-Nb}
\calE_{N,\beta}(\xi_a \given \eta) \leq \frac{c_\beta}{(1-\mathrm{e}^{-g_\beta})^2} .
\end{equation}
The upper bound~\eqref{eq:Exp-Nb} holds uniformly in \(\eta\) and \(a\).

By~\eqref{eq:cbeta-nubeta-w}, the expression on the right-hand side of~\eqref{eq:Exp-Nb} tends to zero as \(\beta\to\infty\). In particular, \(\calE_{N,\beta}(\xi_a \given \eta ) < 1/2\) if \(\beta\) is large enough.
Clearly, the occurrence of the event \(\{\frn\in\Theta_\eta^{\rm cd}\}\) implies that \(\sum_{a=1}^{n-1}\xi_a \geq n\).
Thus, by the usual Cramér's upper bound, there exists \(\nu_\beta > 0\) such that, for any \(n\geq 2\) and for any \(\eta\) as in~\eqref{eq:kappa-konc},
\begin{equation}\label{eq:Cram1}
	\calP_{N,\beta}(\eta)
	=
	\calP_{N,\beta}(\frn\in\Theta_\eta^{\rm cd} \given \eta)
	\leq
	\calP_{N,\beta} \Bigl( \sum\xi_a \geq n \Bgiven \eta \Bigr)
	\leq
	c_\beta \mathrm{e}^{-\nu_\beta (n-1)} .
\end{equation}
\step{2}
Recall~\eqref{eq:Cl-kappa}.
Notice that the above argument actually implies a lower bound on the number \(\ell(\frn)\) of disjoint components of \(\underline{{\sf Cl}}(\frn)\).
Indeed, \(\{\ell(\frn) \leq k\}\) implies that \(\sum_{a=1}^{n-1} \xi_a \geq n-k\).
Consequently, one can refine~\eqref{eq:Cram1} as follows:
For all \(\beta\) sufficiently large, there exists \(\nu_\beta > 0\) such that
\begin{equation}\label{eq:Cramer2}
	\calP_{N,\beta} \Bigl( \ell(\frn) \leq \frac{n}{4} \Bgiven \eta \Bigr)
	\leq
	\calP_{N,\beta } \Bigl( \sum_{a}\xi_a \geq \frac{3}{4}n \Bgiven \eta \Bigr)
	\leq
	c_\beta \mathrm{e}^{-\nu_\beta (n-1)} .
\end{equation}
Let us turn to the case of arbitrary fixed \(\beta > \beta_c\).
We cannot anymore assume that \(\calE_{N,\beta} (\xi_a \given \eta) < 1/2\).
This can be dealt with in the following way:
We choose a cutoff value \(r = r_\beta \in \bbN\) and say that a connected interval \(\{a,\dots,a+m\}\) is
\emph{short} if \(m \leq r\) and \emph{long} if \(m > r\).
In this way, we represent \(\Theta_\eta = \Theta_\eta^{\sf short}\times \Theta_\eta^{\rm long}\) and
\(\calP_{N,\beta}(\cdot \given \eta) = \calP_{N,\beta}^{\sf short}(\cdot \given \eta) \otimes
\calP_{N,\beta}^{\sf long}(\cdot \given \eta)\).
Let us redefine random variables \(\xi_a\) in~\eqref{eq:xi-a} as functions on \(\Theta_\eta^{\rm long}\) only,
\begin{equation}\label{eq:xi-a-l}
\xi_a^{\rm long} = \max\bsetof{\ell > 0}{\frn^{\rm long}\bigl(\{a,\dots,a+\ell\}\bigr) = 1}.
\end{equation}
In view of~\eqref{eq:P-kappa}, the expectation \(\calE_{N,\beta}(\xi_a^{\rm long}) < 1/4\), once the cutoff value \(r_\beta\) is chosen large enough. Hence, \eqref{eq:Cramer2} applies for \(\calP_{N,\beta}^{\rm long} (\cdot \given \eta)\). If \(\ell(\frn^{\rm long}) \geq \frac{n}{4}\), then one can still recover the event
\[
\bigl\{ \frn = (\frn^{\rm short},\frn^{\rm long}) \in \Theta_\eta^{\rm cd} \bigr\}
\]
by insisting that short intervals cover all \(\ell(\frn^{\rm long}) - 1\) spacings between successive maximal connected long clusters. But this is already a short-range percolation problem in one dimension, and the corresponding probabilities \(\calP_{N,\beta}^{\sf short}(\cdot \given \eta)\) have tails that decay exponentially in \(n\). \qed

\subsubsection{Property~\ref{property:W3}}
\label{sec:property:W3}

Recall Definition~\ref{def:pnbeta}.
The infinite-volume weights \(\{\sfp_\beta\}\) are defined as follows:
\begin{definition}\label{def:pnbeta-inf}
	For any \(\eta \in \SetRootMarkBackCont \cup \SetRootMarkForwCont\),
	\begin{equation}\label{eq:pnbeta-inf}
		\sfp_{\beta}(\eta) \defby \mathrm{e}^{\taub \theta(\eta)} q_{\beta}(\eta) \calP_{\beta}(\eta) ,
	\end{equation}
	where \(\calP_\beta\) is defined as in Subsection~\ref{subsub:ubPN}, but using the infinite-volume weights \(\rho_\beta\) instead of \(\rho_{N,\beta}\). In their turn, the weights \(\rho_\beta\) are defined as in~\eqref{eq:rho-N-beta}--\,\eqref{eq:rho-N-beta-2}, but using the infinite-volume path weights \(q_{\beta}\) instead of \(q_{N,\beta}\).
\end{definition}
Therefore, \eqref{eq:pnb-pb} is a consequence of~\eqref{eq:P5}. \qed

It remains to show~\eqref{eq:pbet-weights}. The latter, however, is a standard consequence of the multi-dimensional renewal theory under exponential tails~\cite{ioffe2015LN}:
First of all, set \(\SetRootDiaCont(0) = \bigcup_\sfu \SetRootDiaCont(0,\sfu)\) and note that,
for any \(\eta\in\SetRootDiaCont(0)\) which admits an irreducible decomposition~\eqref{eq:IRD1-gen}, the infinite-volume analogue of~\eqref{eq:fact} and~\eqref{eq:qbeta-form} holds:~
\begin{equation}\label{eq:fact-inf}
	q_{\beta}(\eta) \mathrm{e}^{\taub \theta(\eta)} = \sum_{\uomega\sim\eta} \prod_i \sfp_\beta(\omega_i) .
\end{equation}
Consider now the series
\begin{equation}\label{eq:ser-h}
	\sfh\mapsto \sum_{\gamma\in\SetRootDiaCont(0)}\mathrm{e}^{\taub \theta(\gamma) + \sfh\cdot\sfX(\gamma)} q_\beta(\gamma) .
\end{equation}
For small \(\sfh\), this series converges if and only if \(\taub \sfe_1 +\sfh\) belongs to the interior of the Wulff shape \({\mathbf K}_\beta\). In terms of the weights \(\sfp_\beta\), the right-hand side of~\eqref{eq:ser-h} can be expressed as
\begin{equation}\label{eq:ser-h-p}
	\sum_{k=1}^\infty \Bigl( \sum_{\omega\in\SetRootDiaCont(0)} \mathrm{e}^{\sfh\cdot\sfX(\omega)} \sfp_\beta(\omega) \Bigr)^k .
\end{equation}
By the exponential decay of the weights \(\sfp_{\beta}\), the series
\[
	\sfh \mapsto \sum_{\omega\in\SetRootDiaCont(0)} \mathrm{e}^{\sfh\cdot\sfX(\omega)} \sfp_\beta(\omega)
\]
converges in a small neighborhood of the origin.
Hence the sum should be exactly one at \(\sfh = 0\). \qed

\section{Curvature of the Wulff shape \(\mathbf{K}_\beta\) at \(\taub\sfe_1\)}
\label{sec:Curvature}

We briefly explain why \(\chi_\beta\) in~\eqref{eq:chi-fun} can be identified with the curvature of the boundary of the Wulff shape \(\mathbf{K}_\beta\) at \(\taub\sfe_1\).
The starting point is the characterization of \(\partial\mathbf{K}_\beta\) in a neighborhood of \(\taub\sfe_1\) by
\[
\taub\sfe_1 + \sfh \in \partial\mathbf{K}_\beta
\;\Leftrightarrow\;
\sfE_\beta \bigl[ e^{\sfh\cdot\sfX(\gamma)} \bigr] = 1 ,
\]
as explained in Section~\ref{sec:property:W3}. Remember our notation \(\sfX(\eta) = (\theta(\eta),\zeta(\eta))\). In the same coordinate representation, we parameterize vectors \(\sfh\) such that \(\taub\sfe_1 + \sfh \in \partial\mathbf{K}_\beta\) by \(\sfh = -(\phi(t),t)\) for \(t\) in a neighborhood of \(0\). Observe that \(\phi(0)=\phi'(0)=0\). Expanding to second order in \(t\) thus yields the identity
\[
	\phi''(0) = \frac{\sfE_\beta[\zeta^2]}{\sfE_\beta[\theta]},
\]
since \(\sfE_\beta[\zeta] = 0\). Finally, since \(\phi'(0)=0\), the curvature of \(\partial\mathbf{K}_\beta\) at \(\taub\sfe_1\) indeed reduces to \(\phi''(0)\).

\section{Ferrari-Spohn diffusions}
\label{app:FS}

In this appendix, we briefly recall the general definition of Ferrari--Spohn diffusions on \(\bbR_+\). We refer to~\cite{Ioffe+Shlosman+Velenik-15} and~\cite{Ioffe-Velenik-Wachtel-18} for more details.

\smallskip
Fix a real number \(\sigma>0\) and a non-negative function \(q\in\sfC^2(\bbR_+)\)
such that \(\lim_{r\to\infty} q (r) = \infty\). Consider the singular  Sturm--Liouville operator
\[
	\sfL_{\sigma,q} \defby \frac{\sigma^2}{2}\frac{\dd^2}{\dd r^2} - q (r) ,
\]
on \(\bbR_+\) with Dirichlet boundary condition at \(0\).
This operator possesses a complete orthonormal family \(\{\varphi_i, i\in\bbN_{\geq 0}\}\) of simple eigenfunctions in \(\bbL_2(\bbR_+)\) with eigenvalues
\[
	0 >- \eig_0 > -\eig_1 > -\eig_2 > \cdots
\]
satisfying \(\lim_{i\to\infty} \eig_i = \infty\).
The eigenfunctions \(\varphi_i\), \(i\in\bbN_{\geq 0}\), are smooth and \(\varphi_i\) has exactly \(i\) zeros in the interval \((0, \infty)\).

\medskip
The Ferrari--Spohn diffusion associated to \(\sigma\) and \(q\) is the
diffusion on $(0, \infty)$ with generator
\[
	\sfG_{\sigma, q}\psi
	\defby
	\frac{1}{\varphi_0 } (\sfL_{\sigma,q} +\eig_0) (\psi\varphi_0) =
	\frac{\sigma^2}{2}\frac{\dd^2\psi }{\dd r^2} + \sigma^2
	\frac{\varphi_0^\prime }{\varphi_0}\frac{\dd\psi }{\dd r}.
\]
This diffusion is ergodic and reversible with respect to the measure
\(\dd\mu_0(r)=\varphi_0^2(r)\dd r\).

\medskip
In the present paper, the relevant choice for the function \(q\) is \(q(r)= cr\) for some suitably chosen constant \(c\) depending on the parameters of the Ising model.
In this case, since the Airy function \(\mathsf{Ai}\) satisfies \(\frac{\dd^2}{\dd r^2}\mathsf{Ai}(r) = r\mathsf{Ai}(r)\), a simple computation shows that
\[
	\varphi_0
	=
	\mathsf{Ai}(C r-\omega_1) \quad{\rm and}\quad \eig_0
	=
	\frac{c\omega_1}{C},
\]
where \(-\omega_1\) is the smallest zero (in absolute value) of \(\mathsf{Ai}\) (see Fig.~\ref{fig:Airy}) and \(C = (2c/\sigma^2)^{1/3}\).
\begin{figure}
	\includegraphics[width=10cm]{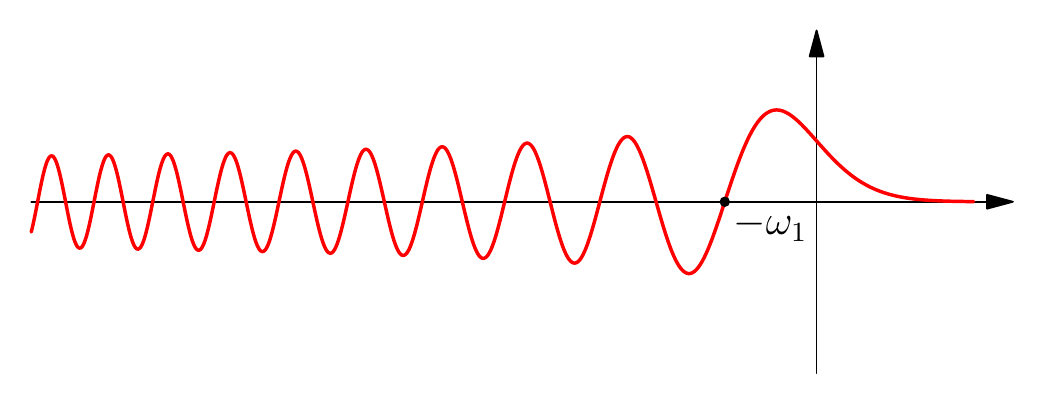}
	\caption{The Airy function \(\mathsf{Ai}\) and its first zero \(-\omega_1\).}
	\label{fig:Airy}
\end{figure}

\bibliographystyle{imsart-number}
\bibliography{ISV18}

\end{document}